\documentclass{amsart}

\usepackage[utf8]{inputenc}
\usepackage{amsmath,amsthm}
\usepackage{comment}
\usepackage{amssymb}
\usepackage{mathrsfs}
\usepackage{stmaryrd}
\usepackage{MnSymbol}
\usepackage{hyperref}
\usepackage{color}
\usepackage[left=3cm,right=3cm,top=3cm,bottom=4.5cm]{geometry}
\usepackage{pdfsync}
\usepackage{tikz-cd}
\usepackage{xypic}
\usepackage{url}
\usepackage{enumitem}
\usepackage{multicol}
\usepackage{listofitems}
\hypersetup{linktoc=all}
\usepackage[capitalise]{cleveref}

\newcommand{\brandon}[1]{\textcolor{blue}{#1}}

\newcommand{\defterm}[1]{\emph{#1}}

\theoremstyle{plain}
	\newtheorem{thm}{Theorem}
	\numberwithin{thm}{section}
	\newtheorem*{thm*}{Theorem}
	\newtheorem{cor}[thm]{Corollary}
	\newtheorem*{cor*}{Corollary}
	\newtheorem{prop}[thm]{Proposition}
	\newtheorem*{prop*}{Proposition}
	\newtheorem{lem}[thm]{Lemma}
	\crefname{lem}{Lemma}{Lemmas}
	\newtheorem*{lem*}{Lemma}
	
	\newtheorem*{ex*}{Exercise}
	
	\newtheorem*{claim*}{Claim}
	
	\newtheorem*{question*}{Question}
	
	\newtheorem*{fact*}{Fact}
	
	\crefname{assume}{Assumption}{Assumptions}
\theoremstyle{definition}
	\newtheorem{Def}[thm]{Definition}
	\newtheorem*{Def*}{Definition}
	
	\newtheorem*{obs*}{Observation}
	\newtheorem{rmk}[thm]{Remark}
	\newtheorem*{rmk*}{Remark}
	
	\newtheorem{soln*}{Solution}
	
	\newtheorem*{note*}{Note}
	
	\newtheorem*{eg*}{Example}	
	\newtheorem{egs}[thm]{Examples}
	\newtheorem*{egs*}{Examples}
	
	\newtheorem*{construction*}{Construction}
	
	\newtheorem*{warning*}{Warning}
	
	\newtheorem*{conj*}{Conjecture}
	
	\newtheorem*{notation*}{Notation}
	
\crefname{prop}{Proposition}{Propositions}
\crefname{thm}{Theorem}{Theorems}
	
\newcommand{\op}{\mathrm{op}}
\newcommand{\id}{\mathrm{id}}

\newcommand{\Set}{\mathsf{Set}}

\newcommand{\sSet}{\mathsf{sSet}}

\newcommand{\calI}{\mathcal{I}}
\newcommand{\calJ}{\mathcal{J}}

\newcommand{\scrK}{\mathscr{K}}

\newcommand{\cell}{\mathsf{cell}}

\newcommand{\bd}{\partial}

\newcommand{\pushoutcorner}[1][dr]{\save*!/#1+1.2pc/#1:(1,-1)@^{|-}\restore}

\newcommand{\mSet}{\mathsf{sSet}^+}

\newcommand{\PreComp}{\mathsf{PreComp}}

\newcommand{\precomp}{\mathrm{pre}}
\newcommand{\join}{\star}
\renewcommand{\Im}{\operatorname{Im}}
\newcommand{\co}{\mathrm{co}}
\newcommand{\coop}{\mathrm{co-op}}

\newcommand{\sDelta}[1]{\Delta^{#1}} 
\newcommand{\mDelta}[1]{\widetilde{\Delta}^{#1}} 
\newcommand{\horn}[2]{\Lambda_{#2}^{#1}} 

\newcommand{\frontface}[2]{\upmodels_1^{#1,#2}}
\newcommand{\backface}[2]{\upmodels_2^{#1,#2}}

\newcommand{\face}[2]{\partial_{#1,#2}} 
\newcommand{\adelta}[2]{\Delta^{#1}_{#2}}
\newcommand{\adeltap}[2]{{\Delta^{#1}_{#2}}'}
\newcommand{\sface}[1]{\partial_{#1}} 

\newcommand{\incl}{\hookrightarrow}
\newcommand{\pushout}{\arrow [dr, phantom, "\ulcorner" very near end]}
\newcommand{\mQ}[1]{\widetilde{Q}^{#1}}
\newcommand{\msSet}{\mathsf{sSet}^+}
\newcommand{\mdelta}[1]{\widetilde{\Delta}^{#1}} 
\newcommand{\triv}[1]{\tau_{#1}}
\newcommand{\eq}{\Delta^3_{\mathrm{eq}}}
\colorlet{shaded}{gray!20!white}

\newcommand{\Cube}{\square}
\newcommand{\cSet}{\mathsf{cSet}}
\newcommand{\mcSet}{\mathsf{cSet}^+}

\newcommand{\mcube}[1]{\widetilde{\Cube}^{#1}}
\newcommand{\cube}[1]{\Cube^{#1}}

\newcommand{\obox}[3]{\sqcap^{#1}_{#2, #3}}

\newcommand{\hxi}{\widehat{\Xi}}
\newcommand{\fcap}[2]{T(\sqcap^{#1}_{#2,0})^\flat}
\newcommand{\mC}{\widetilde{C}}
\newcommand\myitem[1][]{\item[#1]\def\@currentlabel{#1}}
\newcommand{\stbox}{\overline{\Box}}

\newcommand{\stc}{\overline{C}}
\newcommand{\stb}{\overline{B}}
\newcommand{\catC}{\mathcal{C}}
\newcommand{\catD}{\mathcal{D}}
\newcommand{\uhom}{\underline{\hom}}
\newcommand{\sk}{\mathrm{sk}}

\usepackage{subfiles} 

\title{Equivalence of cubical and simplicial approaches to $(\infty,n)$-categories}
\author[B.~Doherty, K.~Kapulkin, Y.~Maehara]{Brandon Doherty \and Krzysztof Kapulkin \and Yuki Maehara}
\date{\today}

\begin{document}

\maketitle

\begin{abstract}
  We prove that the marked triangulation functor from the category of marked cubical sets equipped with a model structure for ($n$-trivial, saturated) comical sets to the category of marked simplicial set equipped with a model structure for ($n$-trivial, saturated) complicial sets is a Quillen equivalence.
  Our proof is based on the theory of cones, previously developed by the first two authors together with Lindsey and Sattler.
\end{abstract}

\section*{Introduction}\label{sec:introduction}
In this paper, we prove that two models of $(\infty,n)$-categories (including the case $n = \infty$) are equivalent.
The first of these models is \emph{complicial sets}, developed by Verity \cite{verity:weak-complicial-1,verity:weak-complicial-2} and Riehl \cite[App.~D]{riehl-verity:more-elements}, using marked simplicial sets, i.e., simplicial sets in which certain simplices are formally designated as equivalences, witnessing composition of lower-dimensional simplices.
The second model, in \emph{comical sets}, introduced in \cite{campion-kapulkin-maehara}, is based on marked cubical sets instead and is intended to provide an alternative to complicial sets, particularly well suited for the treatment of Gray tensor product.

Both of these models arise as fibrant-cofibrant objects in certain model structures on the categories $\msSet$ of marked simplicial sets and $\mcSet$ of marked cubical sets, respectively.
In \cite{campion-kapulkin-maehara}, a triangulation functor $T \colon \mcSet \to \msSet$ is constructed as an extension of the usual triangulation functor $T \colon \cSet \to \sSet$ \cite{CisinskiAsterisque} to the marked setting and it is conjectured that this functor is a Quillen equivalence.

We prove a slight modification of this conjecture.
In our proof, we change the definition of a comical set from the one given in \cite{campion-kapulkin-maehara} by adjusting some of the generating anodyne maps,  to require that the markings on the standard $(i,\varepsilon)$-cube $\Box^n_{i, \varepsilon}$ be more restrictive.
As in \cite{campion-kapulkin-maehara}, we construct a model structure whose fibrant-cofibrant objects are the comical sets.

Thus our main theorem can be stated as follows.

\begin{thm*}[cf.~\cref{comical-model-structure,T-equivalence}]
The category $\mcSet$ of marked cubical sets carries a model structure in which the cofibrations are monomorphisms and the fibrant objects are the ($n$-trivial, saturated) comical sets.

The marked triangulation functor $T \colon \mcSet \to \msSet$ is a left Quillen equivalence between this model structure and the model structure for (resp.~$n$-trivial, saturated) complicial sets.
\end{thm*}

Our proof employs the theory of cones in cubical sets, developed in \cite{doherty-kapulkin-lindsey-sattler}, by generalizing it to the marked setting, i.e., from $(\infty,0)$- and $(\infty,1)$-categories to arbitrary $(\infty,n)$-categories.
In brief, we consider a left Quillen functor $Q \colon \msSet \to \mcSet$ (see \cite{kapulkin-lindsey-wong,kapulkin-voevodsky} for its treatment in the unmarked setting) whose value at $\Delta^n$ is given by applying the cone construction to the point $n$ times.
We then show the composite $TQ$ is equivalent to the identity functor and further that $Q$ itself is a left Quillen equivalence.

\textbf{Organization of the paper.}
In \cref{sec:background}, we recall the standard background on marked simplicial and cubical sets.
Then in \cref{sec:model}, we introduce our comical sets, modifying the definition of \cite{campion-kapulkin-maehara}.
We show that the marked triangulation functor $T \colon \mcSet \to \msSet$ is a left Quillen functor in \cref{Quillen-functor}.
We then proceed to prove that it is also a Quillen equivalence by first setting up a generalization of the theory of cones to the marked setting and introducing the homotopy inverse $Q$ to triangulation in \cref{sec:cones}, before concluding our proof in \cref{sec:equivalence}.

\textbf{Acknowledgements.}
The first-named author was partially supported by a Canada Graduate Scholarship from the Natural Sciences and Engineering Research Council of Canada (NSERC). 
The second-named author was partially supported by a Discovery Grant from the Natural Sciences and Engineering Research Council of Canada (NSERC). 
The third-named author was partially supported by JSPS KAKENHI Grant Number 21K20329.


\section{Background}\label{sec:background}

\subsection{Model structures and Cisinski--Olschok theory}
Many model structures of interest have as their cofibrations precisely the monomorphisms.
A general construction of such a model structure was given by Cisinski \cite{CisinskiAsterisque} in the case where the underlying category is a presheaf topos, and it was later generalised by Olschok \cite{olschok:thesis} to arbitrary locally presentable categories.
The aim of this section is to review a monoidal variant of this construction.

Given $f\colon A \to A'$ and $g \colon B \to B'$ in a monoidal category $(\scrK,\otimes)$, we write $f \hat \otimes g$ for the \emph{Leibniz $\otimes$-product}
\[
f\hat\otimes g \colon (A'\otimes B) \coprod_{A \otimes B}(A \otimes B') \to A'\otimes B'.
\]

Given a collection $\mathcal{I}$ of maps in a category $\mathscr{K}$, the \emph{cellular closure} of $\mathcal{I}$, denoted $\cell(\mathcal{I})$, is the closure of $\mathcal{I}$ under pushouts along arbitrary maps and transfinite composition.

The following useful fact will be used throughout the proof of \cref{CO-with-monoidal}.

\begin{prop}[{\cite[Prop.~5.12]{riehl-verity:theory-and-practice}}]\label{cell-tensor}
	Let $\scrK$ be a cocomplete category equipped with a bifunctor $\otimes \colon\scrK \times \scrK\to \scrK$ that is cocontinuous in each variable.
	Then for any collections $\mathcal{I},\mathcal{J}$ of morphisms in $\mathscr{K}$, we have $\cell(\calI)\hat\otimes\cell(\calJ) \subset \cell(\calI\hat\otimes\calJ)$.
\end{prop}

\begin{thm} \label{CO-with-monoidal}
	Let $\mathscr{K}$ be a locally presentable category equipped with:
	\begin{itemize}
		\item sets $\mathcal{I}$, $\mathcal{J}$ of monomorphisms;
		\item a biclosed monoidal structure $\otimes$ whose unit is the terminal object $1$; and
		\item a bipointed object $\partial_0,\partial_1 \colon1 \to I$
	\end{itemize}
	such that:
	\begin{itemize}
		\item[(0)] the unique map from the initial object to any object is a monomorphism;
		\item[(1)] $\cell(\mathcal{I})$ is precisely the class of monomorphisms in $\mathscr{K}$;
		\item[(2)] $(\partial_0,\partial_1) \colon1 \amalg 1 \to I$ is a monomorphism;
		\item[(3)] $\partial_0, \partial_1 \in \cell(\mathcal{J})$;
		\item[(4)] $\mathcal{I}\hat\otimes\mathcal{I} \subset \cell(\mathcal{I})$;
		\item[(5)] $\mathcal{I}\hat\otimes\mathcal{J} \subset \cell(\mathcal{J})$; and
		\item[(6)] $\mathcal{J}\hat\otimes\mathcal{I} \subset \cell(\mathcal{J})$.
	\end{itemize}
	Then there exists a (necessarily unique) left proper, combinatorial model structure on $\mathscr{K}$ such that:
	\begin{itemize}
		\item[(i)] the cofibrations are precisely the monomorphisms; and
		\item[(ii)] a map into a fibrant object is a fibration if and only if it has the right lifting property with respect to each member of $\mathcal{J}$.
	\end{itemize}
	Moreover this model structure is monoidal with respect to $\otimes$.
\end{thm}

\begin{proof}
	We will apply \cite[Thm.~2.2.5]{olschok:thesis}.
	The weak factorisation system $(\mathcal{L},\mathcal{R})$, the left class of which is the monomorphisms in our case, is obtained by applying the small object argument to the set $\mathcal{I}$.
	Since it is cofibrantly generated by construction and every object is cofibrant by (0), this factorisation system is cofibrant in the sense of \cite[Def.~2.2.3]{olschok:thesis}.
	Our functorial cylinder is given by
	\[
	\begin{tikzcd}[column sep = large]
		X \amalg X \cong X \otimes (1 \amalg 1)
		\arrow [r, "{X \otimes (\partial_0,\partial_1)}"] &
		X \otimes I
		\arrow [r, "{X \otimes !}"] &
		X \otimes 1 \cong X.
	\end{tikzcd}
	\]
	Note that (1), (2) and (4) imply that $X \otimes (\partial_0,\partial_1)$ is a monomorphism for each $X$ since it can be written as $i_X \hat\otimes(\partial_0,\partial_1)$ where $i_X \colon0 \to X$ is the unique map from the initial object.
	Hence this functorial cylinder is good in the sense of \cite[Def.~1.2.1]{olschok:thesis}.
	The biclosedness of $\otimes$ implies that it moreover satisfies the condition (a) of \cite[Def.~2.1.9]{olschok:thesis}.
	To verify that it also satisfies (b), it suffices to observe that any map from a terminal object (and in particular $\partial_0,\partial_1$) is a monomorphism, so $\mathcal{L}\star\gamma$ and $\mathcal{L} \star \gamma^k$, which are respectively $\cell(\mathcal{I}) \hat \otimes (\bd_0,\bd_1)$ and $\cell(\mathcal{I}) \hat \otimes \bd_k$ in our case, are monomorphisms again by (1), (2) and (4).
	Thus, by \cite[Thm.~2.2.5]{olschok:thesis}, we obtain a left proper, combinatorial model structure on $\scrK$ such that the cofibrations are precisely the monomorphisms.
	
	Let $\Lambda$ denote the closure of
	\[
	\calJ \cup \bigl\{f \hat\otimes\partial_\varepsilon \ | \ f \in \calI, \varepsilon \in \{0,1\}\bigr\}
	\]
	under the operation $(-)\hat\otimes(\partial_0,\partial_1)$.
	Then $\Lambda$ is precisely the class $\Lambda \bigl((-)\otimes I, \mathcal{I},\mathcal{J}\bigr)$ in the sense of \cite[Def.~2.1.6]{olschok:thesis}, so our use of the symbol $\Lambda$ agrees with that introduced in \cite[2.2.7]{olschok:thesis}.
	Note that all maps in $\Lambda$ are monomorphisms by (1), (2) and (4).
	In fact, they are trivial cofibrations in this model structure by \cite[Lem.~2.2.15]{olschok:thesis}.
	Thus, on the one hand, any fibration has the right lifting property with respect to $\Lambda$.
	On the other hand, \cite[Lem.~2.2.20(b)]{olschok:thesis} implies that any map into a fibrant object that has the right lifting property with respect to $\Lambda$ is a fibration.
	Therefore we have:	
	\begin{itemize}
	\item[(ii')] a map into a fibrant object is a fibration if and only if it has the right lifting property with respect to each member of $\Lambda$.
	\end{itemize}
	We can reduce (ii') to (ii).
	Indeed, it follows from (3) and (5) that the above generating set of $\Lambda$ is contained in $\cell(\calJ)$.
	Moreover $\cell(\calJ)$ is closed under the operation $(-)\hat\otimes(\partial_0,\partial_1)$ by (1), (2) and (6), which implies $\Lambda \subset \cell(\calJ)$.
	
	That this model structure is monoidal with respect to $\otimes$ follows from \cite[Prop.~A.4]{maehara:gray} and (4-6).
\end{proof}

When comparing model structures, we will frequently make use of the following basic lemma, a generalization of \cite[Cor.~1.5]{doherty-kapulkin-lindsey-sattler}.

\begin{lem}\label{involution-equiv}
Let $\catC, \catD$ be model categories, and $F : \catC \rightleftarrows \catD : G$ an adjoint equivalence of the underlying categories. If both of the adjunctions $F \dashv G$ and $G \dashv F$ are Quillen, then they are Quillen equivalences.
\end{lem}

\begin{proof}
By assumption, both $F$ and $G$ are left Quillen functors, and we have natural isomorphisms $\id_{\catC} \cong GF$, $FG \cong \id_{\catD}$. Since the left derived functor construction is functorial up to natural isomorphism, it follows that the left derived functors of $F$ and $G$ are inverse equivalences of categories.
\end{proof}

Note that in all cases in which we will apply this result, the equivalences of categories $F$ and $G$ will in fact be isomorphisms.

\subsection{Complicial sets}
In this section, we introduce marked simplicial sets and model structures for ($n$-trivial, saturated) complicial sets.

When working with simplicial sets, we will write the action of simplicial operators on the right.
For example, given an $n$-simplex $x \in X_n$ of a simplicial set $X$, we would write $x \partial_{i}$ for its $i$-face; this is consistent with the usual notation for composition when $x$ is regarded as a simplicial map $x \colon \Delta^n \to X$.

\begin{Def}
	A \defterm{marked simplicial set} is a simplicial set $X$ equipped with a subset of its simplices $eX \subseteq \amalg_{n\geq 1} X_n$, called the \defterm{marked} simplices, containing all degenerate simplices.
	A \defterm{map} of marked simplicial sets $f\colon X \to Y$ is a map of simplicial sets which carries marked simplices to marked simplices. We denote $\mSet$ for the category of marked simplicial sets with maps for morphisms.
\end{Def}

Marked simplicial sets used to be called \emph{stratified simplicial sets} (cf.~e.g., \cite{verity:complicial}), but the name `marked' is more descriptive and has since become more popular.

When illustrating marked simplicial sets, we will indicate marked 1-simplices with a tilde. Marked 2-simplices will be shaded, with tildes in their interiors in cases where this is feasible. In certain cases, we will draw thick arrows in the interiors of 2-simplices pointing from the face $\bd_1$ towards the faces $\bd_0$ and $\bd_2$; when this is done, 2-simplices will be distinguished only by shading, with tildes omitted for clarity.

There is a natural forgetful functor $\mSet \to \sSet$, which has both left and right adjoints.
The left adjoint $X \mapsto X^\flat$ endows a simplicial set $X$ with the \defterm{minimal marking}, marking only the degenerate simplices. The right adjoint $X \mapsto X^\sharp$ endows a simplicial set $X$ with the \defterm{maximal marking}, marking all simplices.

For $n \ge 0$, we abuse the notation and write $\Delta^n$ for the minimally marked simplicial set $(\Delta^n)^\flat$.

\begin{Def}
	We say that $X \in \mSet$ is \defterm{$n$-trivial} if every simplex of dimension $\geq n+1$ is marked.
\end{Def}

We write $\tau_n \colon \mSet \to \mSet$ for the functor that is the identity on the underlying simplicial set and marks all $k$-simplices for $k \geq n+1$. Note that this functor has a right adjoint, which sends a marked simplicial set to its maximal subcomplex in which all simplices above dimension $n$ are marked.

\begin{Def}
	A map $X \to Y$ of marked simplicial sets is:
	\begin{itemize}
		\item \emph{regular} if it creates markings, i.e., for an $n$-simplex $x$ of $X$ we have: $x \in eX_n$ if and only if $f(x) \in eY_n$;
		\item \emph{entire} if the induced map between the underlying simplicial sets is invertible.
	\end{itemize}
\end{Def}

We introduce the following distinguished objects and maps in $\mSet$.

\begin{Def}\leavevmode
	\begin{itemize}
		\item For $n \ge 1$, the \defterm{marked $n$-simplex}, denoted $\mDelta n$, is $\tau_{n-1} (\sDelta n)$, \emph{i.e.}, the $n$-simplex with the non-degenerate $n$-simplex marked and no other non-degenerate simplices marked.
		\item For $n \ge 1$ and $0 \le k \le n$:
		\begin{itemize}
			\item the \defterm{$k$-complicial $n$-simplex}, denoted $\adelta n k$, is the marked simplicial set with underlying simplicial set $\Delta^n$ in which a non-degenerate simplex $\alpha \colon [m] \to [n]$ is marked if and only if the image of $\alpha$ contains $\{k-1,k,k+1\}\cap[n]$;
			\item an $n$-simplex $\Delta^n \to X$ of a marked simplicial set $X$ is \defterm{$k$-complicial} if it factors through $\Delta^n \to \adelta n k$;
			\item the \defterm{$k$-complicial horn}, denoted $\horn n k$, is the regular subcomplex of $\adelta n k$ whose underlying simplicial set is the $k$-th horn;
			\item an $n$-horn $(\horn n k)^\flat \to X$ of a marked simplicial set $X$ is \defterm{$k$-complicial} if it factors through $(\horn n k)^\flat \to \horn n k$; and
			\item $\adeltap n k$ is the marked simplicial set obtained from $\adelta n k$ by further marking $\sface{k-1}$ and $\sface{k+1}$.
		\end{itemize}
	\end{itemize}
\end{Def}

\begin{Def}\leavevmode
	\begin{itemize}
		\item For $n \geq 1$, the \defterm{$n$-marker} is the inclusion $\sDelta n \to \mDelta n$.
		\item For $n \geq 1$, the \emph{$k$-complicial horn inclusion} is the inclusion $\horn n k \incl \adelta n k$.
		\item For $n \geq 2$, the \defterm{elementary $k$-complicial marking extension} is the entire inclusion $\adeltap n k \incl \tau_{n-2}\adelta n k$.
		\item We say a map in $\msSet$ is \defterm{complicial} if it is in
		\[
		\cell\bigl(\{\horn n k \incl \adelta n k \ | \ n \ge 1,~0 \le k \le n\} \cup \{\adeltap n k \incl \tau_{n-2}\adelta n k \ | \ n \ge 2,~0 \le k \le n\}\bigr).
		\]
	\end{itemize}
\end{Def}

\begin{Def}
	We write $\Delta^+$ for the full subcategory of $\msSet$ spanned by $\Delta^n$ for $n \ge 0$ and $\mDelta n$ for $n \ge 1$.
\end{Def}

\begin{prop}[{\cite[Obs.~12]{verity:weak-complicial-1}}]
	The category $\msSet$ is equivalent to the reflective, full subcategory of $[(\Delta^+)^\op,\Set]$ spanned by those functors that send each marker $\Delta^n \to \mdelta{n}$ to a monomorphism. \qed
\end{prop}

The join operation $\join$ is extended from simplicial sets to marked simplicial sets by declaring that a simplex in $X \join Y$ is marked if and only if either its component in $X$ or its component in $Y$ is marked.

\begin{lem}[{\cite[Lem.~39]{verity:weak-complicial-1}}]\label{join-preserves-complicial}
Let $K, X, Y \in \sSet^+$, and let $X \to Y$ be a complicial map. Then the map $K \star X \to K \star Y$ is also complicial. \qed
\end{lem}

We let $\eq$ denote the marked simplicial set whose underlying simplicial set is $\Delta^3$, and whose non-degenerate marked simplices consist of the 3-simplex $\id_{[3]}$, together with all non-degenerate 2-simplices and the 1-simplices $0 \to 2$ and $1 \to 3$. The \emph{elementary saturation map} is the entire map $\eq \to (\Delta^3)^\sharp$. In general, a \emph{saturation map} is any map of the form $\Delta^n \star \eq \to \Delta^n \star (\Delta^3)^\sharp$ for $n \geq -1$. (Here $\Delta^{-1}$ denotes $\varnothing$, so that the elementary saturation map is a saturation map.)

\begin{thm}\label{complicial-model-structure}
	The category $\msSet$ carries the following model structures:
	
	\begin{enumerate}
		\item \label{model-struct-complicial} A \emph{model structure for complicial sets} in which
		\begin{itemize}
			\item cofibrations are monomorphisms;
			\item fibrant objects are complicial sets;
			\item fibrations with fibrant codomain are characterized by the right lifting property with respect to complicial horn inclusions and complicial marking extensions.
		\end{itemize}
		
		\item \label{model-struct-complicial-sat} A \emph{model structure for saturated complicial sets} in which
		\begin{itemize}
			\item cofibrations are monomorphisms;
			\item fibrant objects are saturated complicial sets;
			\item fibrations with fibrant codomain are characterized by the right lifting property with respect to complicial
			 horn inclusions, complicial marking extensions, and the saturation maps.
		\end{itemize}
		
		\item \label{model-struct-complicial-triv} A \emph{model structure for $n$-trivial complicial sets} for $n \geq 0$ in which
		\begin{itemize}
			\item cofibrations are monomorphisms;
			\item fibrant objects are $n$-trivial complicial sets;
			\item fibrations with fibrant codomain are characterized by the right lifting property with respect to complicial
			 horn inclusions, complicial marking extensions, and markings $\Delta^m \to \mDelta m$ for $m > n$.
		\end{itemize}
		
		\item \label{model-struct-complicial-sat-triv} A \emph{model structure for $n$-trivial saturated complicial sets} in which
		\begin{itemize}
			\item cofibrations are monomorphisms;
			\item fibrant objects are $n$-trivial saturated complicial sets;
			\item fibrations with fibrant codomain are characterized by the right lifting property with respect to complicial
			 horn inclusions, complicial marking extensions, saturation maps, and markings $\Delta^m \to \mDelta m$ for $m > n$.
		\end{itemize}
	\end{enumerate}
\end{thm}

\begin{proof}
	This is a combination of \cite[Thm.~100, Ex.~104, \&~Lem.~105]{verity:weak-complicial-1} and \cite[App.~B]{ozornova-rovelli:model-structure}.
\end{proof}

Note that since the terminal object is always fibrant, this includes a characterization of the fibrant objects of the model structure, which are called \defterm{($n$-trivial, saturated) complicial sets}.


\begin{lem}\label{either-side-join}
For $m, m' \geq -1$, the map $\Delta^m \star \eq \star \Delta^{m'} \to \Delta^m \star (\Delta^3)^\sharp \star \Delta^{m'}$ is a trivial cofibration in the model structures for ($n$-trivial) saturated complicial sets.
\end{lem}

\begin{proof}
This follows from \cite[Rmk.~1.20]{ozornova-rovelli:model-structure}.
\end{proof}

The involution $(-)^\op \colon \sSet \to \sSet$ extends in a natural way to $\sSet^+$, with a simplex of $X^\op$ being marked if and only if the corresponding simplex of $X$ is marked.

\begin{prop}\label{simplicial-op-equiv}
The adjunction $(-)^\op : \sSet^+ \rightleftarrows \sSet^+ : (-)^\op$ is a Quillen self-equivalence of each of the model structures of \cref{complicial-model-structure}.
\end{prop}

\begin{proof}
By \cref{involution-equiv}, it suffices to show that $(-)^\op$ is a left Quillen functor. For this, it suffices to show that it preserves the classes of complicial horn inclusions, complicial marking extensions, saturation maps, and markers. For saturation maps, this follows from the natural isomorphism $(X \star Y)^\op \cong Y^\op \star X^\op$ and the fact that $(\eq)^\op \to ((\Delta^3)^\sharp)^\op$ is isomorphic to $\eq \to (\Delta^3)^\sharp$, together with \cref{either-side-join}. For the other three classes, it is immediate from the definitions. 
\end{proof}

For our purposes, it will often be more convenient to work with an alternative to the saturation maps.

\begin{Def}
Let $L \subset \eq$ denote the regular subcomplex of $\eq$ whose underlying simplicial set consists of the  faces $\bd_0$ and $\bd_3$ of $\Delta^3$. More concretely, $L$ is the marked simplicial set illustrated below:
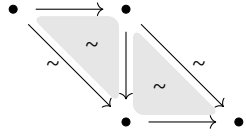
\[
\begin{tikzpicture}
\filldraw
(-1.5,1.5) circle [radius = 1.5pt]
(0,1.5) circle [radius = 1.5pt]
(0,0) circle [radius = 1.5pt]
(1.5,0) circle [radius = 1.5pt];

\draw[->] (0.3,0) -- (1.2,0);
\draw[->] (0,1.2) -- (0,0.3);
\draw[->] (-1.2,1.5) -- (-0.3,1.5);
\draw[->] (-1.3,1.3) -- (-0.2,0.2) node [midway, left] {$\sim$};
\draw[->] (0.2,1.3) -- (1.3,0.2) node [midway, right] {$\sim$};

\filldraw[shaded, rounded corners]
(0.1,0.1) -- (1.25,0.1) -- (0.1,1.25) -- cycle
(-0.1,1.4) -- (-1.25,1.4) -- (-0.1,0.25) -- cycle;

\node at (0.45,0.45) {$\sim$};
\node at (-0.45,1) {$\sim$};

\end{tikzpicture}
\]
Let $L' = \tau_0 L$, i.e.\ the simplicial set obtained by marking the three umarked 1-simplices of $L$. The \emph{elementary Rezk map} is the entire map $L \to L'$. In general, a \emph{Rezk map} is any map of the form $\Delta^n \star L \to \Delta^n \star L'$ for $n \geq -1$. 
\end{Def}

\begin{lem}\label{Rezk-diagram}
In each diagram of inclusions
	\[
	\begin{tikzcd}
		\Delta^n \star L
		\arrow [r]
		\arrow [d]
	    &
		\Delta^n \star \eq
		\arrow [d] \\
		\Delta^n \star L'
		\arrow [r] &
		\Delta^n \star(\Delta^3)^\sharp 
	\end{tikzcd}
	\]
for $n \geq -1$, the horizontal maps are complicial.
\end{lem}

\begin{proof}
We first treat the case $n = -1$ and the general statement then follows from \cref{join-preserves-complicial}.

The inclusion $L \hookrightarrow \eq$ can be obtained as a composite of two pushouts of complicial horn inclusions and one pushout of complicial marking extension:
	\[
	\begin{tikzcd}
		\Lambda^2_1
		\arrow [r, "\partial_2"]
		\arrow [d]
		\pushout
	    &
		L
		\arrow [d] 
		&
		\Lambda^3_1
		\arrow [r]
		\arrow [d]
		\pushout
		&
		L \cup \Delta^2_1
		\arrow [d]
		&
		{\Delta^3_1}' 
		\arrow [r]
		\arrow [d]
		\pushout
		&
		L \cup \Delta^2_1 \cup \Delta^3_1
		\arrow [d]
		\\
		\Delta^2_1
		\arrow [r] 
		&
		L \cup \Delta^2_1
		&
		\Delta^3_1
		\arrow [r]
		& 
		L \cup \Delta^2_1 \cup \Delta^3_1
		&
		\tau_1 \Delta^3
		\arrow [r]
		&
		\eq
	\end{tikzcd}
	\]
This proves that the inclusion $L \hookrightarrow \eq$ is complicial.

The inclusion $L' \hookrightarrow (\Delta^3)^\sharp$ is its image under $\tau_0$, which preserves complicial maps  by \cite[Lem.~23]{verity:weak-complicial-1}, and hence also complicial.
\end{proof}

\begin{cor}\label{Rezk-tcof}
Every Rezk map is a trivial cofibration in the model structures for ($n$-trivial) saturated complicial sets.
\end{cor}

\begin{proof}
This is immediate from \cref{Rezk-diagram} and the two-out-of-three property.
\end{proof}

\begin{lem} \label{join-preserves-pushouts}
  For any marked simplicial set $K$, the join functor $K \star - \colon \sSet^+ \to K \downarrow \sSet^+$ is a left adjoint.
\end{lem}

\begin{proof}
  This follows from an argument involving the analogous fact in the unmarked case (cf.~\cite[Lem.~D.2.7]{riehl-verity:elements}). 
\end{proof}

\begin{lem}\label{Rezk-pushout-complicial}
For any $n \geq -1$, the map $(\Delta^n \star L') \cup_{\Delta^n \star L} (\Delta^n \star \eq) \to \Delta^n \star (\Delta^3)^\sharp$ is complicial.
\end{lem}

\begin{proof}
The marked simplicial set $L' \cup_{L} \eq$ has $\Delta^3$ as its underlying simplicial set, with the edge $0 \to 3$ as its only unmarked simplex. Thus the inclusion $L' \cup_{L} \eq \to (\Delta^3)^\sharp$ simply marks this edge; it is therefore a pushout of the elementary 1-complicial marking extension in dimension 2, with the map $\adeltap 2 1 \to L' \cup_{L} \eq$ given by either of the faces $\bd_1$ or $\bd_2$.

It thus follows that each map $\Delta^n \star (L' \cup_{L} \eq) \to \Delta^n \star (\Delta^3)^\sharp$ is complicial by \cref{join-preserves-complicial}. 
The conclusion then follows by \cref{join-preserves-pushouts}.
\end{proof}

The following result allows us to substitute the Rezk maps for the saturation maps in the set of pseudo-generating trivial cofibrations of the ($n$-trivial) saturated complicial model structures.

\begin{prop}\label{Rezk-lift}\leavevmode
\begin{enumerate}
\item A map with fibrant codomain in the model structure for saturated complicial sets is a fibration if and only if it has the right lifting property with respect to the complicial horn inclusions, complicial marking extensions, and the Rezk maps. 
\item A map with fibrant codomain in the model structure for $n$-trivial saturated complicial sets is a fibration if and only if it has the right lifting property with respect to the complicial horn inclusions, complicial marking extensions, Rezk maps, and markers $\Delta^m \to \widetilde{\Delta}^m$ for $m > n$.
\end{enumerate}
\end{prop}

\begin{proof}
Let $X \to Y$ be a map in $\sSet^+$ having the right lifting property with respect to the complicial horn inclusions and complicial marking extensions, with $Y$ a saturated complicial set; we will show that it has the  right lifting property with respect to the saturation maps if and only if it has the right lifting property with respect to the Rezk maps.

First, suppose that $X \to Y$ has the  right lifting property with respect to the saturation maps. Then it is a fibration in the model structure for saturated complicial sets, and therefore has the right lifting property with respect to the Rezk maps by \cref{Rezk-tcof}.

Now suppose that $X \to Y$ has the  right lifting property with respect to the Rezk maps, and consider a diagram of the form
	\[
	\begin{tikzcd}
		\Delta^m \star \eq 
		\arrow [r]
		\arrow [d]
	    &
		X
		\arrow [d] \\
		\Delta^m \star (\Delta^3)^\sharp
		\arrow [r] &
		 Y
	\end{tikzcd}
	\]
for $m \geq -1$.

We may pre-compose with the diagram shown in the statement of \cref{Rezk-diagram}, and obtain a lift in the resulting composite diagram, as shown below.
	\[
	\begin{tikzcd}
	    \Delta^m \star L
	    \arrow [r]
	    \arrow [d]
	    &
		\Delta^m \star \eq 
		\arrow [r]
		\arrow [d]
	    &
		X
		\arrow [d] 
		\\
		\Delta^m \star L' \ar[r]
		\arrow [r]
		\arrow [rru, dashed]
		&
		\Delta^m \star (\Delta^3)^\sharp
		\arrow [r] 
		&
		 Y
	\end{tikzcd}
	\]
Thus we obtain a map from the pushout $(\Delta^m \star L') \cup_{\Delta^m \star L} (\Delta^m \star \eq)$ into $X$, and the following commuting diagram:
	\[
	\begin{tikzcd} [column sep = large]
		(\Delta^m \star L') \cup_{\Delta^m \star L} (\Delta^m \star \eq) 
		\arrow [r]
		\arrow [d]
	    &
		X
		\arrow [d] \\
		\Delta^m \star (\Delta^3)^\sharp
		\arrow [r] &
		 Y
	\end{tikzcd}
	\]
The left-hand map is complicial by \cref{Rezk-pushout-complicial}, thus this diagram admits a lift. Pre-composing with the pushout inclusion $\Delta^m \star \eq \to (\Delta^m \star L') \cup_{\Delta^m \star L} (\Delta^m \star \eq)$, we obtain a lift in the original diagram.
\end{proof}

\begin{Def}
	Let $[n] \in \Delta$ and let $0 \leq p,q \leq n$ be such that $p+q = n$.
	Then we write $\frontface p q \colon [p] \to [n]$ for the simplicial operator $i \mapsto i$, and $\backface p q \colon [q] \to [n]$ for the operator $i \mapsto p+i$.
\end{Def}

%

\begin{Def}
	Let $X,Y \in \mSet$, let $(x,y) \in X_n \times Y_n$ be a simplex of $X \times Y$, and let $0 \leq i \leq n$. We say that $(x,y)$ is \defterm{$i$-cloven} if \emph{either} $x   \frontface i {n-i}$ is marked in $X$ or $y   \backface i {n-i}$ is marked in $Y$. We say that $(x,y)$ is \defterm{fully cloven} if it is $i$-cloven for all $0 \leq i \leq n$.
	
	The \defterm{Gray tensor product} of $X$ and $Y$, denoted $X \otimes Y$, is defined to be the marked simplicial set with underlying simplicial set $X \times Y$, where a simplex $(x,y) \in X_n \times Y_n$ is marked if and only if it is fully cloven.
\end{Def}

\begin{thm}[{\cite[Lem.~131]{verity:complicial}}]\label{thm:gray}
	The Gray tensor product endows $\mSet$ with a (nonsymmetric) monoidal structure, such that the forgetful functor $(\mSet, \otimes) \to (\sSet, \times)$ is strict monoidal. \qed
\end{thm}

\begin{prop}\label{gray-homotopical}
In any of the model structures of \cref{complicial-model-structure}, given a pair of cofibrations $i \colon A \to B, j \colon X \to Y$ in $\sSet$, the pushout Gray tensor product $i \hat{\otimes} j \colon A \otimes Y \cup_{A \otimes X} B \otimes X \to B \otimes Y$ is a cofibration. Moreover, if either $i$ or $j$ is trivial then so is $i \hat{\otimes} j$.
\end{prop}

\begin{proof}
That the pushout Gray tensor product of cofibrations is a cofibration follows from the corresponding result for the cartesian product on $\sSet$. The remainder of the statement is immediate from \cite[Cor.~2.3]{ozornova-rovelli-verity}.
\end{proof}

\begin{Def}
	A \defterm{pre-complicial set} is a marked simplicial set $X$ with the right lifting property with respect to the complicial marking extensions. These form a reflective subcategory of $\mSet$ which we will denote $\PreComp$ (cf.~\cite[Def.~121 \& p.~68]{verity:complicial}). We will denote the localization functor $X \mapsto X^\precomp$; for $X \in \sSet^+$, the pre-complicial set $X^\precomp$ will be referred to as the \emph{pre-complicial reflection} of $X$.
\end{Def}

\begin{prop}[{\cite[Thm.~1.31]{campion-kapulkin-maehara}}]\label{precomp-tcof}
For every $X \in \sSet^+$, the unit map $X \to X^\precomp$ is a trivial cofibration in all of the model structures of \cref{complicial-model-structure}. \qed
\end{prop}

\begin{lem}\label{more-markings}
Let $X \in \sSet^+$ and $S$ a set of simplices of $X$ which are marked in $X^\precomp$. 
Let $X^\dag$ denote the marked simplicial set obtained from $X$ by marking all simplices of $S$. 
Then the entire map $X \to X^\dag$ is a trivial cofibration.
\end{lem}

\begin{proof}
Any entire map is a cofibration. To see that $X \to X^\dag$ is a weak equivalence, consider the following commuting diagram:
	\[
	\begin{tikzcd}
		X
		\arrow [d]
		\arrow [rd]
	    &
	     \\
		X^\dag
		\arrow [r] &
		X^\precomp 
	\end{tikzcd}
	\]

The map $X \to X^\precomp$ is a trivial cofibration. Because every simplex in $X$ is marked in $X^\precomp$, the map $X^\dag \to X^\precomp$ is a pushout of $X \to X^\precomp$, hence it is a trivial cofibration as well. Thus $X \to X^\dag$ is a weak equivalence by two-out-of-three.
\end{proof}

\subsection{Cubical sets}\label{cSet-unmarked}

In this section, we review the basic theory of cubical sets and their homotopy theory.
Cubical sets are presheaves on the \emph{box category} $\Box$. The objects of $\Box$ are posets of the form $[1]^n$ and the maps are generated (inside the category of posets) under composition by the following four special classes:
\begin{itemize}
  \item \emph{faces} $\partial^n_{i,\varepsilon} \colon [1]^{n-1} \to [1]^n$ for $i = 1, \ldots , n$ and $\varepsilon = 0, 1$ given by:
  \[ \partial^n_{i,\varepsilon} (x_1, x_2, \ldots, x_{n-1}) = (x_1, x_2, \ldots, x_{i-1}, \varepsilon, x_i, \ldots, x_{n-1})\text{;}  \]
  \item \emph{degeneracies} $\sigma^n_i \colon [1]^n \to [1]^{n-1}$ for $i = 1, 2, \ldots, n$ given by:
  \[ \sigma^n_i ( x_1, x_2, \ldots, x_n) = (x_1, x_2, \ldots, x_{i-1}, x_{i+1}, \ldots, x_n)\text{;}  \]
  \item \emph{negative connections} $\gamma^n_{i,0} \colon [1]^n \to [1]^{n-1}$ for $i = 1, 2, \ldots, n-1$ given by:
  \[ \gamma^n_{i,0} (x_1, x_2, \ldots, x_n) = (x_1, x_2, \ldots, x_{i-1}, \max\{ x_i , x_{i+1}\}, x_{i+2}, \ldots, x_n) \text{.} \]
  \item \emph{positive connections} $\gamma^n_{i,1} \colon [1]^n \to [1]^{n-1}$ for $i = 1, 2, \ldots, n-1$ given by:
  \[ \gamma^n_{i,1} (x_1, x_2, \ldots, x_n) = (x_1, x_2, \ldots, x_{i-1}, \min\{ x_i , x_{i+1}\}, x_{i+2}, \ldots, x_n) \text{.} \]
\end{itemize}

These maps obey the following \emph{cubical identities} (here onwards we generally omit the superscript $n$ on face, degeneracy, and connection maps whenever it is irrelevant or can be deduced from context):

\begin{multicols}{2}
$\partial_{j, \varepsilon'} \partial_{i, \varepsilon} = \partial_{i+1, \varepsilon} \partial_{j, \varepsilon'}$ for $j \leq i$;

$\sigma_i \sigma_j = \sigma_j \sigma_{i+1} \quad \text{for } j \leq i$;

$\sigma_j \partial_{i, \varepsilon} = \left\{ \begin{array}{ll}
\partial_{i-1, \varepsilon} \sigma_j   & \text{for } j < i \text{;} \\
\id                                                       & \text{for } j = i \text{;} \\
\partial_{i, \varepsilon} \sigma_{j-1} & \text{for } j > i \text{;}
\end{array}\right.$

$\gamma_{j,\varepsilon'} \gamma_{i,\varepsilon} = \left\{ \begin{array}{ll} \gamma_{i,\varepsilon} \gamma_{j+1,\varepsilon'} & \text{for } j > i \text{;} \\
\gamma_{i,\varepsilon}\gamma_{i+1,\varepsilon} & \text{for } j = i, \varepsilon' = \varepsilon \text{;}\\
\end{array}\right.$

$\gamma_{j,\varepsilon'} \partial_{i, \varepsilon} =  \left\{ \begin{array}{ll}
\partial_{i-1, \varepsilon} \gamma_{j,\varepsilon'}   & \text{for } j < i-1 \text{;} \\
\id                                                         & \text{for } j = i-1, \, i, \, \varepsilon = \varepsilon' \text{;} \\
\partial_{i, \varepsilon} \sigma_i         & \text{for } j = i-1, \, i, \, \varepsilon = 1-\varepsilon' \text{;} \\
\partial_{i, \varepsilon} \gamma_{j-1,\varepsilon'} & \text{for } j > i \text{;} 
\end{array}\right.$

$\sigma_j \gamma_{i,\varepsilon} =  \left\{ \begin{array}{ll}
\gamma_{i-1,\varepsilon} \sigma_j  & \text{for } j < i \text{;} \\
\sigma_i \sigma_i           & \text{for } j = i \text{;} \\
\gamma_{i,\varepsilon} \sigma_{j+1} & \text{for } j > i \text{.} 
\end{array}\right.$
\end{multicols}

A \emph{cubical set} is contravariant functor $\Box^{\op} \to \Set$ and a cubical map is a natural transformation of such functors.
We will write $\cSet$ for the category of cubical sets.
As in the simplicial case, we will write the action of cubical operators on the right.

We write $\Box^n$ for the representable cubical set, represented by $[1]^n$.
We furthermore write $\partial \Box^n$ for the maximal proper subobject of $\Box^n$, which is spanned by all of its faces $\partial_{i, \varepsilon} \colon \Box^{n-1} \hookrightarrow \Box^n$, and $\sqcap^n_{i, \varepsilon}$ for the subobject of $\Box^n$ spanned by all faces except $\partial_{i, \varepsilon}$.

An $n$-cube $x \colon \Box^n \to X$ is \emph{degenerate} if it is in the image of a degeneracy map $\sigma_i$ or a connection $\gamma_{i, \varepsilon}$.
Otherwise it is \emph{non-degenerate}.

When dealing with representable cubical sets, we will sometimes have occasion to consider the \emph{sign} and \emph{parity} of their faces of codimension 1. A face $\bd_{i,\varepsilon}$ of $\Box^n$ is:

\begin{itemize}
\item \emph{negative} if $\varepsilon = 0$;
\item \emph{positive} if $\varepsilon = 1$;
\item \emph{even} if $i + \varepsilon$ is even;
\item \emph{odd} if $i+ \varepsilon$ is odd.
\end{itemize}

We will occasionally represent cubical sets using pictures.
In doing so, we will typically follow the conventions used in \cite{doherty-kapulkin-lindsey-sattler}, in which $0$-cubes are represented as vertices, $1$-cubes as arrows, $2$-cubes as squares, and $3$-cubes as cubes.

For a $1$-cube $f$, we draw
\[
\xymatrix{ x \ar[r]^f & y} \]
to indicate $x = f \partial_{1,0}$ and $y = f \partial_{1,1}$.
For a $2$-cube $s$, we typically draw
\[
\xymatrix{
 x
  \ar[r]^h
  \ar[d]_f
&
   y
  \ar[d]^g
\\
  z
  \ar[r]^k
&
 w
}
\]
to indicate $s\partial_{1,0} = f$,  $s\partial_{1,1} = g$, $s\partial_{2,0} = h$,  and $s\partial_{2,1} = k$. In cases where it is not feasible to consistently assign each cubical dimension to a unique visual direction, such as when the $(1,\varepsilon)$-face of one 2-cube is identified with the $(2,\varepsilon')$-face of another, we will instead use the convention of \cite{campion-kapulkin-maehara}, drawing an arrow inside each 2-cube pointing from its odd faces towards its even faces, as illustrated below.

\[
\begin{tikzpicture}[baseline = 12]
\foreach \x in {0,1,3,4}
\filldraw (\x,0) circle [radius = 1pt] (\x,1) circle [radius = 1pt];
\draw[->] (0.2,1) -- (0.8,1) node[midway, above]{\small{$\bd_{2,0}$}};
\draw[->] (1,0.8) -- (1,0.2) node[midway, right]{\small{$\bd_{1,1}$}};
\draw[->] (3.2,0) -- (3.8,0) node[midway, below]{\small{$\bd_{1,1}$}};
\draw[->] (0,0.8) -- (0,0.2) node[midway, left]{\small{$\bd_{1,0}$}};
\draw[->] (0.2,0) -- (0.8,0) node[midway, below]{\small{$\bd_{2,1}$}};
\draw[->] (3.2,1) -- (3.8,1) node[midway, above]{\small{$\bd_{1,0}$}};
\draw[->] (3,0.8) -- (3,0.2) node[midway, left]{\small{$\bd_{2,0}$}};
\draw[->] (4,0.8) -- (4,0.2) node[midway, right]{\small{$\bd_{2,1}$}};

\draw[->, double] (0.3,0.3) -- (0.7,0.7);
\draw[->, double] (3.7,0.7) -- (3.3,0.3);

\end{tikzpicture}
\]

As for the convention when drawing $3$-dimensional boxes, we use the following ordering of axes:
\[
\begin{tikzpicture}
\filldraw
(0,0) circle [radius = 1pt]	
(2,0) circle [radius = 1pt]	
(0,-2) circle [radius = 1pt]	
(1,-1) circle [radius = 1pt];

\draw[->] (0.2,0) -- (1.8,0) node [midway, above] {$1$};
\draw[->] (0.1,-0.1) -- (0.9,-0.9) node [midway, below] {$3$};
\draw[->] (0,-0.2) -- (0,-1.8) node [midway, left] {$2$};
\end{tikzpicture}
\]
For readability, we do not label $2$- and $3$-cubes.
Similarly, if a specific $0$-cube is irrelevant for the argument or can be inferred from the context, we represent it by $\bullet$, and we omit labels on edges whenever the label is not relevant for the argument.

Lastly, a degenerate $1$-cube $x \sigma_1$ on $x$ is represented by
\[
\xymatrix{ x \ar@{=}[r] & x\text{,}} \]
while a $2$- or $3$-cube whose boundary agrees with that of a degenerate cube is assumed to be degenerate unless indicated otherwise.
For instance, a $2$-cube depicted as
\[
\xymatrix{
 x
  \ar@{=}[r]
  \ar[d]_f
&
   x
  \ar[d]^f
\\
  y
  \ar@{=}[r]
&
 y
}
\]
represents $f \sigma_1$.

The category $\Box$ is an EZ-Reedy category, with $\Box_+$ generated under composition by the face maps $\partial_{i, \varepsilon}$ and $\Box_{-}$ generated by the degeneracies $\sigma_i$ and connections $\gamma_{i,\varepsilon}$.  We will refer to morphisms of $\Box_+$ as \emph{face maps}, or \emph{composite face maps} where greater precision is desired.

We briefly recall certain key definitions and results from the theory of EZ-Reedy categories which will be of use in our study of cubical sets. Throughout the following, let $\catC$ denote an EZ-Reedy category, and $X$ a presheaf on $\catC$.

\begin{itemize}
\item For $c \in \catC$, an element $x \in X_c$ is \emph{degenerate} if it is equal to $x' \phi$ for some non-identity $\phi \colon c' \to c$ in $\catC_{-}$ and some $x' \in X_{c'}$; we then refer to $x$ as a \emph{degeneracy} of $x'$. An element which is not degenerate is \emph{non-degenerate}.
\item The \emph{Eilenberg-Zilber lemma} states that for every $c \in \catC$, every element $x \in X_c$ is equal to $x' \phi$ for a unique (possibly identity) $\phi \colon c' \to c$ in $\catC_{-}$ and a unique non-degenerate $x' \in X_{c'}$.
\item For $n \geq -1$, the \emph{$n$-skeleton} of $X$, denoted $\sk_{n} X$, is the subcomplex of $X$ consisting of all elements of $X_c$ for all $c \in \catC$ of degree less than or equal to $n$, together with their degeneracies. (In the case $n = -1$ this is simply the empty presheaf $\varnothing$.) A presheaf $X$ is \emph{$n$-skeletal} if it is equal to $\sk_{n} X$, or equivalently, if it has no non-degenerate elements above degree $n$.
\item For $c \in \catC$ of degree $n$, the \emph{boundary} of the representable presheaf $\catC(-,c)$, denoted $\bd \catC(-,c)$, is its $(n-1)$-skeleton, i.e. $\sk_{n-1} \catC(-,c)$. The \emph{boundary} of an element $x \in X_c$ is the composite of the corresponding map $\catC(-,c) \to X$ with the inclusion $\bd \catC(-,c) \hookrightarrow \catC(-,c)$.
\item Each presheaf $X$ is the colimit of the inclusions of its skeleta $\sk_{n} X \hookrightarrow \sk_{n+1} X$. Moreover, for $n \geq 0$, we may construct $\sk_n X$ from $\sk_{n-1} X$ by ``gluing in elements of degree $n$ along their boundaries''; more precisely, we have a pushout diagram
\[
\begin{tikzcd}
\coprod\limits_{\catC(-,c) \to X}\bd \catC(-,c) \pushout  \arrow[r] \arrow[d] & \mathrm{sk}_{n-1} X \arrow[d] \\
\coprod\limits_{\catC(-,c) \to X} \catC(-,c) \arrow[r] & \sk_n X  \\
\end{tikzcd}
\]
where the coproducts range over all elements of $X_c$ for $c \in \catC$ of degree $n$  Thus we may prove results about presheaves by the technique of \emph{induction on skeleta}, i.e. prove that some result holds for representable presheaves and then extend to all presheaves via these colimit decompositions.
\end{itemize}

In the specific case $\catC = \Box$, these definitions of degenerate elements and boundaries coincide with those we have given for degenerate cubes and boundaries of $n$-cubes. Degeneracies in cubical sets, in the sense of this general definition, consist of degeneracies and connections (and their composites), while the $n$-skeleton of a cubical set $X$ is its subcomplex whose non-degenerate cubes consist of all non-degenerate cubes of dimension less than or equal to $n$.

Of particular importance to us will be the following characterization of maps in $\Box$.

\begin{prop}[{\cite[Thm.~5.1]{grandis-mauri}}] \label{cube-standard-form}
  Every map in the category $\Box$ can be factored uniquely as a composite
  \[ (\partial_{c_1, \varepsilon'_1} \ldots \partial_{c_r, \varepsilon'_r})
     (\gamma_{b_1,\varepsilon_1} \ldots \gamma_{b_q,\varepsilon_q})
     (\sigma_{a_1} \ldots \sigma_{a_p})\text{,} \]
  where $1 \leq a_1 < \ldots < a_p$, $1 \leq b_1 \leq \ldots \leq b_q$, $b_i < b_{i+1}$ if $\varepsilon_{i} = \varepsilon_{i+1}$, and $c_1 > \ldots > c_r \geq 1$.   \qed
\end{prop}

We will refer to the decomposition given by \cref{cube-standard-form} as the \emph{standard form} of $\delta$.

Similarly, by the Eilenberg-Zilber lemma, any cube $x$ in a cubical set $X$ may be written as 
\[
x'
     (\gamma_{b_1,\varepsilon_1} \ldots \gamma_{b_q,\varepsilon_q})
     (\sigma_{a_1} \ldots \sigma_{a_p})
\]
for some unique non-degenerate cube $x'$ of $X$ and some unique (possibly empty) sequences $1 \leq a_1 < \ldots < a_p$, $1 \leq b_1 \leq \ldots \leq b_q$, with $b_i < b_{i+1}$ if $\varepsilon_{i} = \varepsilon_{i+1}$. We refer to this as the \emph{standard form} of $x$.

We will occasionally use the phrase ``the standard form of $x$ is $z\psi$'', where $\psi$ is a map in $\Box$; this is taken to mean that $\psi$ is the final map in the standard form of $x$. 
For instance, for a non-degenerate cube $y$, the standard form of $y \sigma_1 \sigma_2$ is $z \sigma_2$, where $z = y \sigma_1$.

The category $\Box$ admits a monoidal product $\otimes \colon \Box \times \Box \to \Box$ given by $[1]^m \otimes [1]^n = [1]^{m+n}$.
Its extension to $\cSet$ via Day convolution is the \emph{geometric product} of cubical sets, also denoted by $\otimes$.
We will often need its more explicit description.

\begin{prop}[{\cite[Prop.~1.24]{doherty-kapulkin-lindsey-sattler}}]\label{geo-prod-description}
  Given cubical sets $X$ and $Y$, we have the following description of their geometric product $X \otimes Y$.
\begin{itemize}
	\item For $n \geq 0$, the $n$-cubes in $X \otimes Y$ are the formal products $x \otimes y$ of pairs $x \in X_k$ and $y \in Y_\ell$ such that $k+\ell = n$, subject to the identification $(x\sigma_{k+1})\otimes y = x\otimes(y\sigma_{1})$.
 	\item For $x \in X_k$ and $y \in Y_\ell$, the faces, degeneracies, and connections of the $(k+\ell)$-cube $x \otimes y$ are computed as follows:
 	\begin{itemize}
 		\item $(x\otimes y)\partial_{i,\varepsilon} = 
 		\begin{cases} (x\partial_{i,\varepsilon})\otimes y & 1 \leq i \leq k \\ 
 		x\otimes( y\partial_{i-k,\varepsilon})  & k + 1 \leq i \leq k+l
 		\end{cases}$
 		\item $(x\otimes y)\sigma_{i} = 
 		\begin{cases} (x\sigma_{i})\otimes y     & 1 \leq i \leq k + 1 \\
 		x\otimes (y\sigma_{i-k}) & k + 1 \leq i \leq k+l+ 1 
 		\end{cases}$  
 		\item $(x\otimes y)\gamma_{i,\varepsilon} = 
 		\begin{cases} (x\gamma_{i, \varepsilon})\otimes y     & 1 \leq i \leq k \\ 
 		x\otimes (y\gamma_{i-k,\varepsilon}) & k + 1 \leq i \leq k+l
 		\end{cases}$
 	\end{itemize}
 \end{itemize}
 In particular, an $n$-cube $x \otimes y$ of $X \otimes Y$ is non-degenerate exactly when both $x$ and $y$ are non-degenerate in $X$ and $Y$, respectively. \qed
 \end{prop}
 
The \emph{triangulation} functor $T \colon \cSet \to \sSet$ is given by the Yoneda extension of the co-cubical object $\Box \to \sSet$ given by $[1]^n \mapsto (\Delta^1)^n$.
We write $U \colon \sSet \to \cSet$ for its right adjoint, given by $(UX)_n = \sSet((\Delta^1)^n, X)$.

The category $\Cube$ admits two canonical identity-on-objects involutions $(-)^\co, (-)^\coop \colon \Cube \to \Cube$.
The first maps $\partial^n_{i, \varepsilon}$ to $\partial^n_{n+1-i, \varepsilon}$, $\sigma^n_i$ to $\sigma^n_{n+1-i}$, and $\gamma^n_{i,\varepsilon}$ to $\gamma^n_{n+1-i, \varepsilon}$, while the second maps $\partial^n_{i, \varepsilon}$ to $\partial^n_{i, 1- \varepsilon}$, $\sigma^n_i$ to $\sigma^n_{i}$, and $\gamma^n_{i,\varepsilon}$ to $\gamma^n_{i, 1- \varepsilon}$.
We write $(-)^\op$ for the composite of those.
Precomposition with these automorphisms induces involutions also denoted $(-)^\co , (-)^\coop, (-)^\op \colon \cSet \to \cSet$.

\begin{prop}[{\cite[Prop.~1.17]{campion-kapulkin-maehara}}]\label{op-monoidal}
  The functors $(-)^\co , (-)^\op \colon \cSet \to \cSet$ are strong anti-monoidal, while $(-)^\coop  \colon \cSet \to \cSet$ is strong monoidal. \qed
\end{prop}

One may also consider various other notions of cubical sets, defined as presheaves on different box categories; see \cite{buchholtz-morehouse:varieties-of-cubes} for an overview. In particular, we may consider the following categories of cubical sets, defined as presheaves on certain wide subcategories of $\Box$:

\begin{itemize}
\item $\cSet_{\varnothing}$, the category of \emph{minimal cubical sets}, defined on the subcategory $\Box_{\varnothing} \subseteq \Box$ generated under composition by face and degeneracy maps;
\item $\cSet_0$, the category of \emph{cubical sets with negative connections}, defined on the subcategory $\Box_0 \subseteq \Box$ generated by faces, degeneracies, and negative connections;
\item $\cSet_1$, the category of \emph{cubical sets with positive connections}, defined on the subcategory $\Box_1 \subseteq \Box$ generated by faces, degeneracies, and positive connections.
\end{itemize}

The definitions of the geometric product and triangulation, as well as \cref{cube-standard-form,geo-prod-description}, generalize to these categories. The involutions $(-)^{\co}, (-)^{\coop}, (-)^{\op}$ restrict to involutions of $\Box_{\varnothing}$, which in turn induce involutions of $\cSet_{\varnothing}$. For $\Box_0$ and $\Box_1$, $(-)^\co$ again restricts to an involution on each of these categories, while the restrictions of $(-)^\coop$ and $(-)^\op$ define inverse pairs of isomorphisms between them, from which we obtain isomorphisms $\cSet_0 \cong \cSet_1$. \cref{op-monoidal} then generalizes to these categories as well, but must now be interpreted as referring to these isomorphisms.

\subsection{Marked cubical sets}\label{cSet-marked}

The objects of $\Cube^+$ consist of: $[1]^n$ for every $n \geq 0$ and $[1]^n_e$ for every $n \geq 1$.
The morphisms of $\Cube^+$ are generated by the maps
\begin{itemize}
	\item[-] $\partial^n_{i, \varepsilon} \colon [1]^{n-1} \to [1]^n$ for every $n \geq 1$, $i = 1, \ldots, n$, and $\varepsilon = 0, 1$,
	\item[-] $\sigma^n_i \colon [1]^n \to [1]^{n-1}$ for $n \geq 1$ and $i = 1, \ldots, n$,
	\item[-] $\gamma^n_i \colon [1]^n \to [1]^{n-1}$ for $n \geq 2$, $i = 1, \ldots, n-1$, and $\varepsilon = 0, 1$,
	\item[-] $\varphi^n \colon [1]^n \to [1]^n_e$ for $n \geq 1$,
	\item[-] $\zeta^n_i \colon [1]^n_e \to [1]^{n-1}$ for $n \geq 1$ and $i = 1, \ldots, n$,
	\item[-] $\xi^n_{i, \varepsilon} \colon [1]^n_e \to [1]^{n-1}$ for $n \geq 1$, $i = 1, \ldots, n$, and $\varepsilon = 0, 1$,
\end{itemize}
subject to the usual cubical identities and the following additional relations:
\begin{multicols}{2}
  $\zeta_i \varphi = \sigma_i$;

  $\xi_{i, \varepsilon} \varphi = \gamma_{i, \varepsilon}$;

  $\sigma_i \zeta_j = \sigma_j \zeta_{i+1}$ for $j \leq i$;

  $\gamma_{j, \varepsilon} \xi_{i, \delta} = \left\{ \begin{array}{ll}
	\gamma_{i, \delta} \xi_{j, \varepsilon} & \text{for } j > i\text{;} \\
	\gamma_{i, \delta} \xi_{i+1, \delta} & \text{for } j=i \text{ and } \delta = \varepsilon\text{;}
	\end{array}\right.$	

 $\sigma_j \xi_{i, \delta} =  \left\{ \begin{array}{ll}
	\gamma_{i-1, \delta} \zeta_j    & \text{for } j < i \text{;} \\
	\sigma_i \zeta_i                        & \text{for } j = i \text{;} \\
	\gamma_{i, \delta} \zeta_{j+1} & \text{for } j > i \text{.} 
	\end{array}\right.$
\end{multicols}

A \emph{structurally marked cubical set} is a functor $(\Cube^+)^{\op} \to \Set$ and a map of structurally marked cubical sets is a natural transformation of such functors.
We think of a structurally marked cubical set as a cubical set in which $n$-cubes can carry markings.

A \emph{marked cubical set} is a structurally marked cubical set $X \colon (\Cube^+)^\op \to \Set$ such that the maps $\varphi^* \colon eX_n \to X_n$ are monomorphisms for all $n \geq 1$.
We write $\mcSet$ for the full subcategory of $\Set^{(\Cube^+)^\op}$ spanned by the marked cubical sets.

In other words, an $n$-cube of a marked cubical set can be marked at most once, and hence marking is property of a cube, not additional structure on it.
The canonical inclusion $\mcSet \hookrightarrow \Set^{(\Box^+)^{\op}}$ admits a left adjoint obtained by factoring all the functions $\varphi^* \colon eX_n \to X_n$ via their image into a surjection followed by an injection.
This establishes $\mcSet$ as a locally presentable category.

When representing cubical sets visually, we will label marked 1-cubes with a tilde:
\[
\begin{tikzpicture}
	\filldraw
	(0,0) circle [radius = 1pt]
	(1,0) circle [radius = 1pt];
	
	\draw[->]
	(0.2,0) -- (0.8,0) node[midway, above] {$\sim$};
\end{tikzpicture}
\]
Marked 2-cubes will be shaded, and sometimes labelled with tildes in their interiors:
\[
\begin{tikzpicture}
	\foreach \x in {0,1,2,3,4,5}
	\filldraw
	(\x,0) circle [radius = 1pt]
	(\x,1) circle [radius = 1pt];
	
	\foreach \x in {0,1,2,3,4,5}
	\draw[->] (\x,0.8) -- (\x,0.2);
	
	\foreach \x in {0,2,4}
	\draw[->] (\x+0.2,0) -- (\x+0.8,0);
	
	\foreach \x in {0,2,4}
	\draw[->] (\x+0.2,1) -- (\x+0.8,1);
	
	\foreach \x in {0,2,4}
	\filldraw[shaded, rounded corners]
	(\x+0.1,0.1) -- (\x+0.9,0.1) -- (\x+0.9,0.9) -- (\x+0.1,0.9) -- cycle;
	
	\node at (0.5,0.5) {$\sim$};
	
	\draw[double = shaded, ->]
	(2.3,0.3) -- (2.7,0.7);
	\draw[double = shaded, ->]
	(4.7,0.7) -- (4.3,0.3);
\end{tikzpicture}
\]
We will describe the markings of 3-cubes in accompanying text.

As in the simplicial case, for each $n \geq 0$ we have a left adjoint functor $\tau_n \colon \mcSet \to \mcSet$ which marks all cubes above dimension $n$. Likewise, as with $(-)^\op \colon \sSet \to \sSet$, the involutions $(-)^\co, (-)^\coop, (-)^\op \colon \cSet \to \cSet$ extend naturally to define involutions of $\cSet^+$.

Similarly, the forgetful functor $\mcSet \to \cSet$ admits both the left and the right adjoint.
The left adjoint, denoted $X \mapsto X^\flat$ equips the cubical set $X$ with the \defterm{minimal marking}, i.e., only the degenerate cubes of $X$ are marked in $X^\flat$.
The right adjoint, denoted $X \mapsto X^\sharp$ equips the cubical set $X$ with the \defterm{maximal marking}, i.e., all cubes are marked in $X^\sharp$.

We again abuse the notation by writing $\Cube^n$ for $(\Cube^n)^\flat$.
The marked $n$-cube $\tau_{n-1}(\Cube^n)$ will be denoted $\mcube n$.

\begin{prop}[{\cite[Prop.~2.10]{campion-kapulkin-maehara}}] \label{cset-cellular-model}
  The monomorphisms of $\mcSet$ (and $\Set^{(\Cube^+)^\op}$) are the cellular closure of the set
  \[ I = \{ \partial \Cube^n \hookrightarrow \Cube^n \ | \ n \geq 0 \}
  \cup \{ \Cube^n \hookrightarrow \mcube n \ | \ n \geq 1 \}\text{.} \]
  
\end{prop}

The geometric product of cubical sets can be extended to the marked setting in two ways, yielding either lax or pseudo Gray tensor product.
In this paper, we will work exclusively with the former.

	The \emph{lax Gray tensor product} $X \otimes Y$ of two marked cubical sets $X,Y \in \mcSet$ is given by:
	\begin{itemize}
	  \item the underlying cubical set is given by the geometric product of $X$ and $Y$;
	  \item a non-degenerate cube $x \otimes y$ is marked if and only if either $x$ is marked in $X$ or $y$ is marked in $Y$.
    \end{itemize}
	This extends to a functor $\otimes \colon \mcSet \times \mcSet \to \mcSet$, equipping $\mcSet$ with a biclosed monoidal structure (cf.~\cite[Thm.~2.16]{campion-kapulkin-maehara}). For $X \in \mcSet$, we denote the right adjoints of the functors $- \otimes X$ and $X \otimes -$ by $\uhom_{L}(X,-)$ and $\uhom_{R}(X,-)$, respectively. Explicitly, these functors are given by $\uhom_L(X,Y)_{n} = \mcSet(\Box^n \otimes X,Y)$, $ \uhom_R(X,Y)_{n} = \mcSet(X \otimes \Box^n, Y)$, with an $n$-cube of $\uhom_{L}(X,Y)$ being marked if the corresponding map factors through $\widetilde{\Box}^n \otimes X$, and similarly for $\uhom_{R}(X,Y)$.
	
	Finally, recall that a map $f \colon X \to Y$ of marked cubical sets is:
  \begin{itemize}
    \item \emph{regular} if it creates markings, i.e., for an $n$-cube $x$ of $X$ we have: $x \in eX_n$ if and only if $f(x) \in eY_n$;
    \item \emph{entire} if the induced map between the underlying cubical sets is invertible.
\end{itemize}

We then have the following closure properties of regular/entire morphisms.

\begin{lem}[cf.~{\cite[Lem.~2.17]{campion-kapulkin-maehara}}]\label{Gray-tensor-of-monos}
	Let $f$ and $g$ be monomorphisms in $\mcSet$.
	\begin{enumerate}
		\item If both $f$ and $g$ are regular, then so is $f \hat \otimes g$.
		\item If either $f$ or $g$ is entire, then so is $f \hat \otimes g$.
		\item If both $f$ and $g$ are entire, then $f \hat \otimes g$ is an isomorphism. \qed
  \end{enumerate}
\end{lem}

Lastly, we may extend the triangulation functor $T \colon \cSet \to \sSet$ to the marked setting, following \cite{campion-kapulkin-maehara}. To do this, we first need an explicit description of the simplices of $T\Box^n = (\Delta^1)^n = N[1]^n$. For $r \geq 0$, observe that since $\Delta^r = N[r]$ and the nerve functor is fully faithful, $r$-simplices $\Delta^r \to (\Delta^1)^n$ can be identified with order-preserving maps $\phi \colon [r] \to [1]^n$. Such a map $\phi$ can be identified with a unique function $\{1, \ldots ,n\} \to \{1, \ldots ,r,\pm \infty\}$, defined as follows:

\[
i \mapsto \left\{\begin{array}{cl}
+\infty, & \pi_i \circ \phi(r) = 0,\\
p, & \pi_i \circ \phi(p-1) = 0~\text{and}~\pi_i \circ \phi(p) = 1,\\
-\infty, & \pi_i \circ \phi(0) = 1.
\end{array}\right.
\]
  
It will typically be convenient to represent such functions as strings of length $n$ with entries drawn from the set $\{1, \ldots ,r,\pm \infty\}$ (for brevity, we will write $+$ for $+\infty$ and $-$ for $-\infty$). We let $\iota_n$ denote the inclusion $\{1, \ldots ,n\} \to \{1, \ldots ,n,\pm \infty\}$, viewed as an $n$-simplex of $T \Box^n$; represented as a string, this is $1 \ldots n$.
  
Under this identification, a simplicial operator $\alpha \colon [q] \to [r]$ sends an $r$-simplex $\phi$ to the $q$-simplex $\alpha$ defined as follows:

\[
	(\phi   \alpha)(i) = \left\{\begin{array}{cl}
	+\infty, & \phi(i) > \alpha(q),\\
	p, & \alpha(p-1) < \phi(i) \le \alpha(p),\\
	-\infty, & \phi(i) \le \alpha(0).
	\end{array}\right.
	\]

In particular, when representing simplices as strings, face maps of an $m$-simplex $\phi$ can be computed as follows:

\begin{itemize}
\item The face $\phi \bd_0$ is computed by replacing every 1 in $\phi$ by $-$, and reducing all other entries by 1. For instance, $(1\,2\,3+-) \bd_0 = -1\,2+-$.
\item For $0 < i < m$, the face $\phi \bd_i$ is computed by reducing every entry of $\phi$ which is greater than $i$ by 1. For instance, $(1\,2\,3+-) \bd_1 = 1\,1\,2+-$, while $(1\,2\,3+-) \bd_2 = 1\,2\,2+-$.
\item The face $\phi \bd_n$ is computed by replacing every $n$ in $\phi$ by $+$. For instance, $(1\,2\,3+-) \bd_3 = 1\,2++\,-$.
\end{itemize}

Alternatively, we may view every face map $\bd_i$ as being computed by reducing all entries of $\phi$ which are greater than $i$ by 1, identifying entries less than $1$ with $-$ and entries greater than $n-1$ with $+$ when dealing with $(n-1)$-simplices.

Likewise, the degeneracy $\phi \sigma_i$ can be computed by raising all entries of $\phi$ greater than $i$ by 1; from this we obtain the following result.

\begin{lem}\label{T-cube-degen}
An $r$-simplex $\phi \colon \{1, \ldots ,n\} \to \{1, \ldots ,r,\pm \infty\}$ of $(\Delta^1)^n$ is degenerate if and only if there is some $i \in \{1, \ldots ,r\}$ for which $\phi^{-1}(i) = \varnothing$. \qed
\end{lem}

\begin{Def} \label{marked-triangulation}
We define the functor $T \colon \Box^+ \to \sSet^{+}$ as follows:
\begin{itemize}
\item $T[1]^n$ has $(\Delta^1)^n$ as its underlying simplicial set, with an $r$-simplex $\phi \colon \{1, \ldots,n\} \to \{1, \ldots ,r,\pm \infty\}$ unmarked if and only if there exists a sequence $i_1 < \cdots < i_r$ in $\{1, \ldots ,n\}$ such that $\phi(i_p) = p$ for all $p \in \{1, \ldots ,r\}$;
\item $T[1]^n_e$ is obtained from $T[1]^n$ by marking the $n$-simplex $\iota_n$.
\end{itemize}
\end{Def}  

For instance, $T \Box^0 = \Delta^0$ and $T \Box^1 = \Delta^1$; similarly, $T \widetilde{\Box}^1 = \widetilde{\Delta}^1$.
The simplicial set $T \Box^2$ can be depicted as
	
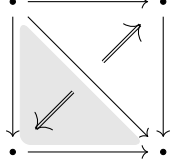
\[
	\begin{tikzpicture}
		\filldraw
		(0,0) circle [radius = 1pt]
		(2,0) circle [radius = 1pt]
		(0,2) circle [radius = 1pt]
		(2,2) circle [radius = 1pt];
		
		\draw[->] (0.2,0) -- (1.8,0);
		\draw[->] (0.2,2) -- (1.8,2);
		\draw[->] (0,1.8) -- (0,0.2);
		\draw[->] (2,1.8) -- (2,0.2);
		\draw[->] (0.2,1.8) -- (1.8,0.2);
		
		\filldraw[shaded, rounded corners]
		(0.1,0.1) -- (1.75,0.1) -- (0.1,1.75) -- cycle;
		
		\draw[->, double] (1.2,1.2) -- (1.7,1.7);
		\draw[->, double = shaded] (0.8,0.8) -- (0.3,0.3);
	\end{tikzpicture}
	\]
In $T \widetilde{\Box}^2$, the other non-degenerate 2-simplex would also be marked.
	
The triangulation of $\Box^3$ has six non-degenerate $3$-simplices, corresponding to the permutations of the set $\{1, 2, 3\}$, five of which are marked, with the unmarked simplex corresponding to the identity permutation.
Furthermore, there are six interior $2$-simplices, four of which are unmarked: $112$, $121$, $122$, and $212$; and two of which are marked: $211$ and $221$.
The boundary of the triangulation of $\Box^3$ consists of six copies of $T \Box^2$, and hence six marked and six unmarked $2$-simplices.

In $T \widetilde{\Box}^3$, we would additionally mark the $3$-simplex corresponding to the identity permutation, but the marking of the $2$-simplices would remain unchanged.
  
By left Kan extension, this definition extends to a colimit-preserving functor $T \colon \Set^{(\Box^+)^\op} \to \sSet^{+}$, with a right adjoint  $U \colon \sSet^{+} \to \Set^{(\Box^+)^\op}$. Once again, it is clear that these functors restrict to an adjunction $T : \cSet^{+} \rightleftarrows \sSet^{+} : U$. From the definition, we can see that the only unmarked $n$-simplex of $T\Box^n$ is $\iota_n$, while all $n$-simplices of $T\widetilde{\Box}^n$ are marked.

The following results will be of use in comparing the model structures of \cref{complicial-model-structure} with model structures on $\cSet^+$ via the adjunction $T \dashv U$; the first two describe useful properties of the marked triangulation functor, while the third straightforwardly extends \cref{op-monoidal} to the marked setting.

\begin{prop}[{\cite[Prop.~5.8]{campion-kapulkin-maehara}}]\label{T-op}
For $X \in \cSet^+$, there is a natural isomorphism $(TX)^\op \cong T(X^\op)$. \qed
\end{prop}

\begin{prop}\label{T-monoidal}
Let $i$ and $j$ be cofibrations in $\cSet^+$, and let $\sSet^+$ be equipped with any of the model structures of \cref{complicial-model-structure}. If either $Ti$ or $Tj$ is a trivial cofibration, then $T$ sends the pushout lax Gray tensor product $i \hat{\otimes} j$ to a trivial cofibration as well.
\end{prop}

\begin{proof}
This follows from \cref{gray-homotopical} together with \cite[Thm.~6.5]{campion-kapulkin-maehara}.
\end{proof}

\begin{prop}[{\cite[Prop.~1.17]{campion-kapulkin-maehara}}]\label{op-monoidal-marked}
  The functors $(-)^\co , (-)^\op \colon \cSet^+ \to \cSet^+$ are strong anti-monoidal, while $(-)^\coop  \colon \cSet \to \cSet$ is strong monoidal. \qed
\end{prop}

Just as we can consider alternate versions of cubical sets having only one kind of connection or none, we may likewise consider marked cubical sets having similarly restricted structure maps. Specifically, we have the following wide subcategories of $\Box^+$:

\begin{itemize}
\item $\Box^+_{\varnothing}$, generated by the maps $\bd_{i,\varepsilon}, \sigma_i$, and $\zeta_i$;
\item $\Box^+_0$, generated by $\bd_{i,\varepsilon}, \sigma_i, \gamma_{i,0}, \zeta_i$, and $\xi_{i,0}$;
\item $\Box^+_1$, generated by $\bd_{i,\varepsilon}, \sigma_i, \gamma_{i,1}, \zeta_i$, and $\xi_{i,1}$.
\end{itemize}

As with $\Box^+$, presheaves on each of these categories may be seen as presheaves on the corresponding cube category equipped with markings on some cubes, including all degenerate cubes. By restricting our attention to those presheaves for which each cube has at most one marking, we obtain alternative categories of marked cubical sets, denoted $\cSet^+_{\varnothing}, \cSet^+_0,$, and $\cSet^+_1$. As in \cref{cSet-unmarked}, all of the definitions and results of this section generalize to these categories as well (note that none of the proofs provided in \cite{campion-kapulkin-maehara} for the cited results from that paper use connections in any way). Once again, the only complication which arises is that $(-)^{\coop}$ and $(-)^{\op}$ do not define automorphisms of $\cSet_{0}$ and $\cSet_{1}$, but instead each defines an inverse pair of isomorphisms between these categories. Once again, our results about these functors, such as \cref{op-monoidal-marked}, remain true when interpreted as referring to these isomorphisms.

In subsequent sections, similarly to the approach taken in \cite{doherty-kapulkin-lindsey-sattler}, we will work in $\cSet^+$ for the sake of concreteness, but will write our proofs in such a way as to be generalizable to as many of the other marked cubical set categories mentioned as possible, either verbatim or using dual arguments. The construction of the comical model structures in \cref{sec:model} and proof that the triangulation functor is left Quillen in \cref{Quillen-functor} can both be done for any of our marked cubical set categories, while the proof that triangulation is a Quillen equivalence, covered in \cref{sec:cones,sec:equivalence}, requires at least one connection. At the beginning of each section, we will comment in further detail on which results can be adapted to which marked cubical set category.

\section{Model structure for comical sets}\label{sec:model}

In this section, we introduce the notion of a comical set and construct a model structure on $\mcSet$ whose fibrant-cofibrant objects are precisely the comical sets.
As indicated in the introduction, our definition differs slightly from the corresponding definition given in \cite{campion-kapulkin-maehara}. None of our definitions or proofs will require connections, and thus the results of this section are applicable in any of the categories $\cSet^+_{\varnothing}, \cSet^+_{0}, \cSet^+_{1}, \cSet^+$.

We begin by defining the (co)domains of our tentative anodyne maps.

\begin{Def} \leavevmode
\begin{enumerate}
\item For $n \geq 1$, $\varepsilon \in \{0,1\}$, and $1 \leq i \leq n$, the \emph{$(i,\varepsilon)$-comical cube} in dimension $n$, denoted $\Box^n_{i,\varepsilon}$, is the marked cubical set with underlying cubical set $\Box^n$ in which a non-degenerate $m$-cube $\delta \colon \Box^m \to \Box^n$ is unmarked if and only if at least one of the following three conditions holds:
\begin{enumerate}
\item\label{comical-middle} the standard form of $\delta$ contains $\bd_{i,\varepsilon}$ or $\bd_{i,1-\varepsilon}$;
\item\label{comical-high} for some $j > i$, the standard form of $\delta$ contains $\bd_{j,\varepsilon}$, as well as $\bd_{k, 1-\varepsilon}$ for all $j > k > i$;
\item\label{comical-low} for some $j < i$, the standard form of $\delta$ contains $\bd_{j,\varepsilon}$, as well as $\bd_{k, 1-\varepsilon}$ for all $j < k < i$;
\end{enumerate}

\item For $n \geq 1$, $\varepsilon \in \{0,1\}$, and $1 \leq i \leq n$, an $n$-cube $\Cube^n \to X$ of a marked cubical set $X$ is \defterm{$(i, \varepsilon)$-comical} if it factors through $\Box^n \to \Box^n_{i,\varepsilon}$. 

\item The \emph{$(i,\varepsilon)$-comical open box} in dimension $n$, denoted $\sqcap^n_{i,\varepsilon}$, is the regular subcomplex of $\Box^n_{i,\varepsilon}$ whose underlying cubical set is the $n$-dimensional $(i,\varepsilon)$-open box. The marked cubical set $(\Box^n_{i,\varepsilon})'$ is obtained from $\Box^n_{i,\varepsilon}$ by marking all $(n-1)$-cubes except for $\bd_{i,\varepsilon}$.

\item An $(i, \varepsilon)$-open box $(\sqcap^n_{i,\varepsilon})^\flat \to X$ in a marked cubical set $X$ is \defterm{$(i, \varepsilon)$-comical}, or just \defterm{comical}, if it factors through $(\sqcap^n_{i,\varepsilon})^\flat \to \sqcap^n_{i,\varepsilon}$.
\end{enumerate}
\end{Def}

Note that we can rephrase condition (1) above as follows. A non-degenerate cube of $\Box^n_{i,\varepsilon}$ is marked if and only if its standard form does not contain any ``excluded strings'' of face maps: namely $\bd_{i,\varepsilon}$; $\bd_{i,1-\varepsilon}$; $\bd_{i-1,\varepsilon}$; $\bd_{i+1,\varepsilon}$; any string of the form $\bd_{j,\varepsilon} \bd_{j-1,1-\varepsilon} \ldots \bd_{i+1,1-\varepsilon}$ for $j > i + 1$; or any string of the form $\bd_{i-1,1-\varepsilon} \ldots \bd_{j+1,1-\varepsilon} \bd_{j,\varepsilon}$ for $j < i - 1$.

For example, both the $(1,0)$- and $(1,1)$-comical cubes are the marked $1$-cube $\mcube 1$.
We depict below the comical cubes $\Box^n_{i,1}$ for $n = 2,3$ and $1 \le i \le n$ (recall the drawing convention described in \cref{cSet-unmarked}); the non-degenerate marked cubes in these comical cubes are listed in \cref{comical-table}.
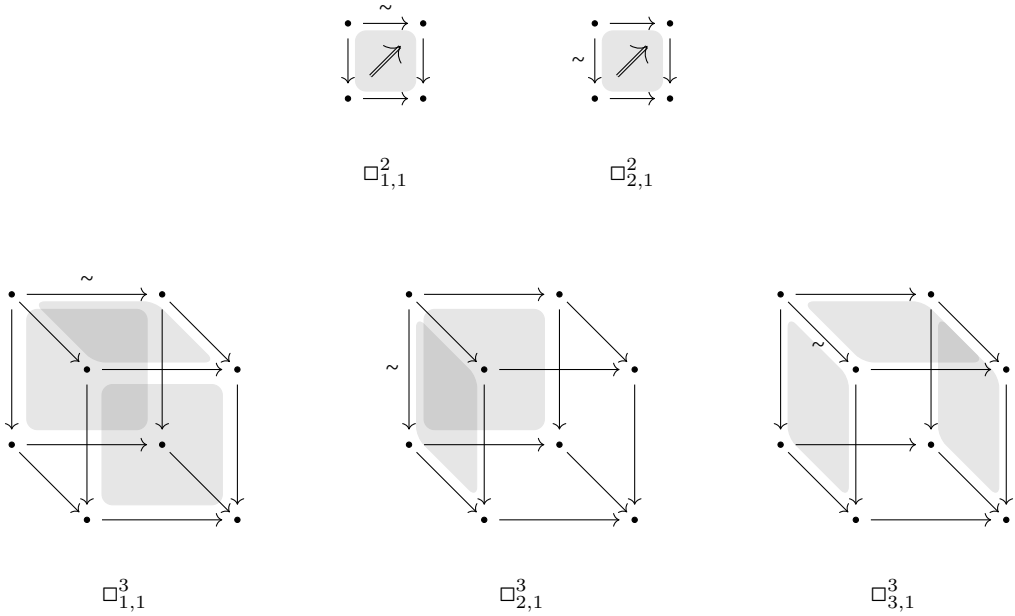
\[
\begin{tikzpicture}
\filldraw
(0,0) circle [radius = 1pt]
(0,1) circle [radius = 1pt]
(1,0) circle [radius = 1pt]
(1,1) circle [radius = 1pt];

\draw[->] (0.2,0) -- (0.8,0);
\draw[->] (0.2,1) -- (0.8,1) node [midway, above] {$\sim$};
\draw[->] (0,0.8) -- (0,0.2);
\draw[->] (1,0.8) -- (1,0.2);

\filldraw[shaded, rounded corners]
(0.1,0.1) -- (0.9,0.1) -- (0.9,0.9) -- (0.1,0.9) -- cycle;

\draw[->, double = shaded] (0.3,0.3) -- (0.7,0.7);

\node at (0.5,-1) {$\Box^2_{1,1}$};
\end{tikzpicture}
\hspace{5em}
\begin{tikzpicture}
	\filldraw
	(0,0) circle [radius = 1pt]
	(0,1) circle [radius = 1pt]
	(1,0) circle [radius = 1pt]
	(1,1) circle [radius = 1pt];
	
	\draw[->] (0.2,0) -- (0.8,0);
	\draw[->] (0.2,1) -- (0.8,1);
	\draw[->] (0,0.8) -- (0,0.2) node [midway, left] {$\sim$};
	\draw[->] (1,0.8) -- (1,0.2);
	
	\filldraw[shaded, rounded corners]
	(0.1,0.1) -- (0.9,0.1) -- (0.9,0.9) -- (0.1,0.9) -- cycle;
	
	\draw[->, double = shaded] (0.3,0.3) -- (0.7,0.7);
	
	\node at (0.5,-1) {$\Box^2_{2,1}$};
\end{tikzpicture}
\]

\[
\begin{tikzpicture}
\begin{scope}[blend group = multiply]
\filldraw[shaded, rounded corners] (0.3,2.9) -- (1.9,2.9) -- (2.7,2.1) -- (1.1,2.1) -- cycle; 
\filldraw[shaded, rounded corners] (0.2,1.2) -- (1.8,1.2) -- (1.8,2.8) -- (0.2,2.8) -- cycle; 
\filldraw[shaded, rounded corners] (1.2,0.2) -- (2.8,0.2) -- (2.8,1.8) -- (1.2,1.8) -- cycle; 
\end{scope}
	
\filldraw
(1,0) circle [radius = 1pt]	
(3,0) circle [radius = 1pt]	
(1,2) circle [radius = 1pt]	
(3,2) circle [radius = 1pt]	
(0,1) circle [radius = 1pt]	
(2,1) circle [radius = 1pt]	
(0,3) circle [radius = 1pt]	
(2,3) circle [radius = 1pt];

\draw[->] (1.2,0) -- (2.8,0);
\draw[->] (1.2,2) -- (2.8,2);
\draw[->] (0.2,1) -- (1.8,1);
\draw[->] (0.2,3) -- (1.8,3) node [midway, above] {$\sim$};

\draw[->] (0,2.8) -- (0,1.2);
\draw[->] (2,2.8) -- (2,1.2);
\draw[->] (1,1.8) -- (1,0.2);
\draw[->] (3,1.8) -- (3,0.2);

\draw[->] (0.1,0.9) -- (0.9,0.1);
\draw[->] (2.1,0.9) -- (2.9,0.1);
\draw[->] (0.1,2.9) -- (0.9,2.1);
\draw[->] (2.1,2.9) -- (2.9,2.1);

\node at (1.5,-1) {$\Box^3_{1,1}$};
\end{tikzpicture}
\hspace{5em}
\begin{tikzpicture}
	\begin{scope}[blend group = multiply]
		\filldraw[shaded, rounded corners] (0.2,1.2) -- (1.8,1.2) -- (1.8,2.8) -- (0.2,2.8) -- cycle; 
		\filldraw[shaded, rounded corners] (0.1,1.1) -- (0.9,0.3) -- (0.9,1.9) -- (0.1,2.7) -- cycle; 
	\end{scope}
	
	\filldraw
	(1,0) circle [radius = 1pt]	
	(3,0) circle [radius = 1pt]	
	(1,2) circle [radius = 1pt]	
	(3,2) circle [radius = 1pt]	
	(0,1) circle [radius = 1pt]	
	(2,1) circle [radius = 1pt]	
	(0,3) circle [radius = 1pt]	
	(2,3) circle [radius = 1pt];
	
	\draw[->] (1.2,0) -- (2.8,0);
	\draw[->] (1.2,2) -- (2.8,2);
	\draw[->] (0.2,1) -- (1.8,1);
	\draw[->] (0.2,3) -- (1.8,3);
	
	\draw[->] (0,2.8) -- (0,1.2) node [midway, left] {$\sim$};
	\draw[->] (2,2.8) -- (2,1.2);
	\draw[->] (1,1.8) -- (1,0.2);
	\draw[->] (3,1.8) -- (3,0.2);
	
	\draw[->] (0.1,0.9) -- (0.9,0.1);
	\draw[->] (2.1,0.9) -- (2.9,0.1);
	\draw[->] (0.1,2.9) -- (0.9,2.1);
	\draw[->] (2.1,2.9) -- (2.9,2.1);
	
	\node at (1.5,-1) {$\Box^3_{2,1}$};
\end{tikzpicture}
\hspace{5em}
\begin{tikzpicture}
	\begin{scope}[blend group = multiply]
		\filldraw[shaded, rounded corners] (0.3,2.9) -- (1.9,2.9) -- (2.7,2.1) -- (1.1,2.1) -- cycle; 
		\filldraw[shaded, rounded corners] (0.1,1.1) -- (0.9,0.3) -- (0.9,1.9) -- (0.1,2.7) -- cycle; 
		\filldraw[shaded, rounded corners] (2.1,1.1) -- (2.9,0.3) -- (2.9,1.9) -- (2.1,2.7) -- cycle; 
	\end{scope}
	
	\filldraw
	(1,0) circle [radius = 1pt]	
	(3,0) circle [radius = 1pt]	
	(1,2) circle [radius = 1pt]	
	(3,2) circle [radius = 1pt]	
	(0,1) circle [radius = 1pt]	
	(2,1) circle [radius = 1pt]	
	(0,3) circle [radius = 1pt]	
	(2,3) circle [radius = 1pt];
	
	\draw[->] (1.2,0) -- (2.8,0);
	\draw[->] (1.2,2) -- (2.8,2);
	\draw[->] (0.2,1) -- (1.8,1);
	\draw[->] (0.2,3) -- (1.8,3);
	
	\draw[->] (0,2.8) -- (0,1.2);
	\draw[->] (2,2.8) -- (2,1.2);
	\draw[->] (1,1.8) -- (1,0.2);
	\draw[->] (3,1.8) -- (3,0.2);
	
	\draw[->] (0.1,0.9) -- (0.9,0.1);
	\draw[->] (2.1,0.9) -- (2.9,0.1);
	\draw[->] (0.1,2.9) -- (0.9,2.1) node [midway, below] {$\sim$};
	\draw[->] (2.1,2.9) -- (2.9,2.1);
	
	\node at (1.5,-1) {$\Box^3_{3,1}$};
\end{tikzpicture}
\]

\begin{table}
\begin{center}
	\begin{tikzpicture}
		\matrix (magic) [matrix of nodes,nodes={minimum width=3cm,minimum height=1cm,draw,very thin},draw,inner sep=0]
		{
			 & marked $3$-cubes & marked $2$-cubes  & marked $1$-cubes \\
			 $\Box^2_{1,1}$ & - & $\id$ & $\bd_{2,0}$ \\
			 $\Box^2_{2,1}$ & - & $\id$ & $\bd_{1,0}$ \\
			$\Box^3_{1,1}$ & $\id$ & $\bd_{2,0},~\bd_{3,0},~\bd_{3,1}$ & $\bd_{3,0}\bd_{2,0}$\\
			$\Box^3_{2,1}$ & $\id$ & $\bd_{1,0},~\bd_{3,0}$ & $\bd_{3,0}\bd_{1,0}$ \\
			$\Box^3_{3,1}$ & $\id$ & $\bd_{1,0},~\bd_{1,1},~\bd_{2,0}$ &  $\bd_{2,0}\bd_{1,0}$ \\
		};
	\end{tikzpicture}
\caption{Marked cubes in $\Box^n_{i,1}$}
\label{comical-table}
\end{center}
\end{table}

\begin{Def} \leavevmode
\begin{itemize}
\item The \emph{$(i,\varepsilon)$-comical open box inclusion} is the inclusion $\sqcap^n_{i,\varepsilon} \hookrightarrow \Box^n_{i,\varepsilon}$. 
\item For $n \geq 2$, the \emph{elementary $(i,\varepsilon)$-comical marking extension} is the entire map $(\Box^n_{i,\varepsilon})' \to \tau_{n-2} \Box^n_{i,\varepsilon}$.
		\item We say a map in $\mcSet$ is \defterm{comical} if it is in
		\[
		\cell\bigl(\{\sqcap^n_{i,\varepsilon} \hookrightarrow \Box^n_{i,\varepsilon} | \ n \ge 1,~1 \le i \le n,~\varepsilon \in \{0,1\}\} \cup \{(\Box^n_{i,\varepsilon})' \to \tau_{n-2} \Box^n_{i,\varepsilon} | \ n \ge 2,~1 \le i \le n,~\varepsilon \in \{0,1\}\}\bigr).
		\]
\end{itemize}
\end{Def}   
  
\begin{rmk}
Note that in general, this definition of comical cubes and open boxes does not coincide with that given in \cite{campion-kapulkin-maehara}, as that definition only includes conditions \ref{comical-high} and \ref{comical-low} in the cases $j = i + 1$ and $j = i -1$, respectively. 
This modification was necessary to align the comical model structure with the functor $Q$ of \cite{doherty-kapulkin-lindsey-sattler} and make the key technical part of our proof (\cref{counit-anodyne}) work.

While it is not currently known whether the model structures of \cref{comical-model-structure} coincide with those developed in \cite[Thms.~3.3 \& 3.6]{campion-kapulkin-maehara}, the comical open box inclusions and marking extensions of \cite{campion-kapulkin-maehara} are pushouts of those defined above.
As a result, we can apply closure results for anodyne maps from \cite{campion-kapulkin-maehara} in our setting, although it will take additional work to establish such closure results in our setting.
\end{rmk}  
  
We next define our cubical analogues of the Rezk maps; this definition is somewhat more involved than its simplicial counterpart.  
  
\begin{Def}
For $x \in \{1,2\}$, let $L_x$ denote the marked cubical set whose underlying cubical set is $\Box^2$, with non-degenerate marked cubes consisting of the 2-cube $\id_{[1]^2}$, together with the 1-dimensional faces whose parity matches that of $x$. In other words, $L_1$ is obtained from $\Box^2$ by marking $\id_{[1]^2}, \bd_{1,0}$ and $\bd_{2,1}$, while $L_2$ is similarly obtained by marking $\id_{[1]^2}, \bd_{1,1}$ and $\bd_{2,0}$.

For $x,y \in \{1,2\}$, we define $L_{x,y}$ by the following pushout in $\cSet^+$:
	\[
	\begin{tikzcd}
		\Box^1
		\arrow [r, "\bd_{3-y,0}"]
		\arrow [d, swap, "\bd_{x,1}"]
		\pushout &
		 L_y
		\arrow [d] \\
		L_x
		\arrow [r] &
		L_{x,y}
	\end{tikzcd}
	\]
In other words, $L_{x,y}$ is obtained by identifying the unmarked positive face of $L_x$ with the unmarked negative face of $L_y$.	
	

Let $L_{x,y}' = \tau_0 L_{x,y}$, i.e. the simplicial set obtained by marking the three unmarked edges of $L_{x,y}$. The \emph{$(x,y)$-elementary Rezk map} is the entire map $L_{x,y} \to L_{x,y}'$.

In general, a \emph{Rezk map} is any map of the form
\[
(\bd \Box^m \hookrightarrow \Box^m) \hat{\otimes} (L_{x,y} \to L_{x,y}') \hat{\otimes} (\bd \Box^n \hookrightarrow \Box^n)
\]
for $x, y \in \{1,2\}$, $m, n \geq 0$.
\end{Def}

 To better understand the definition of the Rezk maps, we illustrate the marked cubical sets $L_{1}$ and $L_{2}$, depicted below on the left and right, respectively:
	
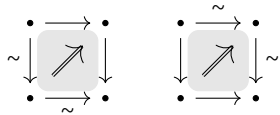
\[
	\begin{tikzpicture}
		\foreach \x in {0,1,2,3}
		\filldraw
		(\x,0) circle [radius = 1pt]
		(\x,1) circle [radius = 1pt];
		
		\foreach \x in {1,2}
		\draw[->] (\x,0.8) -- (\x,0.2);
		
		\draw[->] (2.2,0) -- (2.8,0);
		\draw[->] (0.2,1) -- (0.8,1);
		\draw[->] (0,0.8) -- (0,0.2) node[midway,left]{$\sim$};
		\draw[->] (3,0.8) -- (3,0.2) node[midway,right]{$\sim$};
		\draw[->] (0.2,0) -- (0.8,0) node[midway,below]{$\sim$};
		\draw[->] (2.2,1) -- (2.8,1) node[midway,above]{$\sim$};
		
		\foreach \x in {0,2}
		\filldraw[shaded, rounded corners]
		(\x+0.1,0.1) -- (\x+0.9,0.1) -- (\x+0.9,0.9) -- (\x+0.1,0.9) -- cycle;
		
		\draw[->, double = shaded] (0.3,0.3) -- (0.7,0.7);
		\draw[->, double = shaded] (2.3,0.3) -- (2.7,0.7);
		
	\end{tikzpicture}
	\]

To depict the full set of objects $L_{x.y}$ and $L_{x,y}'$, we must modify our drawing conventions slightly; in each of the marked cubical sets depicted below, both non-degenerate 2-cubes are marked, and the thick arrow inside each 2-cube points from its odd faces ($\bd_{1,0}$ and $\bd_{2,1}$) towards its even faces ($\bd_{1,1}$ and $\bd_{2,0})$.

\begin{align*}
L_{1,1} &= 
\left\{\begin{tikzpicture}[baseline = 12]
\filldraw[shaded, rounded corners]
(0.1,0.1) -- (0.1,0.9) -- (0.9,0.9) -- (0.9,0.1) -- cycle
(1.1,0.1) -- (1.1,0.9) -- (1.9,0.9) -- (1.9,0.1) -- cycle;
\foreach \x in {0,1,2}
\filldraw (\x,0) circle [radius = 1pt] (\x,1) circle [radius = 1pt];
\draw[->] (0.2,1) -- (0.8,1);
\draw[->] (1,0.8) -- (1,0.2);
\draw[->] (1.2,0) -- (1.8,0);
\draw[->] (0,0.8) -- (0,0.2) node[midway,left]{$\sim$};
\draw[->] (0.2,0) -- (0.8,0) node[midway,below]{$\sim$};
\draw[->] (1.2,1) -- (1.8,1) node[midway,above]{$\sim$};
\draw[->] (2,0.8) -- (2,0.2) node[midway,right]{$\sim$};
\draw[->, double = shaded] (0.3,0.3) -- (0.7,0.7);
\draw[->, double = shaded] (1.7,0.7) -- (1.3,0.3);
\end{tikzpicture}\right\} &
L_{1,1}' &= 
\left\{\begin{tikzpicture}[baseline = 12]
	\filldraw[shaded, rounded corners]
	(0.1,0.1) -- (0.1,0.9) -- (0.9,0.9) -- (0.9,0.1) -- cycle
	(1.1,0.1) -- (1.1,0.9) -- (1.9,0.9) -- (1.9,0.1) -- cycle;
\foreach \x in {0,1,2}
\filldraw (\x,0) circle [radius = 1pt] (\x,1) circle [radius = 1pt];
\draw[->] (0.2,1) -- (0.8,1) node[midway,above]{$\sim$};
\draw[->] (1,0.8) -- (1,0.2) node[midway]{$\sim$};
\draw[->] (1.2,0) -- (1.8,0) node[midway,below]{$\sim$};
\draw[->] (0,0.8) -- (0,0.2) node[midway,left]{$\sim$};
\draw[->] (0.2,0) -- (0.8,0) node[midway,below]{$\sim$};
\draw[->] (1.2,1) -- (1.8,1) node[midway,above]{$\sim$};
\draw[->] (2,0.8) -- (2,0.2) node[midway,right]{$\sim$};
\draw[->, double = shaded] (0.3,0.3) -- (0.7,0.7);
\draw[->, double = shaded] (1.7,0.7) -- (1.3,0.3);
\end{tikzpicture}\right\}\\
L_{1,2} &= 
\left\{\begin{tikzpicture}[baseline = 12]
	\filldraw[shaded, rounded corners]
	(0.1,0.1) -- (0.1,0.9) -- (0.9,0.9) -- (0.9,0.1) -- cycle
	(1.1,0.1) -- (1.1,0.9) -- (1.9,0.9) -- (1.9,0.1) -- cycle;
\foreach \x in {0,1,2}
\filldraw (\x,0) circle [radius = 1pt] (\x,1) circle [radius = 1pt];
\draw[->] (0.2,1) -- (0.8,1);
\draw[->] (1,0.8) -- (1,0.2);
\draw[->] (1.2,0) -- (1.8,0);
\draw[->] (0,0.8) -- (0,0.2) node[midway,left]{$\sim$};
\draw[->] (0.2,0) -- (0.8,0) node[midway,below]{$\sim$};
\draw[->] (1.2,1) -- (1.8,1) node[midway,above]{$\sim$};
\draw[->] (2,0.8) -- (2,0.2) node[midway,right]{$\sim$};
\draw[->, double = shaded] (0.3,0.3) -- (0.7,0.7);
\draw[->, double = shaded] (1.3,0.3) -- (1.7,0.7);
\end{tikzpicture}\right\} &
L_{1,2}' &= 
\left\{\begin{tikzpicture}[baseline = 12]
	\filldraw[shaded, rounded corners]
	(0.1,0.1) -- (0.1,0.9) -- (0.9,0.9) -- (0.9,0.1) -- cycle
	(1.1,0.1) -- (1.1,0.9) -- (1.9,0.9) -- (1.9,0.1) -- cycle;
\foreach \x in {0,1,2}
\filldraw (\x,0) circle [radius = 1pt] (\x,1) circle [radius = 1pt];
\draw[->] (0.2,1) -- (0.8,1) node[midway,above]{$\sim$};
\draw[->] (1,0.8) -- (1,0.2) node[midway]{$\sim$};
\draw[->] (1.2,0) -- (1.8,0) node[midway,below]{$\sim$};
\draw[->] (0,0.8) -- (0,0.2) node[midway,left]{$\sim$};
\draw[->] (0.2,0) -- (0.8,0) node[midway,below]{$\sim$};
\draw[->] (1.2,1) -- (1.8,1) node[midway,above]{$\sim$};
\draw[->] (2,0.8) -- (2,0.2) node[midway,right]{$\sim$};
\draw[->, double = shaded] (0.3,0.3) -- (0.7,0.7);
\draw[->, double = shaded] (1.3,0.3) -- (1.7,0.7);
\end{tikzpicture}\right\}\\
L_{2,1} &= 
\left\{\begin{tikzpicture}[baseline = 12]
	\filldraw[shaded, rounded corners]
	(0.1,0.1) -- (0.1,0.9) -- (0.9,0.9) -- (0.9,0.1) -- cycle
	(1.1,0.1) -- (1.1,0.9) -- (1.9,0.9) -- (1.9,0.1) -- cycle;
\foreach \x in {0,1,2}
\filldraw (\x,0) circle [radius = 1pt] (\x,1) circle [radius = 1pt];
\draw[->] (0.2,1) -- (0.8,1);
\draw[->] (1,0.8) -- (1,0.2);
\draw[->] (1.2,0) -- (1.8,0);
\draw[->] (0,0.8) -- (0,0.2) node[midway,left]{$\sim$};
\draw[->] (0.2,0) -- (0.8,0) node[midway,below]{$\sim$};
\draw[->] (1.2,1) -- (1.8,1) node[midway,above]{$\sim$};
\draw[->] (2,0.8) -- (2,0.2) node[midway,right]{$\sim$};
\draw[->, double = shaded] (0.7,0.7) -- (0.3,0.3);
\draw[->, double = shaded] (1.7,0.7) -- (1.3,0.3);
\end{tikzpicture}\right\} &
L_{2,1}' &= 
\left\{\begin{tikzpicture}[baseline = 12]
	\filldraw[shaded, rounded corners]
	(0.1,0.1) -- (0.1,0.9) -- (0.9,0.9) -- (0.9,0.1) -- cycle
	(1.1,0.1) -- (1.1,0.9) -- (1.9,0.9) -- (1.9,0.1) -- cycle;
\foreach \x in {0,1,2}
\filldraw (\x,0) circle [radius = 1pt] (\x,1) circle [radius = 1pt];
\draw[->] (0.2,1) -- (0.8,1) node[midway,above]{$\sim$};
\draw[->] (1,0.8) -- (1,0.2) node[midway]{$\sim$};
\draw[->] (1.2,0) -- (1.8,0) node[midway,below]{$\sim$};
\draw[->] (0,0.8) -- (0,0.2) node[midway,left]{$\sim$};
\draw[->] (0.2,0) -- (0.8,0) node[midway,below]{$\sim$};
\draw[->] (1.2,1) -- (1.8,1) node[midway,above]{$\sim$};
\draw[->] (2,0.8) -- (2,0.2) node[midway,right]{$\sim$};
\draw[->, double = shaded] (0.7,0.7) -- (0.3,0.3);
\draw[->, double = shaded] (1.7,0.7) -- (1.3,0.3);
\end{tikzpicture}\right\}\\
L_{2,2} &= 
\left\{\begin{tikzpicture}[baseline = 12]
	\filldraw[shaded, rounded corners]
	(0.1,0.1) -- (0.1,0.9) -- (0.9,0.9) -- (0.9,0.1) -- cycle
	(1.1,0.1) -- (1.1,0.9) -- (1.9,0.9) -- (1.9,0.1) -- cycle;
\foreach \x in {0,1,2}
\filldraw (\x,0) circle [radius = 1pt] (\x,1) circle [radius = 1pt];
\draw[->] (0.2,1) -- (0.8,1);
\draw[->] (1,0.8) -- (1,0.2);
\draw[->] (1.2,0) -- (1.8,0);
\draw[->] (0,0.8) -- (0,0.2) node[midway,left]{$\sim$};
\draw[->] (0.2,0) -- (0.8,0) node[midway,below]{$\sim$};
\draw[->] (1.2,1) -- (1.8,1) node[midway,above]{$\sim$};
\draw[->] (2,0.8) -- (2,0.2) node[midway,right]{$\sim$};
\draw[->, double = shaded] (0.7,0.7) -- (0.3,0.3);
\draw[->, double = shaded] (1.3,0.3) -- (1.7,0.7);
\end{tikzpicture}\right\} &
L_{2,2}' &= 
\left\{\begin{tikzpicture}[baseline = 12]
	\filldraw[shaded, rounded corners]
	(0.1,0.1) -- (0.1,0.9) -- (0.9,0.9) -- (0.9,0.1) -- cycle
	(1.1,0.1) -- (1.1,0.9) -- (1.9,0.9) -- (1.9,0.1) -- cycle;
\foreach \x in {0,1,2}
\filldraw (\x,0) circle [radius = 1pt] (\x,1) circle [radius = 1pt];
\draw[->] (0.2,1) -- (0.8,1) node[midway,above]{$\sim$};
\draw[->] (1,0.8) -- (1,0.2) node[midway]{$\sim$};
\draw[->] (1.2,0) -- (1.8,0) node[midway,below]{$\sim$};
\draw[->] (0,0.8) -- (0,0.2) node[midway,left]{$\sim$};
\draw[->] (0.2,0) -- (0.8,0) node[midway,below]{$\sim$};
\draw[->] (1.2,1) -- (1.8,1) node[midway,above]{$\sim$};
\draw[->] (2,0.8) -- (2,0.2) node[midway,right]{$\sim$};
\draw[->, double = shaded] (0.7,0.7) -- (0.3,0.3);
\draw[->, double = shaded] (1.3,0.3) -- (1.7,0.7);
\end{tikzpicture}\right\}
\end{align*}

As in the simplicial case, the Rezk maps capture the principle that a cube representing an invertible higher morphism should be marked. Each of the objects $L_{x}$ represents the composition of a pair of edges to obtain an equivalence, i.e. the appropriate notion in this setting of a section-retraction pair. The objects $L_{x,y}$ are then obtained by gluing together a pair of these objects along an unmarked edge, identifying the ``retraction'' in $L_x$ with the ``section'' in $L_y$. Thus the elementary Rezk maps represent the marking of invertible 1-morphisms; the more general Rezk maps extend this intuition to higher dimension.
We can now define comical sets, which will be the fibrant objects of our model structure.

\begin{Def}
A \emph{comical set} is a marked cubical set having the right lifting property with respect to all comical open-box fillings and elementary comical marking extensions. 
\end{Def}

\begin{Def}
A comical set is:
\begin{itemize}
\item \emph{saturated} if it has the right lifting property with respect to all Rezk maps;
\item \emph{$n$-trivial}, for $n \geq 0$, if it has the right lifting property with respect to all markings $\Box^m \to \widetilde{\Box}^m$ for $m > n$ (in other words, if all of its cubes of dimension greater than $n$ are marked).
\end{itemize}
\end{Def}
  
We are now ready to construct the desired model structures on $\cSet^+$.

\begin{thm}\label{comical-model-structure}
The category $\mcSet$ carries the following model structures:

\begin{enumerate}
  \item \label{model-struct-comical} A \emph{model structure for comical sets} in which
    \begin{itemize}
      \item cofibrations are monomorphisms;
      \item fibrant objects are comical sets;
      \item fibrations with fibrant codomain are characterized by the right lifting property with respect to comical
      open box inclusions and comical marking extensions.
    \end{itemize}
    
    \item \label{model-struct-comical-sat} A \emph{model structure for saturated comical sets} in which
    \begin{itemize}
      \item cofibrations are monomorphisms;
      \item fibrant objects are saturated comical sets;
      \item fibrations with fibrant codomain are characterized by the right lifting property with respect to comical
      open box inclusions, comical marking extensions, and the Rezk maps.
    \end{itemize}
    
      \item \label{model-struct-comical-triv} A \emph{model structure for $n$-trivial comical sets} for $n \geq 0$ in which
    \begin{itemize}
      \item cofibrations are monomorphisms;
      \item fibrant objects are $n$-trivial comical sets;
      \item fibrations with fibrant codomain are characterized by the right lifting property with respect to comical
      open box inclusions, comical marking extensions, and markings $\Box^m \to \widetilde{\Box}^m$ for $m > n$.
    \end{itemize}
    
      \item \label{model-struct-comical-sat-triv} A \emph{model structure for $n$-trivial saturated comical sets} in which
    \begin{itemize}
      \item cofibrations are monomorphisms;
      \item fibrant objects are $n$-trivial saturated comical sets;
      \item fibrations with fibrant codomain are characterized by the right lifting property with respect to comical
      open box inclusions, comical marking extensions, Rezk maps, and markings $\Box^m \to \widetilde{\Box}^m$ for $m > n$.
    \end{itemize}
\end{enumerate}
All of these model structures are monoidal with respect to the lax Gray tensor product.
\end{thm}

\begin{proof}
In all four cases, we will apply the Cisinski--Olschok theory (cf.~\cref{CO-with-monoidal}) with the set $I$ of \cref{cset-cellular-model} with the cylinder functor given by $\widetilde{\Box}^1 \otimes -$ and the natural transformations $\bd^0, \bd^1$ induced by face inclusions $\varphi \partial_{1,0}, \varphi \partial_{1, 1} \colon \Box_0 \to  \widetilde{\Box}^1$. 
In each case, the generating set $S$ of anodyne maps is chosen differently.

We check that for any maps $f \in S$ and $g \in I$, the pushout product $f \hat{\otimes} g$ is again in the saturation of $S$.
The case when $f \in I$ and $g \in S$ is analogous.
Thus we need to address the following eight cases:
\begin{center}
  \begin{tikzpicture}
  \matrix (magic) [matrix of nodes,nodes={minimum width=3cm,minimum height=1cm,draw,very thin},draw,inner sep=0]
  {
   $f \hat{\otimes} g$  & $\partial \Box^n \to \Box^n$ & $\Box^n \to \widetilde{\Box}^n$  \\
  $\sqcap^m_{i,\varepsilon} \hookrightarrow \Box^n_{i,\varepsilon}$ & 1 & 2 \\
  $(\Box^m_{i,\varepsilon})' \to \tau_{n-2} \Box^n_{i,\varepsilon}$ & 3 & 4  \\
    $- \hat{\otimes} (L_{x,y} \to L_{x,y}') \hat{\otimes} -$ & 5 & 6 \\
  $\Box^k \to \widetilde{\Box}^k, k > m$ & 7 & 8 \\
  };
  \end{tikzpicture}
\end{center}

For case 5, we may note that for any $m, n \geq 0$, the pushout product of boundary inclusions $(\bd \Box^m \hookrightarrow \Box^m) \hat{\otimes} (\bd \Box^n \hookrightarrow \Box^n)$ is isomorphic to the boundary inclusion $\bd \Box^{m+n} \hookrightarrow \Box^{m+n}$. By associativity of the pushout product, it thus follows that the pushout product of a Rezk map and a boundary inclusion is again a Rezk map.

Cases 4, 6, and 8 are pushout products of two entire maps, and hence isomorphisms by \cref{Gray-tensor-of-monos}.
Case 7 is clear since it is an entire map with no markings added in dimension $m$ or below, and hence a pushout of markers in dimension above $m$.

For case 1, we consider the pushout product $(\sqcap^m_{i,\varepsilon} \hookrightarrow \Box^m_{i,\varepsilon}) \hat{\otimes} (\partial \Box^n \to \Box^n)$, which is regular by \cref{Gray-tensor-of-monos} and an open box inclusion on the underlying cubical sets.
It therefore suffices to show that $\Box^m_{i,\varepsilon} \otimes \Box^n$ is a pushout of $\sqcap^{m+n}_{i,\varepsilon} \hookrightarrow \Box^{m+n}_{i,\varepsilon}$.
For that, we show that all of the marked faces of $\sqcap^{m+n}_{i,\varepsilon}$ are marked in $\Box^m_{i,\varepsilon} \otimes \Box^n$.
To see this, we consider the standard form of one of the faces of $\Box^m_{i,\varepsilon} \otimes \Box^n$, say given by $\partial_{a_1, \varepsilon_1} \ldots \partial_{a_p, \varepsilon_p} \partial_{b_1, \varepsilon'_1} \ldots \partial_{b_q, \varepsilon'_q}$, with $a_p \geq m+1$ and $b_q \leq m$.
By the characterization of cubes in the geometric product, we obtain that this corresponds to a pair
\[ (\partial_{b_1, \varepsilon'_1} \ldots \partial_{b_q, \varepsilon'_q}, \partial_{a_1 - m, \varepsilon_1} \ldots \partial_{a_p - m,  \varepsilon_p}) \in (\Box^m_{i,\varepsilon})_{m-p} \times (\Box^n)_{n-q}\text{.}   \]
If the standard form $\partial_{a_1, \varepsilon_1} \ldots \partial_{a_p, \varepsilon_p} \partial_{b_1, \varepsilon'_1} \ldots \partial_{b_q, \varepsilon'_q}$ does not include any of the strings excluded by the definition of a comical cube, then neither does its terminal segment $\partial_{b_1, \varepsilon'_1} \ldots \partial_{b_q, \varepsilon'_q}$.
Hence this terminal segment corresponds to a marked cube of $\Box_{i, \varepsilon}$, and thus the cube corresponding to the pair above is marked in $\Box^m_{i,\varepsilon} \otimes \Box^n$.

For case 2, we consider the pushout product $(\sqcap^m_{i,\varepsilon} \hookrightarrow \Box^m_{i,\varepsilon}) \hat{\otimes} (\Box^n \to \widetilde{\Box}^n)$, which is entire by \cref{Gray-tensor-of-monos}.
By definition, this is $ \Box^m_{i,\varepsilon} \otimes \Box^n \cup \sqcap^m_{i,\varepsilon} \otimes \widetilde{\Box}^n \to \Box^m_{i,\varepsilon} \otimes \widetilde{\Box}^n$.
A face of $\Box^m_{i,\varepsilon} \otimes \widetilde{\Box}^n$ is marked either if its standard form does not contain any strings excluded by the definition of a comical $m$-cube or it does not contain any face maps with indices greater than $m$.
Cubes satisfying the first condition are marked in the domain as well, but of the cubes satisfying the second condition, the face $\partial_{i, \varepsilon}$ is unmarked.
Thus this map is a pushout of the comical marking extension $(\Box^{m+n}_{i,\varepsilon})' \to \tau_{m+n-2} \Box^{m+n}_{i,\varepsilon}$.

Finally,  for case 3, we consider the pushout product $((\Box^m_{i,\varepsilon})' \to \tau_{m-2} \Box^m_{i,\varepsilon}) \hat{\otimes} (\partial \Box^n \to \Box^n)$, which is once again entire by \cref{Gray-tensor-of-monos}.
By definition, this is the map $(\Box^m_{i,\varepsilon})' \otimes \Box^n \cup \tau_{m-2} \Box^m_{i,\varepsilon} \otimes \partial \Box^n \to \tau_{m-2} \Box^m_{i,\varepsilon} \otimes \Box_n$.
A face is marked in the codomain if its standard form does not contain any of the strings excluded by the definition of an  $(i,\varepsilon)$-comical cube or it contains at most one face map $\partial_{j, \mu}$ for $j \leq m$.
The only one of these maps to be unmarked in the domain is $\partial_{i, \varepsilon}$, and hence the desired map is a pushout of a comical marking extension.
\end{proof}

By construction,  the weak equivalences of these model structures can be characterized by inducing a bijection on sets of homotopy classes of maps, denoted $[-, -]$, where the notion of homotopy is induced by the cylinder $\widetilde{\Box}^1 \otimes -$.

\begin{cor}
The weak equivalences of the model structure for ($n$-trivial, saturated) comical sets are maps $X \to Y$ inducing bijections $[Y,Z] \to [X,Z]$ for all ($n$-trivial, saturated) comical sets $Z$. \qed
\end{cor}

We next consider a few basic lemmas which describe the interactions of the distinguished maps in the comical model structures with the canonical endofunctors on $\cSet^+$, and which will be useful in analyzing these model structures.

\begin{lem}\label{tau-comical}
For each $k \geq 0$, the functor $\tau_k \colon \mcSet \to \mcSet$ preserves comical maps.
\end{lem}

\begin{proof}
Each $\tau_k$ preserves colimits as a left adjoint, so it suffices to consider the cases of comical open box inclusions and comical marking extensions.

First, consider the image under $\tau_k$ of a comical open box inclusion, $\tau_k \sqcap^n_{i,\varepsilon} \hookrightarrow \tau_k \Box^n_{i,\varepsilon}$. This map adds two non-degenerate cubes to $\tau_k \sqcap^n_{i,\varepsilon}$: the interior $n$-cube $\id_{[1]^n}$, which is marked, and the composite face, the $(n-1)$-cube $\bd_{i,\varepsilon}$. If $k \geq n - 1$, then this map is isomorphic to $\sqcap^n_{i,\varepsilon} \hookrightarrow \Box^n_{i,\varepsilon}$, as neither the domain nor the codomain of an $n$-dimensional comical open box inclusion contain any unmarked non-degenerate cubes in dimension $n$ or higher. On the other hand, if $k < n - 1$, then the composite face is marked in $\tau_k \Box^n_{i,\varepsilon}$, despite being unmarked in $\Box^n_{i,\varepsilon}$. In this case, all other $(n-1)$-cubes of $\tau_k \Box^n_{i,\varepsilon}$ are marked as well. Thus we may write this map as a composite of a pushout of $\sqcap^n_{i,\varepsilon} \hookrightarrow \Box^n_{i,\varepsilon}$, which adds $\id_{[1]^n}$ (marked) and $\bd_{i,\varepsilon}$ (unmarked), with a pushout of the comical marking extension $(\Box^n_{i,\varepsilon})' \to \tau_{n-2} \Box^n_{i,\varepsilon}$, which marks the composite face.

Next we consider the image under $\tau_k$ of an elementary comical marking extension. Once again, if $k \geq n - 1$ then $\tau_k (\Box^n_{i,\varepsilon})' \to \tau_k \tau_{n-2} \Box^n_{i,\varepsilon}$ is isomorphic to the original map $(\Box^n_{i,\varepsilon})' \to \tau_{n-2} \Box^n_{i,\varepsilon}$. On the other hand, if $k < n - 1$ then $\tau_k (\Box^n_{i,\varepsilon})' \to \tau_k \tau_{n-2} \Box^n_{i,\varepsilon}$ is a pushout of $(\Box^n_{i,\varepsilon})' \to \tau_{n-2} \Box^n_{i,\varepsilon}$, again similarly to the previous case.
\end{proof}

\begin{lem}\label{involutions-comical}
The involutions $(-)^\co, (-)^\coop, (-)^\op \colon \cSet^+ \to \cSet^+$ preserve comical maps.
\end{lem}

\begin{proof}
It suffices to show that the image under any of these involutions of a comical open box inclusion or comical marking extension is comical; this follows immediately from the definitions of these classes of maps, from which we can see that each involution sends each such map to another of the same class.
\end{proof}

\begin{lem}\label{involutions-Rezk}
For any $x,y \in \{1,2\}$, we have the following isomorphisms of maps in $\mcSet$:

\begin{itemize}
\item $(L_{x,y} \to L_{x,y}')^\co \cong (L_{3-x,3-y} \to L_{3-x,3-y}')$;
\item $(L_{x,y} \to L_{x,y}')^\coop \cong (L_{3-y,3-x} \to L_{3-y,3-x}')$;
\item $(L_{x,y} \to L_{x,y}')^\op \cong (L_{y,x} \to L_{y,x}')$.
\end{itemize}
\end{lem}

\begin{proof}
It is immediate from the definitions that for $x \in \{1,2\}$, both $(-)^\co$ and $(-)^\coop$ send $L_x$ to $L_{3-x}$. The stated result then follows from considering the effects of these involutions on the pushout diagram defining $L_{x,y}$.
\end{proof}

\begin{prop}\label{cubical-op-equiv}
For each of the model structures of \cref{comical-model-structure}, the self-adjunctions arising from the involutions $(-)^\co , (-)^\coop, (-)^\op \colon \cSet^+ \to \cSet^+$ are Quillen self-equivalences.
\end{prop}

\begin{proof}
For each involution $F$, it suffices to show that $F$ is left Quillen, so that the adjunction $F \dashv F$ satisfies the conditions of \cref{involution-equiv}. For this, it suffices to show that they preserve the classes of comical open box inclusions, comical marking extensions, Rezk maps, and markers. For comical open box inclusions and comical marking extensions, this follows from \cref{involutions-comical}. For the elementary Rezk maps, this follows from \cref{involutions-Rezk}; the result for general Rezk maps then follows from \cref{op-monoidal}. For markers, it is immediate from the definition. 
\end{proof}

We conclude this section with some further discussion of the Rezk maps. One might reasonably wonder whether it is necessary to use all four objects $L_{x,y}$ in defining the Rezk maps, or whether a simpler presentation would suffice. In fact, the results which we will now prove demonstrate that we could indeed define the saturated comical model structures using only the Rezk maps of the form $(\bd \Box^m \hookrightarrow \Box^m) \hat{\otimes} (L_{x,y} \to L_{x,y}') \hat{\otimes} (\bd \Box^n \hookrightarrow \Box^n)$ for some fixed choice of $x$ and $y$. 

Our reason for defining the model structures as we did in the statement of \cref{comical-model-structure} is that any such choice would be arbitrary and inelegant; in particular, only the Rezk maps defined in terms of the chosen $L_{x,y}$ would be anodyne according to the standard definition of that term in Cisinski model structures, while the others would merely be trivial cofibrations. Nevertheless, it is useful to know that these simpler presentations exist, as they may be of use in identifying fibrations between fibrant objects, or in proving that functors out of these model categories are left Quillen.

We begin by introducing an equivalence relation on the maps in $\mcSet$ which will be of use in establishing our simpler presentation of the comical model structures.

\begin{Def}
Let $f \colon X \to Y, f' \colon X' \to Y'$ be maps in $\mcSet$. An \emph{elementary comical equivalence} from $f$ to $f'$ is a commuting diagram of the form
\[
\xymatrix{
X \ar[d]_{f} \ar[r] & X' \ar[d]^{f'} \\
Y \ar[r] & Y' \\
}
\]
in which the vertical maps $X \to X', Y \to Y'$ are comical. A pair of maps are \emph{comically equivalent} if they are connected by a zigzag of elementary comical equivalences.
\end{Def}

\begin{lem}\label{comical-2-of-3}
In any model structure on $\mcSet$ in which all comical maps are weak equivalences, if $f$ and $f'$ are comically equivalent maps, then $f$ is a weak equivalence if and only if $f'$ is a weak equivalence.
\end{lem}

\begin{proof}
This is immediate from the two-out-of-three property.
\end{proof}

\begin{lem}\label{comical-equiv-involutions}
If $f$ and $f'$ are comically equivalent maps in $\cSet^+$, then so are each of the following pairs of maps:
\begin{itemize}
\item $f^\co$ and $(f')^\co$;
\item $f^\coop$ and $(f')^\coop$;
\item $f^\op$ and $(f')^\op$.
\end{itemize}
\end{lem}

\begin{proof}
By \cref{involutions-comical}, it follows that the involutions $(-)^\co, (-)^\coop, (-)^\op$ preserve elementary comical equivalences, and hence preserve the relation of comical equivalence.
\end{proof}

\begin{lem}\label{Rezk-comic-equiv}
All elementary Rezk maps are comically equivalent.
\end{lem}

\begin{proof}
We first show that the maps $L_{1,1} \to L_{1,1}'$ and $L_{1,2} \to L_{1,2}'$ are comically equivalent. Let $M$ denote the marked cubical set depicted below, in which the interior 3-cube and all 2-cubes are marked. Note that the top and bottom faces of the cube are non-degenerate; indeed, their $(1,0)$- and $(2,0)$-faces are distinct, so that their boundaries do not agree with those of a connection, despite appearances. 
(Note that we could have alternatively used a variant of $M$ where these faces are connections on $p$ and $g$ respectively, giving a perhaps simpler construction.
We chose not to do so to avoid the use of connections in the construction and characterization of the model structure.)
The right and front faces, however, are degeneracies.
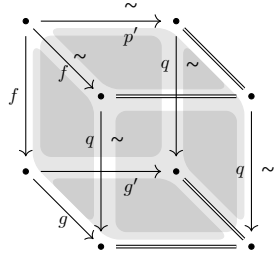
\[
\begin{tikzpicture}
	\begin{scope}[blend group = multiply]
		\filldraw[shaded, rounded corners] (0.3,2.9) -- (1.9,2.9) -- (2.7,2.1) -- (1.1,2.1) -- cycle; 
		\filldraw[shaded, rounded corners] (0.3,0.9) -- (1.9,0.9) -- (2.7,0.1) -- (1.1,0.1) -- cycle; 
		\filldraw[shaded, rounded corners] (0.2,1.2) -- (1.8,1.2) -- (1.8,2.8) -- (0.2,2.8) -- cycle; 
		\filldraw[shaded, rounded corners] (1.2,0.2) -- (2.8,0.2) -- (2.8,1.8) -- (1.2,1.8) -- cycle; 
		\filldraw[shaded, rounded corners] (0.1,1.1) -- (0.9,0.3) -- (0.9,1.9) -- (0.1,2.7) -- cycle; 
		\filldraw[shaded, rounded corners] (2.1,1.1) -- (2.9,0.3) -- (2.9,1.9) -- (2.1,2.7) -- cycle; 
	\end{scope}
	
	\filldraw
	(1,0) circle [radius = 1pt]	
	(3,0) circle [radius = 1pt]	
	(1,2) circle [radius = 1pt]	
	(3,2) circle [radius = 1pt]	
	(0,1) circle [radius = 1pt]	
	(2,1) circle [radius = 1pt]	
	(0,3) circle [radius = 1pt]	
	(2,3) circle [radius = 1pt];
	
	\draw[double] (1.2,0) -- (2.8,0);
	\draw[double] (1.2,2) -- (2.8,2);
	\draw[->] (0.2,1) -- (1.8,1) node [near end, below, scale = 0.7] {$g'$};
	\draw[->] (0.2,3) -- (1.8,3) node [near end, above] {$\sim$} node [near end, below, scale = 0.7] {$p'$};
	
	\draw[->] (0,2.8) -- (0,1.2) node [midway, left, scale = 0.7] {$f$};
	\draw[->] (2,2.8) -- (2,1.2) node [near start, right] {$\sim$} node [near start, left, scale = 0.7] {$q$};
	\draw[->] (1,1.8) -- (1,0.2) node [near start, right] {$\sim$} node [near start, left, scale = 0.7] {$q$};
	\draw[->] (3,1.8) -- (3,0.2) node [midway, right] {$\sim$} node [midway, left, scale = 0.7] {$q$};
	
	\draw[->] (0.1,0.9) -- (0.9,0.1) node [midway, below, scale = 0.7] {$g$};
	\draw[double] (2.1,0.9) -- (2.9,0.1);
	\draw[->] (0.1,2.9) -- (0.9,2.1) node [midway, right] {$\sim$} node [midway, below, scale = 0.7] {$f$};
	\draw[double] (2.1,2.9) -- (2.9,2.1);
\end{tikzpicture}
\]

We note that $L_{1}$ and $L_{2}$ embed into $M$ as the left and back faces of the cube, respectively. In fact, each of these inclusions is a comical map. To see that the inclusion of $L_1$ as the left face is comical, observe that we may construct $M$ from this face by the following sequence of steps.
\begin{itemize}
\item We begin with the left face; the front and right faces are defined to be degeneracies on the marked edge $q$.
\item We can construct the top face via a $(2,0)$-comical open box filling, obtaining $p'$ as the composite edge.
As the interior and the other three edges ($p$ and two degenerate edges) are marked, we may mark $p'$ via a comical marking extension. 
\item Similarly to the previous step, we may construct the bottom face, together with the edge $g'$, via $(2,0)$-comical open box filling.
\item We may then construct the interior of the 3-cube, together with the back face, via $(3,0)$-comical open box filling, as all 2-cubes and the edge $\bd_{2,1}\bd_{1,1}$ are marked. We may then mark the back face via a comical marking extension, again using the fact that all other 2-cubes are marked.
\end{itemize}

Thus we see that the inclusion $L_1 \hookrightarrow M$ is comical, as a composite of pushouts of comical open box inclusions and comical marking extensions. Now consider the following composite of pushout squares:
\[
\xymatrix{
\Box^1 \ar[r]^{\bd_{2,0}} \ar[d]_{\bd_{1,1}} & L_1 \ar[d] \ar[r]^{\bd_{1,0}} & M \ar[d] \\
L_1 \ar[r] & L_{1,1} \pushoutcorner \ar[r] & L_{1,1} \cup_{L_1} M \pushoutcorner \\
}
\]

The left square is the pushout square defining $L_{1,1}$. Examining the right square, we see that the inclusion of $L_{1,1}$ into the pushout object $L_{1,1} \cup_{L_1} M$ is comical, as a pushout of the comical map $L_1 \hookrightarrow M$. Furthermore, by composition of pushout squares, we may note that $L_{1,1} \cup_{L_1} M$ is the pushout object of $\bd_{1,1} \colon \Box^1 \to L_1$ and $\bd_{1,0} \bd_{2,0} = \bd_{3,0} \bd_{1,0} \colon \Box^1 \to M$. In other words, $L_{1,1} \cup_{L_1} M$ is obtained by identifying the unmarked edge $\bd_{1,1}$ of $L_1$ with the edge $f$ of $M$. 

A similar proof shows that $L_2 \hookrightarrow M$ is comical as well. We may then consider the following composite of pushout diagrams:

\[
\xymatrix{
\Box^1 \ar[r]^{\bd_{1,0}} \ar[d]_{\bd_{1,1}} & L_2 \ar[d] \ar[r]^{\bd_{3,0}} & M \ar[d] \\
L_1 \ar[r] & L_{1,2} \pushoutcorner \ar[r] & L_{1,2} \cup_{L_2} M \pushoutcorner \\
}
\]

As in the previous case, we see that $L_{1,2} \hookrightarrow L_{1,2} \cup_{L_2} M$ is comical, as a pushout of the comical map $\bd_{3,0} \colon L_2 \hookrightarrow M$. Furthermore, the composite of these two pushout squares agrees with the composite of the two squares above; thus the two pushout objects coincide. We denote this common pushout object by $LM$.

Now let $LM' = \tau_0 LM$; by \cref{tau-comical}, the inclusions of $L_{1,1}'$ and $L_{1,2}'$ into $LM'$ are comical. Thus we have the following zigzag of commuting squares in which the horizontal maps are comical, witnessing the comical equivalence of the maps $L_{1,1} \to L_{1,1}'$ and $L_{1,2} \to L_{1,2}'$.
\[
\xymatrix{
L_{1,1} \ar[r] \ar[d] & LM \ar[d] & L_{1,2} \ar[l] \ar[d] \\
L_{1,1}' \ar[r] & LM' & L_{1,2}' \ar[l] \\
}
\]

Applying \cref{involutions-Rezk,comical-equiv-involutions}, we obtain further comical equivalences between the following pairs of elementary Rezk maps:

\begin{itemize}
\item $L_{2,2} \to L_{2,2}'$ and $L_{2,1} \to L_{2,1}'$;
\item $L_{2,2} \to L_{2,2}'$ and $L_{1,2} \to L_{1,2}'$'
\item $L_{1,1} \to L_{1,1}'$ and $L_{2,1} \to L_{2,1}'$. 
\end{itemize}

By the transitivity of comical equivalence, it follows that all four elementary Rezk maps are comically equivalent.
\end{proof}

We are now able to prove our claim regarding the sufficiency of any one choice of elementary Rezk map in defining the pseudo-generating trivial cofibrations of the saturated comical model structures.

\begin{thm}\label{one-Rezk-map}
In the ($n$-trivial) saturated comical model structure on $\cSet^+$, for any $x, y \in \{0,1\}$, a map with fibrant codomain is a fibration if and only if it has the right lifting property with respect to the following classes of maps:
\begin{itemize}
\item comical open box inclusions;
\item comical marking extensions;
\item Rezk maps of the form $(\bd \Box^m \hookrightarrow \Box^m) \hat{\otimes} (L_{x,y} \to L_{x,y}') \hat{\otimes} (\bd \Box^n \hookrightarrow \Box^n)$;
\item in the $n$-trivial case, markers of dimension greater than $n$.
\end{itemize}
\end{thm}

\begin{proof}
For any choice of $x$ and $y$, we may apply \cref{CO-with-monoidal} to construct a model structure $\mathcal{M}$ on $\mcSet$, monoidal with respect to the Gray tensor product, in which the cofibrations are the monomorphisms and the fibrations with fibrant codomain are defined by the given lifting property. The proof that the hypotheses of \cref{CO-with-monoidal} are satisfied is effectively identical to the proof of \cref{comical-model-structure}. 

We will show that $\mathcal{M}$ coincides with the ($n$-trivial) saturated comical model structure. As the two model structures have the same cofibrations, it suffices to show that they have the same fibrations with fibrant codomain. 

It is clear that any fibration with fibrant codomain in the ($n$-trivial) saturated comical model structure is a fibration with fibrant codomain in $\mathcal{M}$. To prove the converse, observe that by \cref{comical-2-of-3,Rezk-comic-equiv}, all elementary Rezk maps are trivial cofibrations in $\mathcal{M}$. By monoidality, it then follows that all Rezk maps are trivial cofibrations in $\mathcal{M}$. Thus the fibrations with fibrant codomain in $\mathcal{M}$ have the right lifting property with respect to all Rezk maps; it follows that these maps are fibrations with fibrant codomain in the ($n$-trivial) saturated comical model structure as well.
\end{proof}


\section{Triangulation is a Quillen functor}\label{Quillen-functor}


The main result of this section is \cref{T-Quillen} asserting that the adjunction $T : \cSet^+ \rightleftarrows \sSet^+ : U$ is Quillen.
The proofs in this section will not make use of connections, and thus our results are valid for all of the marked cubical set categories under consideration.

Before proceeding with the proof, we explain the strategy.
It is clear that $T$ preserves monomorphisms, and hence the key difficulty lies in showing that it takes pseudo-generating trivial cofibrations to trivial cofibrations between marked simplicial sets.
Of those, the majority of the work goes into the cases of comical open box inclusions and comical marking extensions, as the preservation of the other ones by $T$ is simple enough that it can be treated directly in the proof of \cref{T-Quillen} at the end of the section.

To show that $T$ takes comical open box inclusions and comical marking extensions to acyclic cofibrations, we will analyze the simplices of $T \Box^n$, defining a partial ordering relation on them and showing that, e.g., $T\Box^n_{i, \varepsilon}$ can be built out of $T(\sqcap^n_{i,\varepsilon})$ via trivial cofibrations, by induction on this partial order.

We will use $(\sqcap^n_{i,\varepsilon})^\flat$ to denote the minimal marking of the $(i,\varepsilon)$-open box. 
Note that although $(\sqcap^n_{i,\varepsilon})^\flat$ is minimally marked, its triangulation $T(\sqcap^n_{i,\varepsilon})^\flat$ is not in general.
Indeed, by the definition of triangulation $T \colon \cSet^+ \to \sSet^+$ (\cref{marked-triangulation}), for $n \geq 3$, there are non-degenerate marked simplices in $T \Box^{n-1}$, although there are no non-degenerate marked cubes in $\Box^{n-1}$, and $(\sqcap^n_{i,\varepsilon})^\flat$ is a colimit of those.
In particular, the triangulation of each of the $2n-1$ faces of $(\sqcap^n_{i,\varepsilon})^\flat$ contains a unique unmarked $(n-1)$-simplex (and hence $(n-1)!-1$ marked).
This will be important in the proof of \cref{open-box-xi-anodyne}.

Recall from \cref{sec:background} that $r$-simplices of $T \Box^n$ can be identified with strings $\{1, 2, \ldots, n\} \to \{1, 2, \ldots, r, \pm \infty\}$ and we write $\phi_i$ for the $i^{\text{th}}$ entry of the string associated to a simplex $\phi \colon \Delta^r \to T \Box^n$.

\begin{lem}\label{face-preserve-order}
Let $\phi \colon \Delta^m \to T \Box^n$, and suppose that for some $i, j \in \{1,\ldots,n\}$ we have $\phi_i \leq \phi_j$. Then for any $l \leq m$ and any (composite) face map $\delta \colon \Delta^l \to \Delta^m$ we have $(\phi \delta)_i \leq (\phi \delta)_j$.
\end{lem}

\begin{proof}
It suffices to consider the case $l = m - 1, \delta = \bd_k$ for some $1 \leq k \leq m$. If $k < \phi_i$, then both $\phi_i$ and $\phi_j$ are lowered by 1 in computing $\phi \bd_k$, thus the inequality is preserved. Likewise, if $k \geq \phi_j$ then both $\phi_i$ and $\phi_j$ are unchanged in $\phi \bd_k$. On the other hand, if $\phi_i \leq k < \phi_j$, then $\bd_k$ lowers $\phi_j$ by 1 while leaving $\phi_i$ unchanged. But in this case $\phi_i < \phi_j$, implying $\phi_i \leq \phi_j - 1$.
\end{proof}

\begin{Def}\label{complete-substring-def}
Given a simplex $\phi \colon \Delta^m \to T\Box^n$, a \emph{complete substring} of $\phi$ is an order-preserving map $\rho \colon \{1, \ldots ,m\} \to \{1, \ldots ,n\}$ such that the composite $\phi \rho$ is equal to the inclusion $\{1, \ldots ,m\} \hookrightarrow \{1, \ldots ,m,\pm \infty\}$.
\end{Def}

Intuitively, a complete substring $\rho$ of a simplex $\phi$ is an increasing sequence of positions $\rho_i$ such that $\phi_{\rho_i} = i$ for all $i$. For instance, the string $1\,3\,3\,2\,3$, representing a $3$-simplex of $T\Box^5$, has one complete substring $\rho$, given by taking its first, fourth and fifth entries. Viewing $\rho$ formally as a function $\{1,2,3\} \to \{1,2,3,4,5\}$, as in \cref{complete-substring-def}, we have $\rho(1) = 1, \rho(2) = 4, \rho(3) = 5$. On the other hand, the string $1\,4\,2\,3\,3$, representing a $4$-simplex of $T\Box^5$, has no complete substrings.

A simplex of $T \Box^n$ is marked if and only if it has no complete substrings.

We will also have occasion to consider the images of simplices of triangulated cubes under cubical face maps $T\bd_{i,\varepsilon} \colon T\Box^{n-1} \to T\Box^n$.
To this end, for an $n$-simplex $\phi \colon \Delta^m \to T\Box^{n-1}$, we will slightly abuse notation by writing $\bd_{i,\varepsilon}\phi$ for the composite $(T \bd_{i, \varepsilon}) \phi$.
Explicitly, $\bd_{i,\varepsilon}\phi$ can be described as follows:
\[
	(\bd_{i,\varepsilon}\phi)_j = \left\{\begin{array}{cl}
	\phi_i, & j < i \\
	+\infty & j = i, \varepsilon = 0 \\
	-\infty & j = i, \varepsilon = 1 \\
	\phi_{j-1}, & j > i \end{array}\right.
	\]

Generalizing this description to an arbitrary (composite) face map $\delta \colon \Box^m \to \Box^n$ and a simplex $\phi \colon \Delta^k \to T\Box^m$, we denote the composite $(T \delta) \phi$ by $\delta \phi$. Then the entries of the string representing the $k$-simplex $\delta \phi$ of $T \Box^n$ may be described as follows.

\begin{itemize}
\item If $\bd_{i,0}$ appears in the standard form of $\delta$, then $(\delta \phi)_{i} = + \infty$.
\item If $\bd_{i,1}$ appears in the standard form of $\delta$, then $(\delta \phi)_{i} = - \infty$.
\item The remaining entries of $\delta \phi$ consist of the entries of $\phi$, in the same order in which they appear in $\phi$. More precisely, denote the $m$ values between $1$ and $n$ which do not appear as indices of any map in the standard form of $\delta$ as $i_{1} < \ldots < i_{m}$; then for each $1 \leq j \leq m$ we have $(\delta \phi)_{i_{j}} = \phi_{j}$.
\end{itemize}

In other words, $\delta \phi$ is represented by the string of length $n$ obtained by inserting $+ \infty$ to the string representing $\phi$ in each position $i$ for which $\bd_{i,0}$ appears in the standard form of $\delta$, and inserting $- \infty$ in each position $i$ for which $\bd_{i,1}$ appears in the standard form of $\delta$.

In particular, we will have considerable occasion to study simplices of the form $\delta \phi$ in the case where $\phi = \iota_m$ (recall that this is defined to be the $m$-simplex of $T \Box^m$ represented by the string $1\,\ldots\,m$). Thus we define special terminology and notation for these simplices.

\begin{Def} \label{def:linear-simplices}
For $1 \leq m \leq n$, given a (composite) face map $\delta \colon \Box^m \to \Box^n$, the \emph{linear simplex} of $T \Box^n$ associated to $\delta$, denoted $\iota_\delta$, is the image under $\delta$ of the $m$-simplex $\iota_m$.
\end{Def}

\begin{egs}
We consider some examples of linear simplices to better illustrate the concept.

\begin{itemize}
\item For $\delta = \bd_{2,0} \colon \Box^2 \to \Box^3$, $\iota_\delta = 1+2$.
\item For $\delta = \bd_{5,0}\bd_{2,1}\bd_{1,0} \colon \Box^3 \to \Box^6$, $\iota_\delta = +-1\,2+3$.
\item For any $n$, $\iota_{\mathrm{id}_{[1]^n}} = \iota_n$.
\end{itemize}
\end{egs}

The concept of a linear simplex is useful in studying the triangulations of marked cubical sets obtained by marking certain faces of a standard cube, such as those of the form $\Box^n_{i,\varepsilon}$ and $(\Box^n_{i,\varepsilon})'$.

\begin{lem}\label{linear-simplex-marked}
Let $X$ denote a marked cubical set whose underlying cubical set is $\Box^n$ for some $n \geq 0$. Then $TX$ is obtained from $T \Box^n$ by marking the linear simplex $\iota_\delta$ associated to each face map $\delta \colon \Box^m \to \Box^n$ such $\delta$ is marked when viewed as an $m$-cube of $X$.
\end{lem}

\begin{proof}
We first note that each linear simplex $\iota_\delta$ contains a complete substring, and is thus unmarked in $T \Box^n$. Next we observe that $X$ may be constructed from $\Box^n$ by the following pushout, in which the coproducts range over all non-degenerate marked cubes of $X$.
\[
\begin{tikzcd}
\bigsqcup \Box^m \arrow[d] \arrow[r] \pushout & \Box^n \arrow[d] &  \\
\bigsqcup \widetilde{\Box}^m \arrow[r] & X \\
\end{tikzcd}
\]
Since $T$ preserves pushouts as a left adjoint, we have a corresponding pushout diagram in $\cSet^+$:
\[
\begin{tikzcd}
\bigsqcup T\Box^m \arrow[d] \arrow[r] \pushout & T\Box^n \arrow[d] &  \\
\bigsqcup T\widetilde{\Box}^m \arrow[r] & TX \\
\end{tikzcd}
\]
For each $m$, the map $T \Box^m \to T \widetilde{\Box}^m$ is the entire map which marks the $m$-simplex $\iota_m$. It thus follows that $T X$ is obtained from $T \Box^n$ by marking $\iota_\delta$ for each non-degenerate $\delta \colon \Box^m \to \Box^n$ which is marked in $X$.
\end{proof}

The following results are immediate from our earlier characterization of the actions of cubical face maps.

\begin{lem}\label{linear-characterization}
An $m$-simplex $\phi \colon \Delta^m \to T\Box^n$ is linear if and only if it has a unique complete substring $\rho$, and for any $1 \leq i \leq n$ not in the image of $\rho$, $\phi_i \in \{\pm \infty\}$. \qed
\end{lem}

\begin{lem}\label{linear-simplex-recover}
Let $\phi = \iota_\delta \colon \Delta^m \to T\Box^n$ be the linear simplex associated to a cubical face map $\delta \colon \Box^m \to \Box^n$, and let $\rho$ denote the unique complete substring of $\phi$. Then $\delta = \bd_{i_1,\varepsilon_1} \ldots \bd_{i_{n-m},\varepsilon_{n-m}}$, where:

\begin{itemize}
\item the indices $i_1 > \ldots > i_{n-m}$ range over $\{1,\ldots,n\} \setminus \Im \rho$;
\item for $1 \leq i \leq n - m$, $\varepsilon_i = 0$ if $\phi_i = + \infty$, while $\varepsilon _i = 1$ if $\phi_i = - \infty$. \qed
\end{itemize}
\end{lem}

\begin{lem}\label{comical-triangulation-marking}
For $n \geq 1$ and $1 \leq i \leq n$, let $\phi$ be a linear simplex of $T\Box^n$, and let $\rho$ denote its unique complete substring. Suppose that $i$ is in the image of $\rho$. If for all $\rho_{\phi_{i}-1} < k < \rho_{\phi_{i}+1}$ such that $k \neq i$ we have $\phi_k = - \infty$ (resp. $\phi_k = + \infty$), then $\phi$ is marked in $T \Box^n_{i,0}$ (resp.~$T \Box^n_{i,1}$). (If $\phi_i = 1$ then we interpret $\rho_0$ to be 0; likewise if $\phi_i = n$ then we interpret $\rho_{n+1}$ to be $n+1$.)
\end{lem}

\begin{proof}
We prove the case for $T \Box^n_{i,0}$; the case for $T \Box^n_{i,1}$ is similar. Let $\phi = \iota_\delta$ for some face map $\delta \colon \Box^m \to \Box^n$. Using \cref{linear-simplex-recover}, we may convert the given conditions on the unique complete substring of $\phi$ into a set of conditions on the standard form of $\delta$, as follows.
\begin{itemize}
\item The condition that $i$ is in the image of $\rho$ is equivalent to the condition that neither $\bd_{i,0}$ nor $\bd_{i,1}$ appears in the standard form of $\delta$.
\item The values $\rho_{\phi_i - 1}$ and $\rho_{\phi_i + 1}$ are, respectively, the largest value below $i$ and the smallest value above $i$ which do not appear as indices in the standard form of $\delta$. Thus the condition that $\phi_k = - \infty$ for all $\rho_{\phi_i - 1} < k < \rho_{\phi_i + 1}$ is equivalent to the condition that all face maps in the standard form of $\delta$ with indices in between $i$ and the next higher and lower missing values have $\varepsilon = 1$ -- in other words, that there is no $j > i$ such that the standard form of $\delta$ contains $\bd_{j,0}$ as well as $\bd_{k,1}$ for all $i < k < j$, and similarly there is no $j < i$ such that the standard form of $\delta$ contains $\bd_{j,0}$ as well as $\bd_{k,1}$ for all $i > k > j$.
\end{itemize}

Thus we see that the linear simplices satisfying these conditions are precisely those associated to faces of $\Box^n$ which are marked in $\Box^n_{i,0}$. The stated result thus follows by \cref{linear-simplex-marked}.
\end{proof}

Although the conditions of \cref{comical-triangulation-marking} may seem technical and obscure, simplices which satisfy them may easily be recognized by visual examination of the corresponding strings. Specifically, linear $m$-simplices satisfying the criteria of the lemma correspond to strings consisting of the sequence $1\,\ldots\,m$ interspersed with occurrences of the values $+ \infty$ and $-\infty$, such that an entry of $1\,\ldots\,m$ appears in position $i$, and all values appearing in between this and the previous and following entries of the sequence are equal to $-\infty$ (in the case $\varepsilon = 0$) or $+ \infty$ (in the case $\varepsilon = 1$). To illustrate, we consider various simplices of $\Box^6_{4,0}$.

\begin{itemize}
\item The $3$-simplex $1\,-\,-\,2\,-\,3$ satisfies the conditions of the lemma.
\item The $3$-simplex $+\,1\,-\,2\,3\,-$ satisfies the conditions of the lemma.
\item The $2$-simplex $1\,-\,-\,-\,2\,-$ does not satisfy the conditions of the lemma, as it has the value $-\infty$ in position $4$, indicating the presence of $\bd_{4,1}$ in the corresponding face map.
\item The $3$-simplex $1\,+\,-\,2\,-\,3$ does not satisfy the conditions of the lemma, as it contains the value $+\infty$ in between position $4$ and the previous position occupied by an entry of the sequence $1\,2\,3$ (specifically, we have $+\infty$ in position $2$, in between positions $1$ and $4$). This, together with the appearance of $-\infty$ in position $3$, indicates the presence in the standard form of $\delta$ of $\bd_{3,1} \bd_{2,0}$.
\end{itemize}

The non-degenerate $m$-simplices of $T \Box^n$ are those for which the corresponding string includes all of the values $1, \ldots ,m$; the interior simplices, i.e. those not contained in $T \bd \Box^n$, are those for which the corresponding string does not include the values $+$ or $-$.

\begin{Def}
The \emph{essential} simplices of $T \Box^n$ are those which are both non-degenerate and interior. For $1 \leq m \leq n$, the set of essential $m$-simplices is denoted $K_m$.
\end{Def}

\begin{Def}
Given an essential $m$-simplex $\phi$ in $T \Box^n$, we define the following data:

\begin{itemize}
\item $P(\phi)$ is the largest value $1 \leq r \leq m$ such that for all $1 \leq i \leq r$, $\phi_j = i$ if and only if $j = i$, or $0$ if no such $r$ exists. $\Pi(\phi)$, the \emph{preamble} of $\phi$, is the initial segment of $\phi$ defining $P(\phi)$, i.e. the substring $1\,\ldots\,r$, or the empty string if $P(\phi) = 0$.
\item $Q(\phi) = P(\phi) + 1$. If $Q(\phi) \leq n$, then $q(\phi)$ is the value $\phi_{Q(\phi)}$; otherwise, $q(\phi) = n + 1$.
\end{itemize}
\end{Def}

More intuitively, $P(\phi)$ is the largest $r$ such that $\phi$ begins with a string of the form $\Pi(\phi) = 1  \ldots  r$, none of whose entries appear in any later position of $\phi$, or $0$ if no such string exists. $Q(\phi)$ is the first position $i$ such that $\phi_i$ is not part of such a string, either because its value $\phi_i = q(\phi)$ is greater than $i$ itself, or because this value is repeated later on. The case $P(\phi) = n$, in which $Q(\phi) = q(\phi) = n + 1$, occurs if and only if $\phi$ is the $n$-simplex $\iota_n$.  

\begin{egs}
We compute $\Pi(\phi), P(\phi), Q(\phi)$, and $q(\phi)$ for various essential simplices of $T \Box^5$ in order to better illustrate the concepts.
\begin{itemize}
\item For the $5$-simplex represented by a function $\phi \colon \{1, 2, 3, 4, 5\} \to \{ 1, 2, 3, 4, 5, \pm \infty \}$ given by $\phi = 1\,2\,3\,5\,4$, we have $\Pi(\phi) = 1\,2\,3, P(\phi) = 3$, $Q(\phi) = 4, q(\phi) = 5$.
\item For the $4$-simplex represented by a function $\phi \colon \{1, 2, 3, 4, 5\} \to \{ 1, 2, 3, 4, \pm \infty \}$ given by $\phi = 1\,2\,3\,4\,3$, $\Pi(\phi) = 1\,2, P(\phi) = 2$, $Q(\phi) = 3, q(\phi) = 3$.
\item For the $3$-simplex represented by a function $\phi \colon \{1, 2, 3, 4, 5\} \to \{ 1, 2, 3, \pm \infty \}$ given by $\phi = 2\,3\,1\,1\,1$, $\Pi(\phi) = \varnothing, P(\phi) = 0$, $Q(\phi) = 1, q(\phi) = 2$.
\end{itemize}
\end{egs}

\begin{lem}
For $m \leq n$ and $\phi \in K_m$ we have $q(\phi) \geq Q(\phi)$.
\end{lem}

\begin{proof}
The value of $\phi$ at position $Q(\phi)$ cannot be less than $Q(\phi)$, as it would then be a repetition of some value in the preamble of $\phi$.
\end{proof}

\begin{Def}
For $1 \leq m \leq n$, we define the following subsets of $K_m$:

\begin{itemize}
\item $K_m^\ast$, the set of \emph{normal} essential $m$-simplices, consists of all simplices $\phi \in K_m$ such that the value $q(\phi)$ appears exactly once in $\phi$.
\item $K_m'$, the set of \emph{abnormal} essential $m$-simplices, is $K_m \setminus K_m^\ast$.
\end{itemize}
\end{Def}

Intuitively, when building $T \Box^n_{i, \varepsilon}$ from $T \sqcap^n_{i, \varepsilon}$, the normal simplices will appear as interiors of the fillers for complicial horns, while the abnormal simplices will appear as the composite faces of those fillers.
This is not entirely accurate, as the situation is more complicated when filling horns of dimension $n$.

The following characterization of $K_m'$ is immediate from the definition.

\begin{lem}\label{K-prime}
For $m \leq n-1$, $K_m'$ consists of those $\phi$ for which the value $q(\phi)$ appears at least twice. For $m = n$, $K_m'$ consists of the single $n$-simplex $\iota_n$. \qed
\end{lem}



We next define a construction relating normal and abnormal essential simplices, which will be of significant use in proving that $T \dashv U$ is a Quillen adjunction. 

\begin{Def}
For $1 \leq m \leq n - 1$,  the \emph{normalization} of an abnormal essential $m$-simplex $\phi \colon \Delta^m \to T\Box^n$, denoted $N(\phi)$, is the $(m+1)$-simplex obtained from $\phi$ by raising the value of $\phi$ at $Q(\phi)$, and at all $i$ such that $\phi_i > q(\phi)$, by 1.
\end{Def}

Note that, in constructing $N(\phi)$, occurrences of the value $q(\phi)$ in positions other than $Q(\phi)$ are unchanged. For instance, $N(1\,2\,3\,2) = 1\,3\,4\,2$.

\begin{lem}\label{N-bijection}
For $1 \leq m \leq n - 1$, normalization defines a bijection $N \colon K_{m}' \to K_{m+1}^\ast$, with its inverse given by taking the $(q(\psi)-1)$-face of a simplex $\psi \in K_{m+1}^\ast$.
\end{lem}

\begin{proof}
Let $\phi \in K_{m}'$. We first show that $N(\phi)$ is essential, i.e. that it is non-degenerate and interior. It is clear from the construction of $N(\phi)$ that it does not contain any entries of the form $+$ or $-$. To see that every element of $\{1, \ldots ,m+1\}$ appears in $N(\phi)$ at least once, consider the following:

\begin{itemize}
\item for $1 \leq i \leq q(\phi) - 1$, $i$ appears at least once in $\phi$, and these entries are unchanged in constructing $N(\phi)$;
\item by \cref{K-prime}, the value $q(\phi)$ appears at least twice in $\phi$, and only one of these entries is altered in constructing $N(\phi)$;
\item for $q(\phi) + 1 \leq i \leq m + 1$, $\phi$ contains some instance of the value $i - 1$ which is raised by 1 in constructing $N(\phi)$.
\end{itemize}

To see that $N(\phi)$ is in $K_{m+1}^\ast$, observe that $P(N(\phi)) = P(\phi)$, as the preamble of $\phi$ is unchanged in constructing $N(\phi)$; as $q(\phi) \geq Q(\phi)$, and this value is raised in constructing $N(\phi)$, the entry in position $Q(\phi)$ is not part of the preamble of $N(\phi)$. Thus $Q(N(\phi)) = Q(\phi)$, and $q(N(\phi)) = q(\phi) + 1$. Moreover, any entries having the value $q(\phi) + 1$ in $\phi$ are raised by 1 in constructing $N(\phi)$; thus $q(\phi) + 1$ appears exactly once in $N(\phi)$.

To see that this function is a bijection with the stated inverse, first let $\phi \in K_{m}'$, and consider $N(\phi) \bd_{q(N(\phi))-1} = N(\phi) \bd_{q(\phi)}$. This face is computed by lowering all entries of $N(\phi)$ greater than $q(\phi)$ by 1; as these are precisely the entries that were raised by 1 in order to obtain $N(\phi)$, this recovers the original simplex $\phi$.

Now let $\psi \in K_{m+1}^\ast$. Since $q(\psi)$ appears exactly once in $\psi$ by assumption, it must be greater than or equal to $P(\psi) + 2$, or else it would be part of the preamble of $\psi$. This implies that for some $i > Q(\psi)$ we have $\psi_i = q(\psi) - 1$. Now consider the face $\psi \bd_{q(\psi) - 1}$. This face is computed by lowering every entry of $\psi$ which is greater than or equal to $q(\psi)$ by 1. In particular, $\psi_{Q(\psi)}$ is reduced to $q(\psi) - 1$, while the preamble of $\psi$ is unaffected. As $\psi$ contains at least one other entry having the value $q(\psi) - 1$, which is not changed in computing this face, we see that $P(\psi \bd_{q(\psi)-1}) = P(\psi)$, $Q(\psi \bd_{q(\psi)-1}) = Q(\psi)$, and $q(\psi \bd_{q(\psi) -1}) = q(\psi) - 1$. Therefore, to compute $N(\psi \bd_{q(\psi) -1})$, we raise the entry in position $Q(\psi)$, and all entries greater than or equal to $q(\psi)$, by 1 -- but these were precisely the entries of $\psi$ that were lowered to obtain $\psi \bd_{q(\psi) -1}$. Thus $N(\psi \bd_{q(\psi) -1}) = \psi$.
\end{proof}

In order to establish certain horns as being complicial, we will need to know that they have specific faces marked.
A large class of marked simplices in $T \Box^n_{i, \varepsilon}$ is given by the following characterization.

\begin{Def}
For $2 \leq i \leq m$, a simplex $\phi \colon \Delta^m \to T\Box^n$ is \emph{$i$-disordered} if it has exactly one entry with the value $i$, and none of its preceding entries have the value $i-1$.
\end{Def}

\begin{lem}\label{disordered-marked}
Every simplex of $T \Box^n$ which is $i$-disordered for some $i$ is marked.
\end{lem}

\begin{proof}
It is immediate from the definition that an $i$-disordered simplex cannot contain any complete substring.
\end{proof}

\begin{lem}\label{disordered-face}
Let $\phi$ be an $i$-disordered $m$-simplex of $T \Box^n$ for some $2 \leq i \leq m$, and consider a face $\phi \bd_{j}$. If $j \geq i + 1$ then $\phi \bd_j$ is $i$-disordered. If $i \geq 3$ and $j \leq i - 3$, then $\phi \bd_j$ is $(i-1)$-disordered. 
\end{lem}

\begin{proof}
For the case $j \geq i + 1$, the face $\bd_j$ only lowers (or replaces with $+$) entries with values greater than or equal to $i + 2$. Thus $\phi \bd_j$ will still have a unique entry with value $i$, and will no have no new entries with value $i-1$.

Now consider the case $j \leq i - 3$. In this case, $\bd_j$ lowers all entries having the value $i$, $i-1$, or $i-2$. Thus $\phi \bd_j$ has a unique entry with the value $i - 1$, namely that whose position coincides with that of the unique $i$ in $\phi$. Moreover, any entry having the value $i-2$ in $\phi \bd_j$ must have the value $i-1$ in $\phi$; thus there is no entry preceding the unique $i$ in $\phi \bd_j$  whose value is $i-2$.
\end{proof}

\begin{lem}\label{disordered-complicial}
For $2 \leq i \leq m$, if $\phi$ is an $i$-disordered $m$-simplex of $T \Box^n$, then $\phi$ is $(i-1)$-complicial. 
\end{lem}

\begin{proof}
We must show that each simplex of the form $\phi\bd_{j_1} \ldots \bd_{j_a}\bd_{k_1} \ldots \bd_{k_b}$, where $j_1 >  \ldots  > j_a \geq i + 1$ and $i-3 \geq k_1 >  \ldots  > k_b$, is marked. (Note that either or both of the strings $j_1, \ldots ,j_a$ and $k_1, \ldots ,k_b$ may be empty.) By repeatedly applying \cref{disordered-face}, we can see that this simplex is $(i-b)$-disordered; thus it is marked by \cref{disordered-marked}.
\end{proof}

\begin{cor}\label{N-complicial}
For $\phi \in K_{m}'$, the $(m+1)$-simplex $N(\phi)$ is $q(\phi)$-complicial.
\end{cor}

\begin{proof}
From the definition of $N$, we can see that $N(\phi)$ is $(q(\phi)+1)$-disordered. The statement thus follows from \cref{disordered-complicial}.
\end{proof}

Our construction of $T \Box^n_{i, \varepsilon}$ from $T \sqcap^n_{i, \varepsilon}$ will involve applying \cref{more-markings} to some essential simplices, and hence we need to identify (some of) those simplices that become marked in the precomplicial reflection of $T \Box^n_{i, \varepsilon}$.
For this, we introduce the concept of linearization and show that the simplices whose linearizations are marked in $T \Box^n_{i, \varepsilon}$ are marked in $(T \Box^n_{i, \varepsilon})^\precomp$.

\begin{Def}
Let $\phi$ be a simplex of $T \Box^n$, and let $\rho$ be a complete substring of $\phi$. The \emph{linearization} of $\phi$ associated to $\rho$ is the $m$-simplex $\phi^\rho$ defined as follows:

\[
	\phi^\rho_{i} = \left\{\begin{array}{cl}
	+\infty, & \phi_i = +\infty, \,\, \text{or} \,\,\, \phi_i \in \{1,\ldots,m\} \,\,\, \text{and} \,\,\, i < \rho_{\phi_i} \\
	\phi_i &  \phi_i \in \{1,\ldots,m\} \,\,\, \text{and} \,\,\, i = \rho_{\phi_i} \\
	-\infty, & \phi_i = -\infty, \,\, \text{or} \,\,\, \phi_i \in \{1,\ldots,m\} \,\,\, \text{and} \,\,\, i > \rho_{\phi_i} \end{array}\right.
	\]

\end{Def}

Note that all linearizations are linear simplices.
Indeed, from the definition of linearization, we see that a linearization of an $m$-simplex consists of the sequence $1 \, \ldots \, m$ interspersed with occurrences of $+\infty$ and $- \infty$.
By comparison with the discussion preceding \cref{def:linear-simplices}, we see that linear simplices are precisely the ones fitting this description.
Thus we may, for instance, apply \cref{comical-triangulation-marking} to show that a linearization is marked.

\begin{egs}
To illustrate the concept of a linearization, we consider the linearizations of various simplices:
\begin{itemize}
\item The unique linearization of $2\,1\,2\,1\,3$ is $+\,1\,2-\,3$.
\item The linearizations of $1\,2-2$ are $1\,2--$ and $1+-\,2$.
\item The linearizations of $1\,1\,1$ are $1--$, $+\,1-$, and $++1$.
\item The linearizations of $1\,2\,3\,2\,3$ are $1\,2\,3--$, $1\,2+-3$, and $1++\,2\,3$.
\item Every linear simplex is its own unique linearization.
\item A simplex of $T \Box^n$ is marked if and only if it has no complete substrings, and hence no linearizations.
\end{itemize}
\end{egs}

\begin{lem}\label{cube-face-linearization}
Cubical face maps preserve linearizations. That is, for $\phi \colon \Delta^m \to T \Box^{n-1}$, for any face map $\bd_{i,\varepsilon} \colon \Box^{n-1} \to \Box^n$, the linearizations of $\bd_{i,\varepsilon} \phi$ are precisely the images under $\bd_{i,\varepsilon}$ of the linearizations of $\phi$.
\end{lem}

\begin{proof}
This is immediate from the definitions of linearization and the actions of cubical face maps.
\end{proof}

We next consider the interaction between linearization and normalization, and demonstrate a relationship between the linearizations of an abnormal essential simplex $\phi$ and those of the faces of its normalization $N(\phi)$. Before proving the general result (\cref{N-linearization}), we consider an illustrative example.

Let $\phi$ denote the $3$-simplex of $T\Box^5$ represented by the string $1\,2\,1\,2\,3$; then $q(\phi) = 2$ and $N(\phi) = 1\,3\,1\,2\,4$. Moreover, we may observe that $\phi$ has three linearizations: $1\,2\,-\,-\,3$, $1\,+-\,2\,3$, and $+\,+\,1\,2\,3$.
By \cref{N-bijection} we have $N(\phi) \bd_{2} = \phi$; the remaining faces of $N(\phi)$ and their linearizations are listed below.

\begin{itemize}
\item $N(\phi) \bd_0 = -\,2\,-\,1\,3$, with no linearizations.
\item $N(\phi) \bd_1 = 1\,2\,1\,1\,3$, with linearization $1\,2\,-\,-\,3$.
\item $N(\phi) \bd_3 = 1\,3\,1\,2\,3$, with linearizations $1\,+\,-\,2\,3$ and $+\,+\,1\,2\,3$.
\item $N(\phi) \bd_4 = 1\,3\,1\,2\,+$, with no linearizations.
\end{itemize}

Thus we see that the linearizations of $\phi$ are ``distributed'' between the other faces of $N(\phi)$; the one linearization corresponding to a complete substring which includes position 2 is the unique linearization of $N(\phi) \bd_1$, while those corresponding to complete substrings which do not include position 2 are the linearizations of $N(\phi) \bd_3$.

This is an instance of a general pattern, which we now prove.

\begin{lem}\label{N-linearization}
For an abnormal essential simplex $\phi \colon \Delta^m \to T\Box^n$, the linearizations of the faces of $N(\phi)$ (other than $N(\phi) \bd_{q(\phi)} = \phi$) are as follows:

\begin{enumerate}
\item\label{N-linearization-minus} The linearizations of $N(\phi) \bd_{q(\phi)-1}$ are the linearizations of $\phi$ corresponding to complete substrings which include $Q(\phi)$;
\item\label{N-linearization-plus} The linearizations of $N(\phi) \bd_{q(\phi)+1}$ are the linearizations of $\phi$ corresponding to complete substrings which do not include $Q(\phi)$;
\item\label{N-linearization-other} For $i < q(\phi) - 1$ or $i > q(\phi) + 1$, $N(\phi)\bd_i$ has no linearizations.
\end{enumerate}
\end{lem}

\begin{proof}
For \cref{N-linearization-minus}, observe that $N(\phi) \bd_{q(\phi)-1}$ is obtained from $\phi$ by lowering all entries of $\phi$ having the value $q(\phi)$, other than that in position $Q(\phi)$. Thus any linearization of $N(\phi) \bd_{q(\phi)-1}$ must include $Q(\phi)$. Moreover, it cannot include any of the entries which were changed to obtain $N(\phi) \bd_{q(\phi)-1}$, as these all appear after position $Q(\phi)$ and have value $q(\phi) - 1$. Thus we see that the complete substrings of $N(\phi) \bd_{q(\phi)-1}$ are those of $\phi$ which do not include any entries which are changed in $N(\phi) \bd_{q(\phi) - 1}$, and these are precisely those which include $Q(\phi)$. Furthermore, note that for any such complete substring $\rho$, those entries of $\phi$ which are lowered to obtain $N(\phi) \bd_{q(\phi)-1}$ are replaced by $-$ in the linearization $\phi^\rho$, as they have value $q(\phi)$ and appear after position $Q(\phi) = \rho_{q(\phi)}$; in the corresponding linearization of $N(\phi) \bd_{q(\phi)-1}$ they will still be replaced with $-$, as they now have value $q(\phi) - 1$, and appear after position $\rho_{q(\phi) - 1}$.

The proof of \cref{N-linearization-plus} is similar. Observe that $N(\phi) \bd_{q(\phi) + 1}$ is obtained from $\phi$ by raising the value in position $Q(\phi)$ to $q(\phi) + 1$ and leaving all other entries unchanged. As there are no preceding entries having the value $q(\phi)$, there can be no complete substring of $N(\phi) \bd_{q(\phi) + 1}$ including $Q(\phi)$. Therefore, as in the previous case, we see that the complete substrings of $N(\phi) \bd_{q(\phi) + 1}$ are those of $\phi$ which involve only the positions whose values are unchanged in $N(\phi) \bd_{q(\phi) + 1}$ -- in this case, these are the positions other than $Q(\phi)$. In the linearizations of both $\phi$ and $N(\phi) \bd_{q(\phi) + 1}$ corresponding to these complete substrings, the entry in position $Q(\phi)$ is replaced with $+$, as its value is greater than or equal to $q(\phi)$ and its position is earlier than $\rho_{q(\phi)}$.

Item \ref{N-linearization-other} is immediate from \cref{disordered-face}. 
\end{proof}

We are now ready to define the partial order on the non-degenerate $m$-simplices of $T \Box^n$.

\begin{Def}\label{simplex-order}
For $m \geq 1$, we define a partial order on the non-degenerate $m$-simplices of $T \Box^n$ as follows:
\begin{itemize}
\item if $\phi$ is contained in $T \bd \Box^n$ or $\phi \in K_m^\ast$, and $\psi \in K_m'$ then $\phi < \psi$;
\item for distinct $\phi, \psi \in K_{m}'$, we have $\phi < \psi$ if either $P(\phi) < P(\psi)$, or $P(\phi) = P(\psi)$ and $q(\phi) > q(\psi)$.
\end{itemize}
The relation $<$ defined above is easily seen to be transitive and anti-symmetric; the partial order $\leq$ is defined to be its reflexive closure.
\end{Def}

\begin{lem}\label{N-order}
For $\phi \in K_{m}'$ and $0 \leq i \leq m+1$, $i \neq q(\phi)$, we have $\phi > N(\phi) \bd_i$.
\end{lem}

\begin{proof}
We proceed by case analysis on $i$.

\begin{itemize}
\item If $i = 0$ or $i = m+1$, then $N(\phi) \bd_i$ is part of $T \bd \Box^n$, while $\phi \in K_{m}'$ by assumption.
\item If $1 \leq i \leq P(\phi) - 1$, then $(N(\phi)\bd_i)_i = (N(\phi)\bd_i)_{i+1} = i$, while entries before position $i$ are the same as in $\phi$. Thus $P(N(\phi)\bd_i) = i-1 < i < P(\phi)$. (Note that this case is vacuous if $P(\phi) = 0$ or $P(\phi) = 1$.)
\item If $i = P(\phi)$, then there is some $j > Q(\phi)$ such that $\phi_j = i + 1$, and this value is unchanged in $N(\phi)$, as it is less than or equal to $q(\phi)$ and not in position $Q(\phi)$. Therefore, in $N(\phi)\bd_i$ this value is lowered to $i$, creating a repetition. As values less than or equal to $i$ are unchanged, we can see that $P(N(\phi)\bd_i) = i - 1 = P(\phi) - 1 < P(\phi)$.
\item If $P(\phi) + 1 = Q(\phi) \leq i \leq q(\phi) - 1$, then we first note that $q(\phi) \geq P(\phi) + 2$. In computing $N(\phi)\bd_i$ from $N(\phi)$, we lower the value in position $Q(\phi)$ from $q(\phi) + 1$ to $q(\phi)$, and we also lower every occurrence of the value $q(\phi)$ to $q(\phi) - 1$. Thus $q(\phi)$ appears only once in $N(\phi)\bd_i$, in position $Q(\phi)$. Since entries less than or equal to $P(\phi)$ are unchanged, we have $P(N(\phi)\bd_i) = P(\phi)$ and $q(N(\phi)\bd_i) = q(\phi)$; thus $N(\phi)\bd_i \in K_{m}^\ast$. (Note that this case is vacuous if $q(\phi) = Q(\phi)$.)
\item If $q(\phi) + 1 \leq i \leq m$, then in computing $N(\phi)\bd_i$ from $N(\phi)$ we do not change any values less than or equal to $q(\phi) + 1$. Thus $P(N(\phi)\bd_i) = P(\phi)$ and $q(N(\phi)\bd_i) = q(\phi) + 1 > q(\phi)$. \qedhere
\end{itemize}
\end{proof}

We will decompose $T \sqcap^n_{i, \varepsilon} \to T \Box^n_{i, \varepsilon}$ as
\[ T \sqcap^n_{i, \varepsilon} \to \Xi^n_i \to T \Box^n_{i, \varepsilon} \]
and show that the two maps are acyclic cofibrations.
In the step $T \sqcap^n_{i, \varepsilon} \to \Xi^n_i$, we will add all simplices of dimension $\leq n-2$ and some of those in dimensions $n-1$ and $n$, while the remaining $(n-1)$- and $n$-simplices will be added in the step $ \Xi^n_i \to T \Box^n_{i, \varepsilon}$.
To define this decomposition, we begin by identifying the simplices that will be added in the latter step.

\begin{Def}\label{omega-def}
For $n \geq 1, 1 \leq i \leq j \leq n$, we let $\omega^{n,i,j}$ denote the $(n-1)$-simplex of $T \Box^n$ defined as follows:

\[
	\omega^{n,i,j}_{k} = \left\{\begin{array}{ll}
	k, & k < i \\
	j, & k = i, j \leq n - 1 \\
	+ \infty, & k = i, j = n \\
	k - 1, & k > i \end{array}\right.
	\]

(Note that $\omega^{n,i,n} = \iota_{\bd_{i,0}}$.) For $j \leq n-1$, we let $\Omega^{n,i,j}$ denote the $n$-simplex obtained from $\omega^{n,i,j}$ by raising the value of the entry in the $(j+1)$-position (i.e. the second occurrence of $j$ in $\omega^{n,i,j}$), and all entries greater than $j$, by 1. We let $\Omega^{n,i,n}$ denote the $n$-simplex obtained from $\omega^{n,i,+}$ by changing the unique occurrence of $+$ in $\omega^{n,i,n}$ to $n$. More explicitly, we may define $\Omega^{n,i,j}$ for all $i \leq j \leq n$ as follows:

\[
	\Omega^{n,i,j}_{k} = \left\{\begin{array}{ll}
	k, & k < i \\
	j, & k = i\\
	k - 1, & i < k < j + 1 \\
	k, & i \geq j + 1 \end{array}\right.
	\]
\end{Def}

We will suppress the superscript $n$ from the notation above where there is no risk of ambiguity, and simply write $\omega^{i,j}$ and $\Omega^{i,j}$.

\begin{egs}
To clarify the definition of $\omega^{i,j}$ and $\Omega^{i,j}$, we state their definitions in the case $n = 5, i = 3$.

\begin{itemize}
\item $\omega^{3,3} = 1\,2\,3\,3\,4; \Omega^{3,3} = 1\,2\,3\,4\,5$.
\item $\omega^{3,4} = 1\,2\,4\,3\,4; \Omega^{3,4} = 1\,2\,4\,3\,5$.
\item $\omega^{3,5} = 1\,2+3\,4; \Omega^{3,5} = 1\,2\,5\,3\,4$.
\end{itemize}

In general, we always have $\Omega^{i,i} = \iota_n$.
\end{egs}

\begin{lem}\label{N-omega}
For $n \geq 2$ and $1 \leq i \leq n - 1$, we have $N(\omega^{i,j}) = \Omega^{i,j+1}$.
\end{lem}

\begin{proof}
From the definition of $\omega^{i,j}$ for $1 \leq j \leq n - 1$, we can see that the value $j$ appears in position $i$ and in position $j + 1$, and that all other entries appear exactly once and in order; thus $Q(\omega^{i,j}) = i$ and $q(\omega^{i,j}) = j$. To construct $N(\omega^{i,j})$, we first raise the value in position $i$ to $j+1$, thereby obtaining $\omega^{i,j+1}$ (or $\omega^{i,j+1}$ with the $+$ in position $i$ replaced by $n$, in the case $j = n$). We then raise every entry which is greater than $j$, aside from this first occurrence of $j+1$, by 1, thereby obtaining $\Omega^{i,j+1}$.
\end{proof}

\begin{lem}\label{omega-face}
For $1 \leq i \leq j \leq n$ as above, $\Omega^{i,j}\bd_{j} = \omega^{i,j}$. Moreover, if $j \geq i + 1$ then $\Omega^{i,j} \bd_{j-1} = \omega^{i,j-1}$.
\end{lem}

\begin{proof}
We begin by considering the first statement. For $j \leq n - 1$, observe that we compute $\Omega^{i,j} \bd_{j}$ by lowering those entries of $\Omega^{i,j}$ which are greater than $j$, and these are precisely the entries of $\omega^{i,j}$ which were  raised to obtain $\Omega^{i,j}$. For $j = n$, we compute $\Omega^{i,n} \bd_{n}$ by replacing the one occurrence of $n$ in $\Omega^{i,n}$ by $+$, thereby obtaining $\omega^{i,n}$.
The second statement is immediate from \cref{N-bijection,N-omega}.
\end{proof}

\begin{lem}\label{omega-linearization}
For $1 \leq i \leq n$, an $(n-1)$-simplex $\phi \colon \Delta^{n-1} \to T \Box^n$ has $\omega^{i,n}$ as a linearization if and only if $\phi = \omega^{i,j}$ for some $j \geq i$.
\end{lem}

\begin{proof}
For $\phi$ to have $\omega^{i,n}$ as a linearization, the entries of $\phi$ other than $\phi_i$ must form a complete substring $\rho$; in other words we must have $\phi_j = j$ for $j < i$ and $\phi_j = j - 1$ for $j > i$. For $\phi_j$ to be replaced by $+$ rather than $-$ in the linearization associated to this substring, we must have $i < \rho_{\phi_i}$; in other words, the other occurrence of $\phi_i$ in $\phi$ must come after position $i$. For this to be the case, we must have $\phi_i \geq i$. We can see that these criteria are satisfied if and only if $\phi = \omega^{i,j}$ for some $j \geq i$.
\end{proof}

\begin{Def}
For $n \geq 1$ and $1 \leq i \leq n$, let $\Xi^n_i$ denote the regular subcomplex of $T\Box^n$ containing all of its non-degenerate simplices except for those of the form $\omega^{i,j}$ or $\Omega^{i,j}$. Let $\bd \Xi^n_i$ denote the intersection of $T \bd \Box^n$ with $\Xi^n_i$, i.e. the regular subcomplex of $T\Box^n$ containing all non-degenerate boundary simplices except for $\omega^{i,n}$. Let $\widehat{\Xi}^n_i$ denote the regular subcomplex of $T\Box^n_{i,0}$ whose underlying simplicial set coincides with that of $\Xi^n_i$, i.e. $\Xi^n_i$ with simplices marked whenever they are marked in $T\Box^n_{i,0}$.
\end{Def}

\begin{prop}\label{open-box-xi-anodyne}
For $n \geq 1, 1 \leq i \leq n$, all maps in the following diagram of inclusions are anodyne:
\[
\begin{tikzcd}
T(\sqcap^n_{i,0})^\flat \arrow[dr] \arrow[d] & \\
\bd \Xi^n_i \arrow[r] & \Xi^n_i \\
\end{tikzcd}
\]
\end{prop}

\begin{proof}
We proceed by induction on $n$. For the base case $n = 1$, all objects in the diagram consist of the single vertex $-$, so all of the inclusions in the diagram are the identity.

Now let $n \geq 2$ and assume the statement holds for $n-1$. We first show that $T(\sqcap^{n}_{i,0})^\flat \hookrightarrow \bd \Xi^n_i$ is anodyne. Observe that the triangulation of the missing face of the cube, $\bd_{i,0}$, consists of all simplices $\phi$ for which $\phi_i = + \infty$; the non-degenerate $m$-simplices in the interior of this face are those in which every value between 1 and $m$ occurs at least once, and there is no $+$ or $-$ in any position besides position $i$. Therefore, to construct $\bd \Xi^n_i$ from $T(\sqcap^{n}_{i,0})^\flat$, we must add all such simplices except for $\omega^{i,n}$. 



Identifying $\bd \Xi^{n-1}_{n-1}$ and $\Xi^{n-1}_{n-1}$ with their images under the face map $T \bd_{i,0} \colon T\Box^{n-1} \to T\Box^n$, we can characterize certain subcomplexes of $T \Box^n$ as follows.

\begin{itemize}
\item $\bd \Xi^{n-1}_{n-1}$ consists of all simplices having a $+$ in position $i$ and a $+$ or $-$ in at least one other position, except for $\bd_{i,0}\omega^{n-1,n-1}$.
\item $\Xi^{n-1}_{n-1}$ consists of all simplices having a $+$ in position $i$, except for $\bd_{i,0}\omega^{n-1,n-1}$ and $\bd_{i,0}\Omega^{n-1,n-1} = \omega^{i,n}$.
\item $T(\sqcap^n_{i,0})^\flat$ consists of all simplices either a $-$ in position $i$, or a $+$ or $-$ in some position other than $i$.
\item $\bd \Xi^n_i$ consists of all simplices having a $+$ or $-$ in any position, other than $\bd_{i,0}\Omega^{n-1,n-1} = \omega^{i,n}$.
\end{itemize}

From this characterization, we can see that $\bd \Xi^{n-1}_{n-1}$ is the intersection of $\Xi^{n-1}_{n-1}$ with $T(\sqcap^n_{i,0})^\flat$, while $\bd \Xi^n_i$ is their union. Thus we have the following pushout square:
	\[
	\begin{tikzcd}
		\bd \Xi^{n-1}_{n-1}
		\arrow [r]
		\arrow [d]
		\pushout &
		 T(\sqcap^n_{i,0})^\flat
		\arrow [d] \\
		\Xi^{n-1}_{n-1}
		\arrow [r] &
		\bd \Xi^n_i
	\end{tikzcd}
	\]
%
Since $\bd \Xi^{n-1}_{n-1} \hookrightarrow \Xi^{n-1}_{n-1}$ is anodyne by the induction hypothesis, this implies $T(\sqcap^n_{i,0})^\flat \hookrightarrow \bd \Xi^n_i$ is anodyne as well.

Next we will show that $\bd \Xi^n_i \hookrightarrow \Xi^n_i$ is anodyne. To do this, we will add to $\bd \Xi^n_i$ every essential simplex of $T \Box^n$, except for those of the form $\omega^{i,j}$ or $\Omega^{i,j}$, via a series of complicial horn fillings. We proceed by induction on dimension; for $1 \leq m \leq n - 1$, let $\Xi^{n,m}_i$ denote the subcomplex of $T \Box^n$ consisting of $\bd \Xi^n_i$ together with all essential simplices of dimension less than $m$ and all normal essential simplices of dimension $m$. As there are no essential simplices of dimension 0 and no normal essential simplices of dimension 1, $\Xi^{n,1}_i = \bd \Xi^n_i$. We will show that for $1 \leq m \leq n-2$, the inclusion $\Xi^{n,m}_i \hookrightarrow \Xi^{n,m+1}_i$ is anodyne.

To construct $\Xi^{n,m+1}_i$ from $\Xi^{n,m}_i$, we must add all non-degenerate simplices of $K_m'$ and $K_{m+1}^\ast$ via complicial horn-filling, marking those which are marked in $T\Box^n$. We proceed by induction on the partial order of \cref{simplex-order}. For the base case, note that the minimal $m$-simplices in this partial order are those which are either on the boundary or normal, thus all minimal $m$-simplices are already present in $\Xi^{n,m}_i$, and this is a regular subcomplex of $T\Box^n$ by definition. Now let $\phi \in K_{m}'$, and suppose that we have added all non-degenerate $m$-simplices less than $\phi$, and marked those which are marked in $T\Box^n$. By \cref{N-order}, this includes all faces of $N(\phi)$ except for $N(\phi)\bd_{q(\phi)} = \phi$ itself; by \cref{N-complicial} these faces define a $q(\phi)$-complicial horn which we can fill to add $N(\phi)$ and $\phi$. By induction, therefore, we can add  $\phi$ and $N(\phi)$ to $\Xi^{n,m}_i$ for all $\phi \in K_{m}'$ via complicial horn-filling; by \cref{N-bijection} these are all the additional simplices of $\Xi^{n,m+1}_i$. Moreover, if $\phi$ is marked in $T\Box^n$, i.e. has no linearizations, then by \cref{N-linearization} the same is true of all other faces of $N(\phi)$. Thus these faces are marked in $\Xi^{n,m}_i$ by the induction hypothesis; therefore, we can mark $\phi$ by taking a pushout of an elementary complicial marking extension. Thus we see that the inclusion $\Xi^{n,m}_i \hookrightarrow \Xi^{n,m+1}_i$ is anodyne.

Composing these anodyne maps, we see that $\bd \Xi^n_i \hookrightarrow \Xi^{n,n-1}_i$ is anodyne. To complete the proof, we must show that $\Xi^{n,n-1}_i \hookrightarrow \Xi^n_i$ is anodyne. To do this, we will add via complicial horn-filling (and mark via complicial marking extension where necessary) all remaining simplices of $\Xi^n_i$ to $\Xi^{n,n-1}_i$  -- namely, the essential simplices of dimensions $n-1$ and $n$, other than those of the form $\omega^{i,j}$ or $\Omega^{i,j}$.

We consider all simplices of $K'_{n-1}$ and $K^\ast_{n}$ not of the forms described above, once again proceeding by induction on the partial order of \cref{simplex-order} on $K_{n-1}$. Again, for our base case we note that all minimal non-degenerate $(n-1)$-simplices, except for $\omega^{i,n}$, are already present in $\Xi^{n,n-1}_i$, and those which are marked in $T\Box^n$ are marked in $\Xi^{n,n-1}$. Now let $\phi \in K_{n-1}'$, not equal to any $\omega^{i,j}$, assume we have added all simplices less than $\phi$ and marked those which are marked in $T\Box^n$, and consider the faces of $N(\phi)$ other than $\phi$ itself. By \cref{N-order}, all of these faces are less than $\phi$. Furthermore, by \cref{omega-linearization}, $\phi$ does not have $\omega^{i,n}$ as a linearization; therefore, none of the faces of $N(\phi)$ have $\omega^{i,n}$ as a linearization by \cref{N-linearization}. This implies that none of these faces are of the form $\omega^{i,j}$ by \cref{omega-linearization}; thus all faces of $N(\phi)$ are present except for $\phi$ itself. Therefore, by \cref{N-complicial}, we have a $q(\phi)$-complicial horn which we can fill to obtain $N(\phi)$ and $\phi$. Moreover, as in the previous case, if $\phi$ is marked in $T\Box^n$ then all other faces of $N(\phi)$ are marked in $T\Box^n$ by \cref{N-linearization}. Thus, in this case we can mark $\phi$ via complicial marking extension.

By induction, then, we can add to $\Xi^{n,n-1}_i$ via complicial horn filling all essential $(n-1)$-simplices $\phi$ not of the form $\omega^{i,j}$, together with their associated $n$-simplices $N(\phi)$, and mark those which are marked in $T\Box^n$ via complicial marking extension. By \cref{N-bijection,N-omega}, we see that we have therefore added all simplices of $\Xi^n_i$, except for the normal simplices $N(\omega^{i,j}) = \Omega^{i,j+1}$ for $i \leq j \leq n-1$, and the lone abnormal $n$-simplex $\iota_n = \Omega^{i,i}$. Thus $\Xi^{n,n-1}_i \hookrightarrow \Xi^n_i$ is anodyne, as a composite of complicial horn fillings and complicial marking extensions.

Thus we see that $T(\sqcap^n_{i,0})^\flat \hookrightarrow \Xi^n_i$ is anodyne, as a composite of anodyne maps.
\end{proof}

Next we consider how the comical marking conditions affect the simplices of $T \Box^n$.

\begin{lem}\label{linearization-marking}
For $n \geq 1$, let $X$ be a marked simplicial set admitting an entire map $Y \to X$, where $Y$ is a regular subcomplex of $T \Box^n$, closed under normalization. Let $\phi$ be a non-degenerate $m$-simplex of $X$. Then:

\begin{itemize}
\item all linearizations of $\phi$ are contained in $X$;
\item if all linearizations of $\phi$ are marked in $X$, then $\phi$ is marked in the precomplicial reflection $\overline{X}$.
\end{itemize}
\end{lem}

\begin{proof}
We proceed by induction on $n$. The case $n = 1$ is trivial, as no simplex of $T\Box^1$ has a linearization other than itself.

Now let $n \geq 2$, and suppose the statement holds for $n-1$. We will prove the statement for $n$ by induction on the partial order of \cref{simplex-order} for non-degenerate $m$-simplices $\phi$. First we consider the case where $\phi$ is minimal. If $\phi$ is contained in $T \bd \Box^n$, then $\phi = \bd_{i,\varepsilon} \psi$ for some $\psi \colon \Delta^{m} \to T\Box^{n-1}$ and some cubical face map $T\bd_{i,\varepsilon} \colon T \Box^{n-1} \to T\Box^n$. By \cref{cube-face-linearization}, the linearizations of $\phi$ are precisely the images under $T\bd_{i,\varepsilon}$ of the linearizations of $\psi$; the stated results thus follow by the induction hypothesis. On the other hand, if $\phi \in K_m^\ast$, then $\phi$ has no linearizations and is marked, so both statements are vacuously true.

Now suppose the statement has been proven for all $m$-simplices less than $\phi$. By assumption, $N(\phi)$ and all of its faces are contained in $X$. As all faces of $N(\phi)$ besides $\phi$ are less than $\phi$ by \cref{N-order}, we can apply the induction hypothesis and \cref{N-linearization} to see that all linearizations of $\phi$ are contained in $X$. Similarly, if all linearizations of $\phi$ are marked in $\overline{X}$, then the induction hypothesis and \cref{N-linearization} show that all faces of $N(\phi)$ besides $\phi$ are marked in $\overline{X}$. As $N(\phi)$ is $q(\phi)$-complicial by \cref{N-complicial}, this implies that we may mark $\phi$ via complicial marking extension; more precisely, there is an entire map $X \to X'$ which is a pushout of an elementary complicial marking extension, such that $\phi$ is marked in $X'$. Now consider the following commuting diagram:
\[
\begin{tikzcd}
X \arrow[d] \arrow[r] & \overline{X} \arrow[d] \\
X' \arrow[r] & \Box^0 \\
\end{tikzcd}
\]
As $X \to X'$ is a pushout of an elementary complicial marking extension, and $\overline{X}$ is pre-complicial, this diagram admits a lift. This lift must be an entire map, as $X \to X'$ and $X \to \overline{X}$ are entire; it thus follows that $\phi$ is marked in $\overline{X}$.
\end{proof}

\begin{cor}\label{Box-marking}
Let $\phi$ be an $m$-simplex of $X \in \{\widehat{\Xi}^n_{i}, T\Box^n_{i,0}\}$. Then:

\begin{itemize}
\item all linearizations of $\phi$ are contained in $X$;
\item if all linearizations of $\phi$ are marked in $X$, then $\phi$ is marked in the precomplicial reflection of $X$.
\end{itemize}
\end{cor}

\begin{proof}
By \cref{linearization-marking}, it suffices to show that $\hxi^n_{i}$ and $T\Box^n_{i,0}$ are closed under normalization. For $T\Box^n_{i,0}$ this is trivial; for $\hxi^n_{i}$ it follows from \cref{N-bijection,N-omega}.
\end{proof}

\begin{lem}\label{Omega-complicial}
For $n \geq 1$ and $1 \leq i \leq j \leq n$, the $n$-simplex $\Omega^{i,j}$ is $j$-complicial in $T\Box^n_{i,0}$.
\end{lem}

\begin{proof}
We must show that any simplex of the form $\Omega^{i,j} \bd_{k_1} \ldots \bd_{k_a}\bd_{r_1} \ldots \bd_{r_b}$, where $k_1 >  \ldots  > k_a \geq j + 2$ and $j-2 \geq r_1 >  \ldots  > r_b$, is marked. (Note that either or both of the strings $k_1, \ldots, k_a$ and $r_1, \ldots, r_b$ may be empty.) Fix a particular face map $\delta$ having this form, and consider the face $\Omega^{i,j} \delta$.

We first note that when represented as a string, $\Omega^{i,j}$ contains each of the values $1$ through $n$ exactly once. In particular, $(\Omega^{i,j})_i$ is the unique entry of $\Omega^{i,j}$ having the value $j$, and if $j \leq n - 1$ then $(\Omega^{i,j})_{j+1}$ is the unique entry having the value $j+1$. As the maps $\bd_{k_t}$ only lower entries greater than $j + 2$, this will still be true of the corresponding entries in $\Omega^{i,j} \bd_{k_1} \ldots \bd_{k_a}$. 

Now  let $1 \leq t \leq b$, and suppose that in $\Omega^{i,j} \bd_{k_1} \ldots \bd_{k_a}\bd_{r_1} \ldots \bd_{r_{t-1}}$, the entries in positions $i$ and $j+1$ are the unique entries having the values $j - t + 1$ and $j - t + 2$, respectively. Since $r_t \leq j - t - 1$, the face map $\bd_{r_t}$ lowers these entries, along with any entries having the value $j - t$. Thus the entries in positions $i$ and $j + 1$ are the only entries of $\Omega^{i,j} \bd_{k_1} \ldots \bd_{k_a}\bd_{r_1} \ldots \bd_{r_t}$ having the values $j - (t + 1) + 1$ and $j - (t + 1) + 2$, respectively. By induction, we see that the entries in positions $i$ and $j$ of $\Omega^{i,j} \delta$ are the unique entries having the values $j - b$ and $j - b + 1$, respectively. 

Let $\rho$ be a complete substring of $\Omega^{i,j} \delta$; for notational convenience let $j' = j - b$ and $m = n - a - b$. The discussion above shows that we must have $\rho_{j'} = i$ and $\rho_{j'+1} = j + 1$. (As in the statement of \cref{comical-triangulation-marking}, we interpret $\rho_0$ and $\rho_{m+1}$ to be 0 and $m+1$, respectively.) Now consider $k$ such that $\rho_{j'-1} < k < \rho_{j'+1}, k \neq i$.  Observe that for such $k$ we have $(\Omega^{i,j})_k  < i  \leq j = (\Omega^{i,j})_i$.


By \cref{face-preserve-order}, this implies $(\Omega^{i,j} \delta)_k \leq (\Omega^{i,j} \delta)_i = j'$; as $(\Omega^{i,j} \delta)_i$ is the unique entry with the value $j'$, we in fact have $(\Omega^{i,j} \delta)_k \leq j' - 1$. If $j' = 1$ then this implies $(\Omega^{i,j} \delta)_k = - \infty$. Otherwise, we see that $\rho_{(\Omega^{i,j} \delta)_k} \leq \rho_{j'-1} = w < k$. Either way, we have $(\Omega^{i,j} \delta)^\rho_k = - \infty$. Thus $\Omega^{i,j} \delta$ is marked by \cref{comical-triangulation-marking}.
\end{proof}

\begin{prop}\label{xi-cube-anodyne}
For $n \geq 1$ and $1 \leq i \leq n$, the inclusion $\hxi^n_i \hookrightarrow T\Box^n_{i,0}$ is a trivial cofibration. 
\end{prop}

\begin{proof}
Let $(T\Box^n_{i,0})^\dag$ denote the marked simplicial set with underlying simplicial set $T\Box^n$, and a simplex $\phi$ marked if and only if all of its linearizations are marked in $T\Box^n_{i,0}$; define $(\hxi^n_i)^\dag$ similarly. We have a commuting diagram:
	\[
	\begin{tikzcd}
		\hxi^n_i 
		\arrow [r]
		\arrow [d]
          &
		 (\hxi^n_i)^\dag
		\arrow [d] \\
		T\Box^n_{i,0} 
		\arrow [r] &
		(T\Box^n_{i,0})^\dag
	\end{tikzcd}
	\]
To see that the two horizontal maps exist, recall that by \cref{linear-simplex-marked}, the non-degenerate marked simplices of $T \Box^n_{i,\varepsilon}$ consist of the non-degenerate marked simplices of $T\Box^n$ (which have no linearizations) plus the linear simplices associated to marked faces of $\Box^n_{i,\varepsilon}$ (which have only themselves as linearizations); thus every marked simplex of $T \Box^n_{i,0}$ is marked in $(T \Box^n_{i,0})^\dagger$, and similarly for $\hxi^n_i$ and $(\hxi^n_i)^\dagger$. The two horizontal maps are trivial cofibrations by \cref{more-markings,linearization-marking}. Therefore, to prove the stated result it suffices to prove that $(\hxi^n_i)^\dag \hookrightarrow (T\Box^n_{i,0})^\dag$ is anodyne.

For $i \leq j \leq n + 1$, let $(\hxi^n_{i,j})^\dag$ denote the regular subcomplex of $(T\Box^n_{i,0})^\dag$ consisting of $(\hxi^n_{i})^\dag$ together with all simplices of the form $\omega^{i,j'}$ or $\Omega^{i,j'}$ for $i \leq j' < j$. We can see that $(\hxi^n_{i,i})^\dag = (\hxi^n_i)^\dag$, while $(\hxi^n_{i,n+1})^\dag  = (T\Box^n)^\dag$; thus it suffices to show that each map $(\hxi^n_{i,j})^\dag \hookrightarrow (\hxi^n_{i,j+1})^\dag$ for $i \leq j \leq n$ is anodyne.

For the case $j = i$, we will show that we can add $\Omega^{i,i} = \iota_n$ and $\omega^{i,i}$ to $(\hxi^{n}_i)^\dag$ via complicial horn-filling. Observe that for all $1 \leq k \leq n$ we have $\iota_n \bd_k = \Omega^{k,k} \bd_k = \omega^{k,k}$ by \cref{omega-face}. From the definition of $\omega^{i,j}$ it is clear that the simplices $\omega^{k,k}$ are all distinct, thus all of these faces besides $\omega^{i,i}$ are present in $(\hxi^{n}_i)^\dag$. Furthermore, we have $\iota_n \bd_0 = -1\,2\,...\,(n-1)$, which is contained in $T \sqcap^n_{i,0} \subseteq (\hxi^n_i)^\dag$. By \cref{Omega-complicial}, these faces define an $i$-complicial horn in $(\hxi^n_i)^\dag$, which we can fill to obtain $\iota_n$ and its missing face $\omega^{i,i}$. Thus the inclusion $(\hxi^n_i)^\dag = (\hxi^n_{i,i})^\dag \hookrightarrow (\hxi^n_{i,i+1})^\dag$ is anodyne.

Now consider the case $j \geq i + 1$. By \cref{omega-face} we have $\Omega^{i,j} \bd_{j-1} = \omega^{i,j-1}$, while $\Omega^{i,j} \bd_j = \omega^{i,j}$. Furthermore, by \cref{N-linearization}, these are the only faces of $\Omega^{i,j}$ having $\omega^{i,n}$ as a linearization; therefore, by \cref{omega-linearization}, no other face of $\Omega^{i,j}$ is of the form $\omega^{i,j'}$ for any $j'$. Thus we see that all faces of $\Omega^{i,j}$ besides $\omega^{i,j}$ are present in $(\hxi^n_{i,j})^\dag$. By \cref{Omega-complicial}, we therefore have a $j$-complicial horn in $(\hxi^n_{i,j})^\dag$ which we can fill to obtain $(\hxi^n_{i,j+1})^\dag$. Thus the inclusion $(\hxi^n_{i,j})^\dag \hookrightarrow (\hxi^n_{i,j+1})^\dag$ is anodyne.

Thus we see that $(\hxi^n_i)^\dag \hookrightarrow (T\Box^n_{i,0})^\dag$ is anodyne, as a composite of anodyne maps.
\end{proof}

\begin{cor}\label{open-box-cube-anodyne}
For $n \geq 1, l \leq i \leq n$, the inclusion $T\sqcap^n_{i,0} \hookrightarrow T\Box^n_{i,0}$ is a trivial cofibration.
\end{cor}

\begin{proof}
The inclusion $T\sqcap^n_{i,0} \hookrightarrow \hxi^n_i$ is anodyne by \cref{open-box-xi-anodyne}, as it is a pushout of $\fcap{n}{i} \hookrightarrow \Xi^n_i$. The inclusion $\hxi^n_i \hookrightarrow T\Box^n_{i,0}$ is a trivial cofibration by \cref{xi-cube-anodyne}. Thus $T\sqcap^n_{i,0} \hookrightarrow T\Box^n_{i,0}$ is a trivial cofibration as a composite of trivial cofibrations.
\end{proof}

\begin{prop}\label{marking-extension-anodyne}
For $n \geq 2,$ $1 \leq i \leq n$, the map $T(\Box^n_{i,0})' \to  T\tau_{n-2} \Box^n_{i,0}$ is a trivial cofibration.
\end{prop}

\begin{proof}
As in the proof of \cref{xi-cube-anodyne}, let $(T(\Box^n_{i,0})')^\dag$ denote the marked simplicial set obtained from $T(\Box^n_{i,0})'$ by marking all simplices whose linearizations are all marked in $T(\Box^n_{i,0})'$, and define $(T\tau_{n-2} \Box^n_{i,0})^\dag$ similarly. Note that \cref{omega-linearization} shows that the only unmarked essential simplices of $(T(\Box^n_{i,0})')^\dag$ are those of the form $\omega^{i,j}$, while all $(n-1)$-simplices of $(T\tau_{n-2} \Box^n_{i,0})^\dag$ are marked (in fact, $(T\tau_{n-2} \Box^n_{i,0})^\dag = \tau_{n-2} T \Box^n_{i,0}$. Once again, we have a commuting diagram:
	\[
	\begin{tikzcd}
		T(\Box^n_{i,0})' 
		\arrow [r]
		\arrow [d]
          &
		 (T(\Box^n_{i,0})')^\dag
		\arrow [d] \\
		T\tau_{n-2} \Box^n_{i,0}  
		\arrow [r] &
		(T\tau_{n-2} \Box^n_{i,0})^\dag 
	\end{tikzcd}
	\]
As in the proof of \cref{xi-cube-anodyne}, the existence of the horizontal maps follows from \cref{linear-simplex-marked}, and they are trivial cofibrations by \cref{more-markings,linearization-marking}. Thus it suffices to show that $(T(\Box^n_{i,0})')^\dag \hookrightarrow (T\tau_{n-2} \Box^n_{i,0})^\dag$ is anodyne. To this end, we will show that we may mark every simplex $\omega^{i,j}$ via complicial marking extensions, proceding by induction on $j$. 


For the base case $j = i$, recall that $\Omega^{i,i} = \iota_n$ is $i$-complicial by \cref{Omega-complicial}. We will show that $\iota_n \bd_{i-1}$, and $\iota_n \bd_{i+1}$ in the case $i \neq n$, are both marked in $(T(\Box^n_{i,0})')^\dag$. We begin with $\iota_n \bd_{i-1}$. First consider the case $i = 1$; then $\iota_n \bd_0 = -\,1\,\ldots\,(n-1) = \iota_{\bd_{1,1}}$. This $(n-1)$-simplex is contained in $T\sqcap^n_{1,0}$, hence it is marked. Next consider the case $i \geq 2$; then $\iota_n \bd_{i} = \Omega^{i-1,i-1} \bd_{i-1} = \omega^{i-1,i-1}$. The only repeated entry of this simplex is $i-1$, which appears in positions $i-1$ and $i$. Thus every entry besides these two must appear in any complete substring of $\rho$ of this simplex. Such a complete substring may have $\rho_{i-1}$ equal to either $i-1$ or $i$. In the former case, the associated linearization is $\iota_{\bd_{i,1}}$, while in the latter case it is $\iota_{\bd_{i-1,0}}$. Either way, it is an $(n-1)$-simplex of $T \sqcap^n_{i,0}$, and is therefore marked in $(T(\Box^n_{i,0})')^\dag$. Thus both linearizations of $\iota_n \bd_{i-1}$ are marked, hence so is $\iota_n \bd_{i-1}$ itself.

Next, assume that $i \neq n$, and consider $\Omega^{i,i} \bd_{i+1} = \iota_n \bd_{i+1}$. Similarly to the previous case, we observe that $\iota_n \bd_{i+1,i+1} = \omega^{i+1,i+1}$. If $i = n - 1$ then this is $\omega^{n,n} = \iota_{\bd_{n,0}}$, hence it is marked as an $(n-1)$-simplex of $T \sqcap^n_{n-1,0}$. Otherwise, an argument similar to the above shows that it has two linearizations, namely $\iota_{\bd_{i+1,0}}$ and $\iota_{\bd_{i+2,1}}$. Again, both of these are marked as $(n-1)$-simplices of $T \sqcap^n_{n-1,0}$, hence $\iota_n \bd_{i+1}$ is marked. Thus we see that the $n$-simplex $\iota_n \colon \Delta^n \to T(\Box^n_{i,0})'$ factors through $(\Delta^{n})'_{i}$, hence  we may mark its $i$-face $\omega^{i,i}$ via a complicial marking extension.

Now let $i + 1 \leq j \leq n$, and assume that we have marked $\omega^{i,j-1}$. Once again, \cref{Omega-complicial} shows that $\Omega^{i,j}$ is $j$-complicial, so we will show that the faces $\Omega^{i,j} \bd_{j-1}$ and $\Omega^{i,j} \bd_{j+1}$ (in the case $j \neq n$) are marked. For $\bd_{j-1}$, recall that $\Omega^{i,j} \bd_{j-1} =  \omega^{i,j-1}$ by \cref{omega-face}, thus it is marked by the induction hypothesis. For $\bd_{j+1}$, recall that $\Omega^{i,j} = N(\omega^{i,j-1})$ by \cref{N-omega}. Since $q(\omega^{i,j-1}) = j-1$, this implies that $\Omega^{i,j} \bd_{j+1}$ has no linearizations by \cref{N-linearization}, and is therefore marked. Once again, therefore, we see that $\Omega^{i,j} \colon \Delta^n \to T(\Box^n_{i,0})'$ factors through $(\Delta^n_j)'$, thus we may mark its $j$-face $\omega^{i,j}$ via complicial marking extension.

By induction, we see that we may mark all simplices $\omega^{i,j}$ via complicial marking extensions, thus the inclusion $(T(\Box^n_{i,0})')^\dag \hookrightarrow (T \tau_{n-2} \Box^n_{i,0})^\dag$ is anodyne. 
\end{proof}

We are now able to prove the main result of this section.

\begin{thm}\label{T-Quillen}
The adjunction $T : \cSet^+ \rightleftarrows \sSet^+ : U$ is Quillen, where $\cSet^+$ is equipped with any of the (saturated) (n-trivial) comical model structures, and $\sSet^+$ is equipped with the corresponding complicial model structure.
\end{thm}

\begin{proof}
We must show that $T$ sends the following maps to trivial cofibrations in $\sSet^+$:

\begin{enumerate}
\item\label{QA-open-box} Comical open box fillings $\sqcap^n_{i,\varepsilon} \hookrightarrow \Box^n_{i,\varepsilon}$;
\item\label{QA-extension} Comical marking extensions $(\Box^n_{i,\varepsilon})' \to \tau_{n-2} \Box^n_{i,\varepsilon}$;
\item\label{QA-Rezk} Rezk maps, in the case where the model structures are saturated;
\item\label{QA-marking} $k$-markings $\Delta^k \to \widetilde{\Delta}^k$ for $k > n$ in the case where the model structures are $n$-trivial for some $n \geq 0$.
\end{enumerate}

By \cref{simplicial-op-equiv,cubical-op-equiv,T-op}, it suffices to show \cref{QA-open-box,QA-extension} in the case $\varepsilon = 0$. For item \ref{QA-open-box} this is \cref{open-box-cube-anodyne}, while for item \ref{QA-extension} this is \cref{marking-extension-anodyne}. For \cref{QA-Rezk}, it is easy to see that $T$ sends each elementary Rezk map to a pushout of the simplicial elementary  Rezk map; the general result then follows from \cref{gray-homotopical,T-monoidal}. For \cref{QA-marking}, we may observe that $T$ sends each (cubical) $k$-marker to a pushout of the (simplicial) $k$-marker, as in \cite[Thm.~7.2]{campion-kapulkin-maehara}.
\end{proof}



\section{Cones in marked cubical sets} \label{sec:cones}

In this section, we extend the theory of cones in cubical sets, developed in \cite[Sec.~5]{doherty-kapulkin-lindsey-sattler}, to the marked case. 
We will develop our theory using what are referred to in \cite{doherty-kapulkin-lindsey-sattler} as \emph{left negative cones}; thus all of our constructions will differ in the value of $\varepsilon$ from their counterparts in that reference, but the combinatorics involved will be effectively identical. In particular, we will regularly cite \cite{doherty-kapulkin-lindsey-sattler} for results whose statements and proofs are the same in the marked and unmarked cases.

In \cref{sec:cones-basic-theory}, we define cones and review some of their combinatorial properties, following \cite{doherty-kapulkin-lindsey-sattler}.
Then in \cref{sec:Q}, we define the functor $Q \colon \msSet \to \mcSet$ both treating it as an instance of the cone construction (cf.~\cite{doherty-kapulkin-lindsey-sattler}) and giving an explicit description in terms of quotients of standard cubes (cf.~\cite{kapulkin-lindsey-wong}).
Finally, in \cref{sec:Q-is-Quillen}, we prove that $Q$ is a left Quillen functor from the model structure for ($n$-trivial, saturated) complicial sets to (resp.~$n$-trivial, saturated) comical sets.

The basic theory of cones developed in \cref{sec:cones-basic-theory} can be formulated in all of our marked cubical set categories, save that those results which concern connections, such as \cref{ConeFaceDeg}.(\ref{Low0ConOfCone}), can only be formulated in categories in which the relevant connections exist. The constructions and results of \cref{sec:Q,sec:Q-is-Quillen} can only be formulated in $\cSet^+_0, \cSet^+_1$, and $\cSet^+$, as the definition of the functor $Q$ requires at least one kind of connection. For the sake of concreteness, we will work in $\cSet^+$, but our proofs will only require positive connections. Thus our proofs work verbatim in $\cSet^+_1$, and may be easily dualized to $\cSet^+_0$ by interchanging the roles of positive and negative faces and connections. (In terms of the definitions given in \cite{doherty-kapulkin-lindsey-sattler}, this means substituting $C_{L,1}$ for $C_{L,0}$ and $Q_{L,1}$ for $Q_{L,0}$.)

\subsection{Basic theory of cones} \label{sec:cones-basic-theory}

We begin by defining cones and recalling some of their basic properties from \cite{doherty-kapulkin-lindsey-sattler}.

\begin{Def}\label{left-negative-def}
The \emph{cone functor} $C \colon \mcSet \to \mcSet$ is defined by the following pushout diagram in $\mathrm{End} \, \mcSet$:
	\[
	\begin{tikzcd}
		\id
		\arrow [r]
		\arrow [d, swap, "\bd_{1,0} \otimes -"]
		\pushout
          &
		 \Box^{0}
		\arrow [d] \\
		\Box^{1} \otimes -  
		\arrow [r] &
		C
	\end{tikzcd}
	\]

For $m, n \geq 0$, the \emph{standard $(m,n)$-cone} is the object $C^{m,n} = C^{n}\Box^m$. In particular, for all $m$, $C^{m,0} = \Box^m$. We refer to the natural map $\Box^0 \Rightarrow C$ appearing in this diagram as the \emph{cone point}. The cubical set obtained by marking the image of the $(m+n)$-simplex $\id_{[1]^{m+n}}$ in $C^{m,n}$ is denoted $\mC^{m,n}$.
\end{Def}

Note that the marked cubical set $C^{m,n}$ defined above is the minimal marking of the cubical set $C^{m,n}_{L,0}$ defined in \cite{doherty-kapulkin-lindsey-sattler}.


\begin{lem}[{\cite[Lem.~5.2]{doherty-kapulkin-lindsey-sattler}}]\label{Qcone}
For all $n \geq 1$, $C^{0,n} \cong C^{1,n-1}$. \qed
\end{lem}

For $X \in \mcSet$, and \emph{$(m,n)$-cone in $X$} is a map $C^{m,n} \to X$. We may observe that each standard cone $C^{m,n}$ is a quotient of $\Box^{m+n}$; by pre-composition with the quotient map, we can identify each cone in $X$ with a unique cube of $X$. 

We now consider the combinatorics of the faces of $C^{m,n}$. Recall that $\Box^{m+n}_{k}$ is the set of maps $[1]^{k} \to [1]^{m+n}$ in the box category $\Box$; thus we may write such a $k$-cube $f$ as $(f_{1},...,f_{m+n})$ where each $f_{i}$ is a map $[1]^{k} \to [1]$.

\begin{lem}[{\cite[Lem.~5.8]{doherty-kapulkin-lindsey-sattler}}]\label{ConeDesc}
For all $m, n \geq 0, C^{m,n}$ is the quotient of $\Box^{m+n}$ obtained by identifying two $k$-cubes $f, g$ if there exists $j$ with $1 \leq j \leq n$ such that $f_{i} = g_{i}$ for $i \leq j$ and $f_{j} = g_{j} = \mathrm{const}_{0}$ (the constant map $[1]^{k} \to [1]$ with value 0). \qed
\end{lem}

\begin{lem}\label{cone-desc-faces}
A face $\delta \colon \Box^k \to \Box^{m+n}$ is mapped to a degenerate cube by the quotient map $\Box^{m+n} \to C^{m,n}$ if and only if there is some $1 \leq i \leq n$ such that both of the following conditions are satisfied: 
\begin{itemize}
\item the standard form of $\delta$ contains $\bd_{i,0}$;
\item for some $i < j \leq m + n$, the standard form of $\delta$ does not contain any map $\bd_{j,\varepsilon}$.
\end{itemize}
\end{lem}

\begin{proof}
We fix $m$, and proceed by induction on $n$. For $n = 0$, we have $C^{m,0} = \Box^m$, and the quotient map is the identity; as there are no values $i$ for which the statement holds, it is therefore vacuously true.

Now suppose that the statement is proven for $C^{m,n}$ and consider $C^{m,n+1}$. This cubical set is constructed via the following pushout:
	\[
	\begin{tikzcd}
		C^{m,n}
		\arrow [r]
		\arrow [d, swap, "\bd_{1,0}"]
		\pushout
          &
		\Box^0 
		\arrow [d] \\
		\Box^1 \otimes C^{m,n}
		\arrow [r] &
		 C^{m,n+1}
	\end{tikzcd}
	\]
Because the functor $\Box^1 \otimes -$ preserves colimits, $\Box^1 \otimes C^{m,n}$ is a quotient of $\Box^1 \otimes \Box^{m+n} \cong \Box^{m+n+1}$, with the quotient map given by applying $\Box^1 \otimes -$ to the quotient map $\Box^{m+n} \to C^{m,n}$. 

The degenerate cubes of $\Box^1 \otimes C^{m,n}$ are those corresponding to pairs $(x,y)$ where either $x$ is a degenerate cube of $\Box^1$ or $y$ is a degenerate cube of $C^{m,n}$. 
Therefore, by the inductive hypothesis, the non-degenerate cubes of $\Box^1 \otimes \Box^{m+n}$ which are mapped to degenerate cubes of $\Box^1 \otimes C^{m,n}$ are those of the form $(x,y)$, where both $x$ and $y$ are non-degenerate, and $y$ satisfies the criteria given in the statement when considered as a cube $y \colon \Box^k \to \Box^{m+n}$ for some $1 \leq i \leq n$. 
Under the isomorphism $\Box^1 \otimes \Box^{m+n} \cong \Box^{m+n+1}$, such cubes correspond to faces $\delta \colon \Box^{k'} \to \Box^{m+n+1}$ such that the conditions in the statement of the theorem are satisfied for some $2 \leq i \leq n+1$.

The effect of the quotient map $\Box^1 \otimes C^{m,n} \to C^{m, n+1}$ is precisely to map the subcomplex $\{0\} \otimes C^{m, n}$ onto the cone point.
This subcomplex consists of the images under the quotient map $\Box^{m+n+1} \to \Box^1 \otimes C^{m,n}$ of the cubes of $\Box^{m+n+1}$ whose standard forms contain $\bd_{1,0}$.
Thus every cube of $\Box^{m+n+1}$ of positive dimension whose standard form contains $\bd_{1,0}$ is mapped to a degenerate cube by the composite quotient map $\Box^{m+n+1} \to \Box^1 \otimes C^{m,n} \to C^{m,n+1}$.
The cubes of $\Box^{m+n+1}$ of positive dimension whose standard forms contain $\bd_{1,0}$ are precisely the faces satisfying the criteria of the statement for $i = 1$ when considered as faces $\Box^k \to \Box^{m+n+1}$.
\end{proof}

Next we recall some results on cones from \cite{doherty-kapulkin-lindsey-sattler} which will be of use in our proof.

\begin{lem}[{\cite[Cor.~5.9]{doherty-kapulkin-lindsey-sattler}}]\label{ConeWLOG}
For $k \leq n$, the quotient map $\Box^{m+n} \to C^{m,n}$ factors through $C^{m+k,n-k}$. In particular, if $x \colon \Box^{m+n} \to X$ is an $(m,n)$-cone, then $x$ is also an $(m+k,n-k)$-cone for all $k \leq n$. \qed
\end{lem}

\begin{lem}[{\cite[Lem.~5.10]{doherty-kapulkin-lindsey-sattler}}]\label{FaceIso}
For $i \leq n$, the image of the composite map $\Box^{m+n-1} \xrightarrow{\partial_{i,0}} \Box^{m+n} \to C^{m,n}$ is isomorphic to $C^{m,n-1}$. For $i \geq n + 1$, $\varepsilon \in \{0,1\}$, the image of the composite map $\Box^{m+n-1} \xrightarrow{\partial_{i,\varepsilon}} \Box^{m+n} \to C^{m,n}$ is isomorphic to $C^{m-1,n}$. \qed
\end{lem}

\begin{lem}[{\cite[Lem.~5.11]{doherty-kapulkin-lindsey-sattler}}]\label{ConeFaceDeg}
Let $x$ be an $(m,n)$-cone in a cubical set $X$. Then:

\begin{enumerate}
\item\label{Low0FaceOfCone} for $1 \leq i \leq n$, $x\bd_{i,1}$ is an $(m,n-1)$-cone;
\item\label{High0FaceOfCone} for $n+1 \leq i \leq m + n$, $x\bd_{i,1}$ is an $(m-1,n)$-cone;
\item\label{1FaceOfCone} if $m \geq 1$, then for all $i$, $x\bd_{i,0}$ is an $(m-1,n)$-cone;
\item\label{DegenOfCone} for $n+1 \leq i \leq m + n + 1$, $x\sigma_{i} $ is an $(m+1,n)$-cone;
\item\label{Low0ConOfCone} for $1 \leq i \leq n$, $x\gamma_{i,1}$ is an $(m,n+1)$-cone;
\item\label{HighConOfCone} for $n+1 \leq i \leq m + n$, $x\gamma_{i,\varepsilon}$ is an $(m+1,n)$-cone. \qed
\end{enumerate}
\end{lem} 

Recall that a ``standard form of a cube $x$'' was introduced after \cref{cube-standard-form}.
With that, we may provide further analysis of standard forms of cones.

\begin{lem}[{\cite[Lem.~5.16]{doherty-kapulkin-lindsey-sattler}}]\label{sa1}
Let $m \geq 1$, and let $x \colon C^{m,n} \to X$ be a degenerate $(m,n)$-cone.
\begin{enumerate}
\item\label{sa1-degen} If the standard form of $x$ is $z\sigma_{a_{p}}$, then $a_{p} \geq n + 1$.

\item\label{sa1-con} If the standard form of $x$ is $z\gamma_{b_{q},0}$, then $b_{q} \geq n+1$. \qed
\end{enumerate}
\end{lem}

\subsection{The functor $Q$}\label{sec:Q}
In this section, we construct the \emph{cubification} functor $Q \colon \msSet \to \mcSet$.
Later, we will exhibit this functor to be a homotopical inverse to the triangulation functor $T$.

We will use the cosimplcial object $Q_{L,0}$ from \cite[Prop.~6.5]{doherty-kapulkin-lindsey-sattler} combined with the involution $(-)^\op \colon \msSet \to \msSet$.
More explicitly, for $n \geq 0$, $Q^n$ is the cubical set $C^{0,n}$.
Using \cref{ConeDesc}, we can obtain the following alternative description of the objects $Q^n$, relating this definition of $Q$ to that given in \cite{kapulkin-lindsey-wong}. 

\begin{prop}
	For $n \geq 0$, $Q^n$ is given by the pushout square
	\[
	\begin{tikzcd}
		\displaystyle\coprod_{1 \le i \le n} \cube{i-1}\otimes\cube{n-i}
		\arrow [r]
		\arrow [d]
		\pushout &
		\cube n
		\arrow [d] \\
		\displaystyle\coprod_{1 \le i \le n} \cube{i-1}
		\arrow [r] &
		Q^n
	\end{tikzcd}
	\]
	where the upper horizontal map restricts to $\face i 0$ on the $i$-th summand, and the left vertical map is a coproduct of the projections $\cube{i-1}\otimes\cube{n-i} \to \cube{i-1}\otimes\cube{0}\cong\cube{i-1}$. \qed
\end{prop}

For example, $Q^0$ and $Q^1$ are just $\Box^0$ and $\Box^1$ respectively.
The first non-trivial example $Q^2$ is $\Box^2$ with $\bd_{1,0}$ collapsed to a $0$-cube, which may be visualised as below left.
Below right is the same picture with the degenerate $1$-cube shrunk so that it better resembles the $2$-simplex:
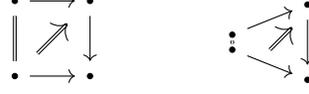
\[
\begin{tikzpicture}
\filldraw
(0,0) circle [radius = 1pt]
(1,0) circle [radius = 1pt]
(0,1) circle [radius = 1pt]
(1,1) circle [radius = 1pt];

\draw[double] (0,0.2) -- (0,0.8);
\draw[->] (1,0.8) -- (1,0.2);
\draw[->] (0.2,0) -- (0.8,0);
\draw[->] (0.2,1) -- (0.8,1);

\draw[->, double] (0.3,0.3) -- (0.7,0.7);
\end{tikzpicture}
\hspace{5em}
\begin{tikzpicture}
	\filldraw
	(0,0.4) circle [radius = 1pt]
	(1,0) circle [radius = 1pt]
	(0,0.6) circle [radius = 1pt]
	(1,1) circle [radius = 1pt];
	
	\draw[double] (0,0.48) -- (0,0.52);
	\draw[->] (1,0.8) -- (1,0.2);
	\draw[->] (0.2,0.32) -- (0.8,0.08);
	\draw[->] (0.2,0.68) -- (0.8,0.92);
	
	\draw[->, double] (0.5,0.4) -- (0.8,0.7);
\end{tikzpicture}
\]
Similarly, $Q^3$ may be visualised as follows (with degenerate $2$-cubes shaded since they are marked):
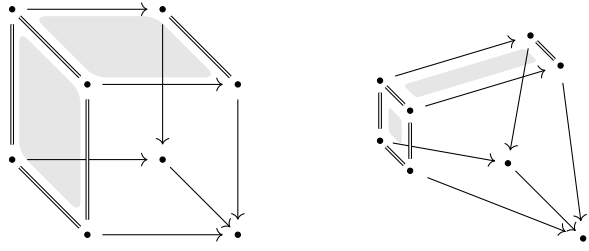
\[
\begin{tikzpicture}
	\begin{scope}[blend group = multiply]
		\filldraw[shaded, rounded corners] (0.3,2.9) -- (1.9,2.9) -- (2.7,2.1) -- (1.1,2.1) -- cycle; 
		\filldraw[shaded, rounded corners] (0.1,1.1) -- (0.9,0.3) -- (0.9,1.9) -- (0.1,2.7) -- cycle; 
	\end{scope}
	
	\filldraw
	(1,0) circle [radius = 1pt]	
	(3,0) circle [radius = 1pt]	
	(1,2) circle [radius = 1pt]	
	(3,2) circle [radius = 1pt]	
	(0,1) circle [radius = 1pt]	
	(2,1) circle [radius = 1pt]	
	(0,3) circle [radius = 1pt]	
	(2,3) circle [radius = 1pt];
	
	\draw[->] (1.2,0) -- (2.8,0);
	\draw[->] (1.2,2) -- (2.8,2);
	\draw[->] (0.2,1) -- (1.8,1);
	\draw[->] (0.2,3) -- (1.8,3);
	
	\draw[double] (0,2.8) -- (0,1.2);
	\draw[->] (2,2.8) -- (2,1.2);
	\draw[double] (1,1.8) -- (1,0.2);
	\draw[->] (3,1.8) -- (3,0.2);
	
	\draw[double] (0.1,0.9) -- (0.9,0.1);
	\draw[->] (2.1,0.9) -- (2.9,0.1);
	\draw[double] (0.1,2.9) -- (0.9,2.1);
	\draw[double] (2.1,2.9) -- (2.9,2.1);
\end{tikzpicture}
\hspace{5em}
\begin{tikzpicture}
	\begin{scope}[blend group = multiply]
		\filldraw[shaded, rounded corners = 1pt] (0.62,2.04) -- (0.78,1.88) -- (2.38,2.36) -- (2.22,2.52) -- cycle; 
		\filldraw[shaded, rounded corners = 1pt] (0.42,1.74) -- (0.42,1.42) -- (0.58,1.26) -- (0.58,1.58) -- cycle; 
	\end{scope}
	
	\filldraw
	(0.3,2.1) circle [radius = 1pt]	
	(0.3,1.3) circle [radius = 1pt]	
	(0.7,1.7) circle [radius = 1pt]	
	(0.7,0.9) circle [radius = 1pt]	%
	(2.3,2.7) circle [radius = 1pt]	
	(2,1) circle [radius = 1pt]
	(2.7,2.3) circle [radius = 1pt]	
	(3,0) circle [radius = 1pt]	;
	
	\draw[->] (0.5,2.16) -- (2.1,2.64); %
	\draw[->] (0.47,1.27) -- (1.83,1.03);
	\draw[->] (0.9,1.76) -- (2.5,2.24); %
	\draw[->] (0.93,0.81) -- (2.77,0.09);
	
	\draw[double]
	(0.3,1.46) -- (0.3,1.94)
	(0.7,1.06) -- (0.7,1.54)
	(0.62,1.78) -- (0.38,2.02)
	(0.62,0.98) -- (0.38,1.22)
	(2.38,2.62) -- (2.62,2.38);
	
	\draw[->] (2.27,2.53) --(2.03,1.17);
	\draw[->] (2.73,2.07) -- (2.97,0.23);
	\draw[->] (2.1,0.9) -- (2.9,0.1);
\end{tikzpicture}
\]

Thus each $Q^n$ may be regarded as a quotient of $\cube n$.
Then the map $Q\sface i \colon Q^{n-1} \to Q^n$ is induced by:
\begin{itemize}
	\item $\face{i+1}{1} \colon \cube{n-1} \to \cube{n}$ if $i<n$; and
	\item $\face n 0 \colon \cube{n-1} \to \cube n$ if $i=n$,
\end{itemize}
whereas the map $Q(\sigma_i) \colon Q^{n+1} \to Q^n$ is induced by:
\begin{itemize}
	\item $\gamma_{i+1,1} \colon \cube{n+1} \to \cube n$ if $i<n$; and
	\item $\sigma_{n+1} \colon \cube{n+1} \to \cube n$ if $i=n$.
\end{itemize}

Recall that a face map $\delta \colon \Delta^{n-m} \to \Delta^n$ in the simplex category $\Delta$ admits the following unique decomposition, referred to as its \emph{standard form}, as $\delta = \partial^{n}_{i_m}\dots\partial^{n - m}_{i_1} \colon \Delta^{n-m} \to \Delta^n$ with $i_m>\dots>i_1$.
With that, we may observe that $Q$ ``preserves standard forms'' of faces in the following sense.

\begin{lem} \label{Q_NF}
  Let $\delta = \partial^{n}_{i_m}\dots\partial^{n-m}_{i_1} \colon \Delta^{n-m} \to \Delta^{n}$ with $i_m>\dots>i_1$ be the standard form of a face map $\delta$ in the simplex category.
  Let $\partial_{k_j, \varepsilon_j} \colon \Box^{n-m+j-1} \to \Box^{n-m+j}$ be the map inducing $Q\partial_{i_j} \colon Q^{n-m+j-1} \to Q^{n-m+j}$ in the definition above.
  Then $Q \delta \colon Q^{n-m} \to Q^n$ is induced by a map $\bar{\delta} \colon \Box^{n-m} \to \Box^n$ whose standard form is
   \[ \partial_{k_m, \varepsilon_m} \ldots \partial_{k_1, \varepsilon_1} \]
\end{lem}

In other words, in order to obtain the standard form of the map inducing $Q \delta$, one only needs to compose the maps inducing the images of the maps in the standard form of $\delta$; no rearrangement is needed, as the result is already in standard form.

\begin{proof}
	Consider a simplicial standard form $\partial^{n_m}_{i_m}\dots\partial^{n_1}_{i_1}$ with $i_m>\dots>i_1$.
	Then there exists a unique integer $1 \le r \le m+1$ such that $i_s = n_s$ for $s \ge r$ and $i_s < n_s$ for $s < r$.
	By definition of $Q$, we have that
	\[
	Q(\partial_{i_s})\text{ is induced by }
	\begin{cases}
		\partial_{i_s+1,1}, & s<r,\\
		\partial_{i_s,0}, & s \ge r.
	\end{cases}
	\]
	Thus the only place where $Q$ can potentially disrupt the standard form is $Q(\sface{i_{r-1}})$ versus $Q(\sface{i_r})$.
	But for $Q$ to actually disrupt it, we must have $i_{r-1}=i_r-1$, so
	\[
	n_{r-1} = n_r-1 = i_r-1 = i_{r-1}
	\]
	which contradicts our choice of $r$.
	This completes the proof.
\end{proof}

\begin{Def}
	We extend $Q$ to a functor $Q \colon \Delta^{+} \to \mcSet$ by defining $\mQ n = Q\mdelta n$ to be the marked cubical set obtained from $Q^n$ by marking the unique non-degenerate $n$-cube.
\end{Def}

\begin{prop}
	The above definition indeed defines a unique functor $Q \colon \Delta^{+}\to\mcSet$, and moreover it extends to a unique-up-to-isomorphism left adjoint functor $Q \colon \msSet \to \mcSet$.
\end{prop}

\begin{proof}
	Since $Q\sigma_i \colon Q^n \to Q^{n-1}$ sends the unique non-degenerate $n$-cube to a degenerate (and so in particular marked) cube for each $i$, we immediately see that there is indeed such a unique functor $Q \colon \Delta^+ \to \mcSet$.
	This functor induces an adjunction between $[(\Delta^+)^\op,\Set]$ and $\mcSet$ whose left adjoint we still denote by $Q$.
	Since $Q(\varphi^n)$ is an epimorphism for each $n$, we see that the right adjoint to $Q$ lands in the full subcategory $\msSet \subset [(\Delta^+)^\op,\Set]$.
	Thus the restriction of this $Q$ to $\msSet$ is still a left adjoint functor; its uniqueness up to isomorphism is obvious.
\end{proof}

\begin{Def}
We denote the right adjoint of $Q$ by $\int \colon \mSet \to \mcSet$.
\end{Def}

By definition of $Q$, we obtain a commuting diagram:
	\[
	\begin{tikzcd}
		\msSet
		\arrow [rrr, bend left=10, "Q"]
		\arrow [dd]
          & & &
		 \mcSet 
		 \arrow [lll, bend left=10, "\int"]
		\arrow [dd] \\
		& & &
		\\
		\sSet
		\arrow [rrr, bend left=10, "Q"]
		& & &
		 \cSet 
		 \arrow [lll, bend left=10, "\int"]
	\end{tikzcd}
	\]
That $Q$ commutes with the vertical forgetful functors is immediate from the definition of the marked version of $Q$. It is also immediate from the definition that $Q$ commutes with the maximal marking functors; it thus follows that their right adjoints commute as well, i.e.~ that $\int$ commutes with the forgetful functors.

\begin{lem}\label{Q-unit-iso}
The unit $\id_{\cSet^+} \Rightarrow \int Q$ is a natural isomorphism.
\end{lem}

\begin{proof}
By \cite[Prop.~1.3]{gabriel-zisman}, the unit of an adjunction is an isomorphism if and only if the left adjoint is fully faithful. The functor $Q \colon \sSet \to \cSet$ is fully faithful by \cite[Thm.~6.10]{doherty-kapulkin-lindsey-sattler}, hence $\int Q  X \to X$ is an isomorphism on underlying simplicial sets, i.e. an entire map, for all $X$. In particular, this implies it is an isomorphism on all $\Delta^n$; it thus suffices to show that it is also an isomorphism on the objects $\widetilde{\Delta}^n$.

For this, it suffices to show that the marked simplices of $\int \widetilde{Q}^n$ coincide with those of $\widetilde{\Delta}^n$. To see this, observe that marked $m$-simplices of $\int \widetilde{Q}^n$ are given by maps $\widetilde{\Delta}^m \to \int \widetilde{Q}^n$; by adjointness, these correspond to maps $\widetilde{Q}^m \to \widetilde{Q}^n$, i.e. to marked $(0,m)$-cones of $\widetilde{Q}^n$. The only such cone which is non-degenerate is given by the identity on $\widetilde{Q}^n$, which indeed corresponds to the unique non-degenerate marked simplex of $\widetilde{\Delta}^n$.
\end{proof}

\begin{lem}\label{Q-counit-mono}
For any $X \in \cSet$, $Q \int X $ is the regular subcomplex of $X$ whose non-degenerate $n$-cubes are the $(0,n)$-cones in $X$, and the counit is the inclusion.
\end{lem}

\begin{proof}
The proof that the counit is a monomorphism and characterization of the underlying cubical set of $Q \int X$  is \cite[Lem.~6.9]{doherty-kapulkin-lindsey-sattler}. To see that the counit is regular, let $x$ be a marked $(0,n)$-cone of $X$; then $x \colon Q^n \to X$ factors through $\widetilde{Q}^n = Q\mdelta n$. Thus the corresponding simplex of $\int X$, i.e. the adjunct $\overline{x} \colon \Delta^n \to \int X$, factors through $\int \widetilde{Q}^n \cong \widetilde{\Delta}^n$.
\end{proof}

\begin{prop}\label{Q_triv}
	There is a natural isomorphism $Q\triv{n} \cong \triv{n}Q$ for any $n \ge 0$.
\end{prop}
\begin{proof}
	Since both $Q\triv{n}$ and $\triv{n}Q$ are cocontinuous, it suffices to check on the marked and the unmarked standard simplices.
	These special cases are straightforward to check.
\end{proof}

\subsection{Quillen adjunctions} \label{sec:Q-is-Quillen}
Now we consider the interactions of the functors $C$ and $Q$ with the complicial and comical model structures.

The cone functor $C \colon \mcSet \to \mcSet$ has neither adjoint; it can be easily seen that it does not preserve products or coproducts. Thus it is neither a left nor a right Quillen functor. We may note, however, that the natural transformation $\Box^0 \Rightarrow C$ allows us to extend $C$ to a functor $C \colon \mcSet \to \mcSet_*$, where $\mcSet_*$ denotes the slice category $\Box^0 \downarrow \mcSet$; in other words, for $X \in \mcSet$, the marked cubical set $CX$ has a natural basepoint given by the cone point. Moreover, this basepoint is ``initial'', in the sense that for every $n$-cube $x$ of a marked cubical set $X$, there is a unique $(n+1)$-cube $Cx$ of $CX$ such that $(Cx)\bd_{1,0}$ is a degeneracy of the cone point and $(Cx) \bd_{1,1} = x$. This extension to $\mcSet_*$ allows us to view $C$ as a left Quillen functor, as we now show.

\begin{lem}\label{C-left-adj}
The functor $C \colon \mcSet \to \mcSet_*$ defined above is a left adjoint.
\end{lem}

\begin{proof}
Let $Y \in \mcSet_*$ with basepoint $y \colon \Box^0 \to Y$. For $X \in \mcSet$, maps from $CX$ to $Y$ in $\mcSet_*$ naturally correspond to diagrams of the form
\[
\xymatrix{
X \ar[r] \ar[d]_{\bd_{1,0} \otimes X} & \Box^0 \ar[d] \ar@/^/[ddr]^{y} \\
\Box^1 \otimes X \ar[r] \ar@/_/[drr] & CX \ar[dr] \\
&& Y \\
}
\]
In other words, these are maps $\Box^1 \otimes X \to Y$ which pre-compose with $\bd_{1,0} \otimes X$ to give the constant map at $y$. By adjointness, these correspond to maps $X \to \uhom_{R}(\Box^1,Y)$ which post-compose with the pre-composition map $\bd_{1,0}^* \colon \uhom_{R}(\Box^1,Y) \to \uhom_{R}(\Box^0,Y) \cong Y$ to give the constant map at $y$.  These maps, in turn, correspond to commuting diagrams of the form:
\[
\xymatrix{
X \ar[r] \ar[d] & \uhom_{R}(\Box^1, Y) \ar[d]^{\bd_{1,0}^*} \\\
\Box^0 \ar[r]^(0.3){y} & \uhom_{R}(\Box^0,Y) \\
}
\]

Thus $C$ has a right adjoint which sends each such $Y$ to the pullback object given by the cospan $\uhom_{R}(\Box^1,Y) \xrightarrow{\bd_{1,0}^*} \uhom_{R}(\Box^0,Y) \xleftarrow{y} \Box^0$.
\end{proof}

\begin{lem}\label{C-pres-cof}
The functor $C \colon \mcSet \to \mcSet$ preserves monomorphisms.
\end{lem}

\begin{proof}
This follows from \cref{geo-prod-description} and the construction of $CX$ as a quotient of $\Box^1 \otimes X$, which we may view as identifying all cubes of $\Box^1 \otimes  X$ which can be written as $\bd_{1,0} \otimes x$ for some cube $x$ in $X$.
\end{proof}

\begin{lem}\label{C-pres-we}
The functor $C \colon \mcSet \to \mcSet$ preserves the weak equivalences of each of the (saturated, $n$-trivial) comical model structures.
\end{lem}

\begin{proof}
Given a map $X \to Y$ in $\mcSet$, we have the following commuting diagram:
\[
\xymatrix{
X \ar[rr] \ar[dd]_{\bd_{1,0} \otimes X} \ar[dr] & & \Box^0 \ar[dd]|{\hole} \ar@{=}[dr] \\
& Y \ar[dd]_(0.3){\bd_{1,0} \otimes Y} \ar[rr] & & \Box^0 \ar[dd] \\
\Box^1 \otimes X \ar[dr] \ar[rr]|(0.53){\hole} & & CX \ar[dr] \\
& \Box^1 \otimes Y \ar[rr] & & CY \\
}
\]
The front and back squares are pushouts, and the maps $X \hookrightarrow \Box^1 \otimes X$ and $Y \hookrightarrow \Box^1 \otimes Y$ are cofibrations. If $X \to Y$ is a weak equivalence, then all of the back-to-front maps are weak equivalences; in the case of $\Box^1 \otimes X \to \Box^1 \otimes Y$ case this follows from the monoidality of the model structure. Thus the map between pushout objects $CX \to CY$ is a weak equivalence by the gluing lemma. 
\end{proof}

\begin{cor}\label{C-QA}
The functor $C \colon \mcSet \to \mcSet_*$ is left Quillen, where $\mcSet$ is equipped with any of the ($n$-trivial, saturated) comical model structures and $\mcSet_*$ is equipped with the corresponding slice category model structure.
\end{cor}

\begin{proof}
This is immediate from \cref{C-left-adj,C-pres-cof,C-pres-we}.
\end{proof}

A natural way of constructing cones on (marked) simplicial sets by adjoining an initial basepoint is given by taking the join on the left with the standard 0-simplex $\Delta^0$; as in the cubical case, we obtain a functor $\Delta^0 \star - \colon \sSet^+ \to \mSet_*$. In this case as well, the functor into the pointed category admits a right adjoint (see \cref{join-preserves-pushouts}), even though this is not the case for the corresponding endofunctor of the unpointed category. In fact,  there is a close relationship between cones on (marked) cubical sets and joins of (marked) simplicial sets with the point, mediated by the functor $Q$.


\begin{prop}\label{CQ-iso}
We have a natural isomorphism of functors $Q(\Delta^0 \star -) \cong CQ \colon \sSet \to \cSet$.
\end{prop}

\begin{proof}
We first note that since $Q \Delta^0 \cong \Box^0$, $Q$ defines a functor $\sSet_* \to \cSet_*$. Thus we may consider the following diagram of functors:
\[
\xymatrix{
\msSet \ar[r]^(0.45){\Delta^0 \star -} \ar[d]_{Q} & \msSet_* \ar[d]^{Q} \\
\mcSet \ar[r]^{C} & \mcSet_* \\
}
\]
We will show that this diagram commutes up to natural isomorphism; the desired natural isomorphism of functors $\sSet \to \cSet$ then follows by post-composition with the forgetful functor $\cSet_* \to \cSet$.

As all of the functors in the diagram above are left adjoints  (cf. \cref{C-left-adj}), it suffices to show that the composites agree on representable objects and maps between them. On objects, we may observe that for all $n \geq 0$ we have $Q(\Delta^0 \star \Delta^n) = Q\Delta^{n+1} = Q^{n+1}$, and similarly $CQ\Delta^n = CQ^n = Q^{n+1}$, with the basepoint in each case given by the initial vertex. For objects $\widetilde{\Delta}^n$, we may observe that $Q(\Delta^0 \star \widetilde{\Delta}^n)$ and $CQ\widetilde{\Delta}^n$ both have exactly two non-degenerate marked cubes, namely the interior $(n+1)$-cube (the image under the quotient map $\Box^{n+1} \to Q^{n+1}$ of $\id_{[1]^{n+1}}$) and the $n$-cube opposite the initial vertex (the image under the quotient map of $\bd_{1,1}$, also the image under $Q$ of $\bd_0$).

It remains to show that the two composites agree on maps between representables; for this, it suffices to consider faces, degeneracies, and markers. For markers, we note that both composites send each map $\Delta^n \to \widetilde{\Delta}^n$ to the unique entire map between the images of its domain and codomain.

We now prove that the composites agree on face maps; the proof for degeneracies is similar. Consider a (simiplicial) face map $\bd_{i} \colon \Delta^{n-1} \to \Delta^n$. The functor $\Delta^0 \star -$ sends this map to $\bd_{i+1} \colon \Delta^n \to \Delta^{n+1}$. The functor $Q$ then sends this map to the map $Q^n \to Q^{n+1}$ induced by the cubical face map $\bd_{i+2,1}$ (if $i < n$) or $\bd_{n+1,0}$ (if $i = n$). On the other hand, if $i < n$ then $Q$ sends $\bd_i$ to $\bd_{i+1,1}$, which $C$ then sends to $\bd_{i+2,1}$, while $Q$ sends $\bd_n$ to $\bd_{n,0}$, which $C$ then sends to $\bd_{n+1,0}$. In either case, the two composites agree.
\end{proof}

\begin{lem}\label{Q-join-we}
Let $X \to Y$ be a map in $\mSet$. If $QX \to QY$ is a weak equivalence in the ($n$-trivial, saturated) comical model structure on $\mcSet$, then so is $Q(\Delta^n \star X) \to Q(\Delta^n \star Y)$ for all $n \geq 0$.
\end{lem}

\begin{proof}
By the associativity of the join and the fact that $\Delta^m \star \Delta^n \cong \Delta^{m+n+1}$ for all $m, n \geq -1$, it suffices to consider the case $n = 0$. By \cref{CQ-iso}, the map $Q(\Delta^0 \star X) \to Q(\Delta^0 \star Y)$ is isomorphic to $CQX \to CQY$, which is a weak equivalence by \cref{C-pres-we}.
\end{proof}

\begin{rmk}
That cones on (marked) cubical sets correspond to joins of (marked) simplicial sets with the point on the left is a consequence of our decision to define cones using the functor $C_{L,0}$ of \cite{doherty-kapulkin-lindsey-sattler}. Had we instead defined $CX$ by collapsing the right end of a cylinder on $X$, i.e. had we used one of the functors $C_{L,1}$ or $C_{R,1}$ from \cite{doherty-kapulkin-lindsey-sattler}, we would instead have obtained an analogous relationship between cones and joins with the point on the right, mediated by the corresponding variant of $Q$.
\end{rmk}

We are now prepared to show that $Q$ is a left Quillen functor.

\begin{lem}\label{Q_horn}
	For any $0 \le k \le n$ and $\varepsilon \in \{0,1\}$, the map $Q(\horn n k \incl \adelta n k)$ is a pushout of a comical open box inclusion.
\end{lem}

\begin{proof}
	We first consider the case $k < n$; in this case $Q(\sface k) = \face{k+1}1$.
	Then the underlying cubical map of $Q(\horn n k \incl \adelta n k)$ is a pushout of the (unmarked) open box inclusion $\obox{n}{k+1}{1} \incl \cube{n}$ (see \cite[Lem.~6.13]{doherty-kapulkin-lindsey-sattler} for details). Thus $Q \adelta n k$ is obtained from $Q \horn n k$ by adding a filler for the open box $\sqcap^n_{k+1,1} \to Q \horn n k$ given by composing the quotient map $\sqcap^n_{k+1,1} \to Q \Lambda^n$ with the entire map $Q \Lambda^n \to Q \horn n k$. This filler is marked in $Q \adelta n k$, as it factors through the image under $Q$ of the marked simplex $\id_{[n]} \colon \Delta^n \to \adelta n k$. Similarly, the $(k+1,1)$-face of the filler is the cube of $Q \adelta n k$ corresponding to the unmarked simplex $\delta_k \colon \Delta^{n-1} \to \adelta n k$; this face is therefore unmarked. Therefore, to exhibit  $Q(\horn n k \incl \adelta n k)$ as a pushout of $\obox{n}{k+1}{1} \incl \Box^n_{k+1,1}$, it suffices to show that the open box in $Q\horn n k$ which is filled to obtain $Q \adelta n k$ is $(k+1,1)$-comical, i.e.~ that its faces whose standard forms do not include $\face k 1$, $\face {k+1}0$, $\face {k+1}1$, or $\face{k+2}1$ are marked.
	
	
	To this end, let $\delta = \face{i_m}{\varepsilon_m}\dots\face{i_1}{\varepsilon_1}$ be a face of $\cube n$ in its standard form, and suppose that this standard form does not include $\face k 1$, $\face {k+1}0$, $\face {k+1}1$, or $\face{k+2}1$.
	We wish to show that the corresponding face of the open box described above, i.e.~ the image of $\delta$ under the composite map $\Box^n \to Q^n \to Q \horn n k$, is marked. To do this, we will apply the correspondence between cubical and simplicial standard forms established in \cref{Q_NF} to show that this cube of $Q \horn n k$ corresponds to a marked simplex of $\horn n k$.

	First we treat the sub-case where $\varepsilon_i = 1$ for all $i$.
	In this case we have the following commutative diagram:
	\[
	\begin{tikzcd}[column sep = large, row sep = large]
		\cube {n-m}
		\arrow [d]
		\arrow [r, "\face{i_1}{1}"] &
		\cube {n-m+1}
		\arrow[d]
		\arrow [r, "\face{i_2}{1}"] &
		\dots
		\arrow [r, "\face{i_m}{1}"] &
		\cube n
		\arrow [d]\\
		Q^{n-m}
		\arrow [r, "Q(\sface{i_1-1})"] &
		Q^{n-m+1}
		\arrow [r, "Q(\sface{i_2-1})"] &
		\dots
		\arrow [r, "Q(\sface{i_m-1})"] &
		Q^n
	\end{tikzcd}
	\]
By \cref{Q_NF}, the bottom composite is the image under $Q$ of a map $\overline{\delta} \colon \Delta^{n-m} \to \Delta^n$ whose standard form is $\sface{i_m-1}\dots\sface{i_1-1}$. Our assumption on the standard form of $\delta$ thus implies that the standard form of $\overline{\delta}$ does not involve $\sface{k-1}$, $\sface{k}$, or $\sface{k+1}$, so this face is marked in $\adelta n k$.
	It follows that the corresponding cube of $Q \horn n k$ is marked, as desired.
	
	We proceed similarly in the case where $\varepsilon_r = 0$ for some $1 \le r \le m$.
	Without loss of generality, we may assume that $r$ is the smallest such integer.
	Write $\bar r = n-m+r$;	then we may further assume $(i_r,\dots,i_m)=(\bar r,\dots,n)$, for otherwise we may apply \cref{cone-desc-faces} for $Q^n = C^{0,n}$ to see that this cube is degenerate in $Q^n$.
	It follows from \cref{ConeDesc} and our choice of $r$ that the two faces 
	\[
	\delta = \face{i_m}{\varepsilon_m}\dots\face{i_1}{\varepsilon_1} = \face{n}{\varepsilon_m}\dots\face{\bar r+1}{\varepsilon_{r+1}}\face{\bar r}{0}\face{i_{r-1}}{1}\dots\face{i_1}{1}
	\]
	and
	\[
	\face{n}{0}\dots\face{\bar r+1}{0}\face{\bar r}{0}\face{i_{r-1}}{1}\dots\face{i_1}{1}
	\]
	represent the same cube in $Q^n$.
	The latter standard form fits into the following commutative diagram:
	\[
	\begin{tikzcd}[column sep = large, row sep = large]
		\cube {n-m}
		\arrow [d]
		\arrow [r, "\face{i_1}{1}"] &
		\cube {n-m+1}
		\arrow[d]
		\arrow [r, "\face{i_2}{1}"] &
		\dots
		\arrow [r, "\face{i_{r-1}}{1}"] &
		\cube {\bar r-1}
		\arrow [d]
		\arrow [r, "\face{\bar r}{0}"] &
		\cube{\bar r}
		\arrow [d]
		\arrow [r, "\face{\bar r+1}{0}"] &
		\dots
		\arrow [r, "\face{n}{0}"] &
		\cube n
		\arrow [d]\\
		Q^{n-m}
		\arrow [r, "Q(\sface{i_1-1})"] &
		Q^{n-m+1}
		\arrow [r, "Q(\sface{i_2-1})"] &
		\dots
		\arrow [r, "Q(\sface{i_{r-1}-1})"] &
		Q^{\bar r -1}
		\arrow [r, "Q(\sface{\bar r})"] &
		Q^{\bar r}
		\arrow [r, "Q(\sface{\bar r+1})"] &
		\dots
		\arrow [r, "Q(\sface{n})"] &
		Q^n
	\end{tikzcd}
	\]
Similarly to the previous case, it follows by \cref{Q_NF} that the bottom composite is the image under $Q$ of a simplex $\overline{\delta} \colon \Delta^{n-m} \to \Delta^n$ whose standard form is given by $\sface{n}\dots\sface{\bar r}\sface{i_{r-1}-1}\dots\sface{i_1-1}$.	To prove that the corresponding cube is marked in $Q \horn n k$, it suffices to show that the standard form of $\overline{\delta}$  does not include $\sface{k-1}$, $\sface{k}$, or $\sface{k+1}$, and hence that $\overline{\delta}$ is marked as a simplex of $\adelta n k$.
	Note that, since $(i_r,\dots,i_m) = (\bar r,\dots,n)$ and the string $\face{i_m}{\varepsilon_m}\dots\face{i_1}{\varepsilon_1}$ does not include $\face {k+1}0$ or $\face {k+1}1$, we must have $\bar r > k+1$.
	Therefore, if $Q(\sface{k-1})$, $Q(\sface{k})$, or $Q(\sface{k+1})$ appears in the above diagram then it must appear to the left of $Q^{\bar r-1}$, which is impossible since the string $\face{i_{r-1}}{1}\dots\face{i_1}{1}$ does not include $\face k 1$, $\face {k+1}1$, or $\face{k+2}1$. It follows that the standard form of $\overline{\delta}$  does not include $\sface{k-1}$, $\sface{k}$, or $\sface{k+1}$. Thus the simplex $\overline{\delta}$ is marked as a face of $\adelta n k$, implying that the corresponding cube $\delta$ is marked as a face of $Q \adelta n k$.
	This completes the proof of the case $k<n$.
	
	The case $k=n$ can be proven similarly.
	In fact, this case is easier since it is impossible for the string $\face{i_m}{\varepsilon_m}\dots\face{i_1}{\varepsilon_1}$ to not include $\face{n}{0}$ or $\face{n}{1}$ while also having $(i_r,\dots,i_m) = (\bar r, \dots, n)$. This allows us to immediately dismiss the sub-case in which $\varepsilon_r = 0$ for some $r$, as all cubes falling under this sub-case are degenerate. 
\end{proof}

\begin{lem}\label{Q_marking_extension}
	For any $0 \le k \le n$ and $\varepsilon \in \{0,1\}$, the map $Q(\adeltap n k \incl \triv{n-2}\adelta n k)$ is a pushout of a comical marking extension.
\end{lem}

\begin{proof}
	Observe that the map $\adeltap n k \incl \triv{n-2}\adelta n k$ fits in the dashed part of the following diagram:
	\[
	\begin{tikzcd}
		\horn n k
		\arrow [r]
		\arrow [d] 
		\arrow [dr, phantom, "\ulcorner" very near end] &
		\triv{n-2}\horn n k
		\arrow [d]
		\arrow [ddr, bend left] & \\
		\adelta n k
		\arrow [r]
		\arrow [drr, bend right] &
		\cdot
		\arrow [dr, dashed] \\
		& & \triv{n-2}\adelta n k
	\end{tikzcd}
	\]
	Since $Q$ is cocontinuous and commutes with trivialisations by \cref{Q_triv}, it follows that $Q(\adeltap n k \incl \triv{n-2}\adelta n k)$ likewise fits in:
	\[
	\begin{tikzcd}
		Q(\horn n k)
		\arrow [r]
		\arrow [d] 
		\arrow [dr, phantom, "\ulcorner" very near end] &
		\triv{n-2}Q(\horn n k)
		\arrow [d]
		\arrow [ddr, bend left] & \\
		Q(\adelta n k)
		\arrow [r]
		\arrow [drr, bend right] &
		\cdot
		\arrow [dr, dashed] \\
		& & \triv{n-2}Q(\adelta n k)
	\end{tikzcd}
	\]
	Thus we see that $Q(\adeltap n k)$ may be obtained from $Q(\adelta nk)$ by marking all $(n-1)$-cubes contained in $Q(\horn nk)$, and further marking the (unique) remaining $(n-1)$-cube yields $Q(\tau_{n-2}\adelta nk)$.
	With this description, it can be shown that the map in question is a pushout of an elementary comical marking extension; the details are similar to that in the proof of \cref{Q_horn}.
\end{proof}

\begin{thm}\label{Q-Quillen}
	The functor $Q$ is left Quillen with respect to the (saturated) ($n$-trivial) complicial model structure on $\msSet$ and the comical model structure on $\mcSet$.
\end{thm}

\begin{proof}
	First we must show that $Q$ preserves cofibrations.
	Since $\msSet \to \sSet$ preserves monomorphisms and $\mcSet \to \cSet$ reflects them, it suffices to check that the ``unmarked version'' $Q \colon \sSet \to \cSet$ preserves them.
	This is easy to check (and also appears as \cite[Lem.~4.5]{kapulkin-lindsey-wong}).
	
	We have shown in \cref{Q_horn,Q_marking_extension} $Q$ sends the complicial horn inclusions to pushouts of comical open box inclusions, and the complicial marking extensions to pushouts of comical marking extensions.
	It is also easy to see from our construction of $\mQ{n}$ that $Q$ sends each simplicial $n$-marker to a pushout of the cubical $n$-marker.
	
	It remains to treat the saturated case; by \cref{Rezk-lift}, it suffices to show that $Q$ sends all simplicial Rezk maps to trivial cofibrations.
	The object $L$, the domain of the elementary simplicial Rezk map, may be written as the colimit of the following diagram:
	\[
	\begin{tikzcd}[column sep = small]
		& \Delta^1
		\arrow [dr, "\sface{1}"]
		\arrow [dl] & &
		\Delta^1
		\arrow[dl, swap, "\sface{0}"]
		\arrow [dr, "\sface{2}"] & & 
		\Delta^1
		\arrow [dl, swap, "\sface{1}"]
		\arrow [dr] & \\
		\mDelta{1} & &
		\mDelta{2} & &
		\mDelta{2} & &
		\mDelta{1}
	\end{tikzcd}
	\]

	It follows from the above colimit description of $L$ that we can obtain $QL$ from two marked $2$-cubes, which we call \emph{left} and \emph{right}, by:
	\begin{itemize}
		\item collapsing each cube to $Q^2$ (by collapsing $\face 1 0$);
		\item gluing $\face 1 1$ of the left cube to $\face 2 0$ of the right cube; and
		\item marking $\face 2 1$ of each cube.
	\end{itemize}
	Thus $QL$ is the cubical set illustrated below on the left, while $L_{1,1}$ is illustrated on the right.
\[
\begin{tikzpicture}
	\foreach \x in {0,1,2,4,5,6}
	\filldraw
	(\x,0) circle [radius = 1pt]
	(\x,1) circle [radius = 1pt];
	
	\foreach \x in {1,2,4,5,6}
	\draw[->] (\x,0.8) -- (\x,0.2);
	
	\foreach \x in {0,1,4,5}
	\draw[->] (\x+0.2,0) -- (\x+0.8,0);
	
	\foreach \x in {0,4,5}
	\draw[->] (\x+0.2,1) -- (\x+0.8,1);
	
	\foreach \x in {0,1,4,5}
	\filldraw[shaded, rounded corners]
	(\x+0.1,0.1) -- (\x+0.9,0.1) -- (\x+0.9,0.9) -- (\x+0.1,0.9) -- cycle;
	
	\draw[double] (0,0.2) -- (0,0.8) (1.2,1) -- (1.8,1);
	
	\draw[->, double = shaded] (0.3,0.3) -- (0.7,0.7);
	\draw[->, double = shaded] (4.3,0.3) -- (4.7,0.7);
	\draw[->, double = shaded] (1.7,0.7) -- (1.3,0.3);
	\draw[->, double = shaded] (5.7,0.7) -- (5.3,0.3);
	
	\node at (0.5,-0.2) {$\sim$};
	\node at (4.5,-0.2) {$\sim$};
	\node at (2.2,0.5) {$\sim$};
	\node at (3.8,0.5) {$\sim$};
	\node at (6.2,0.5) {$\sim$};
	\node at (5.5,1.1) {$\sim$};
	
\end{tikzpicture}
\]

	So we obtain a map $L_{1,1} \to QL$, and similarly a map $L'_{1,1} \to QL'$.
	Since both $L_{1,1} \to L'_{1,1}$ and $QL \to QL'$ are entire, we can deduce by an easy analysis of the marked cubes that the latter is a pushout of the former.
	Therefore $QL \to QL'$ is a trivial cofibration in the saturated model structure. The case for general Rezk maps $\Delta^k \star L \to \Delta^k \star L'$ follows from this, together with \cref{Q-join-we}.
\end{proof}

\section{Triangulation is a Quillen equivalence}\label{sec:equivalence}

We show that the triangulation is a Quillen equivalence between the model structure for ($n$-trivial, saturated) comical sets and the model structure for ($n$-trivial, saturated) complicial sets. Our proofs will make use of the functor $Q$ and results concerning this functor established in \cref{sec:cones}; thus they are valid in $\cSet^+_0, \cSet^+_1$, and $\cSet^+$. Once again, we will work in $\cSet^+$ for the sake of concreteness, but our proofs will only require positive connections; thus these proofs work verbatim in $\cSet^+_1$ and can be dualized to $\cSet^+_0$. Since this paper first appeared, the results of this section have been generalized to $\cSet_{\varnothing}^+$ in \cite{doherty:without-connections}.

\subsection{$TQ$ is weakly equivalent to identity}
The aim of this section is to construct a natural transformation $\rho \colon TQ \Rightarrow \id$ and to exhibit it as a natural weak equivalence.

\begin{rmk}
	Whenever we say a \emph{weak equivalence} in this section, we will always mean one with respect to the complicial model structure (without saturation or $n$-triviality).
\end{rmk}

Note that, since $T$ is cocontinuous, we can compute $TQ^n$ as the following pushout:
\[
\begin{tikzcd}[row sep = large]
	\displaystyle\coprod_{1 \le i \le n} T(\cube{i-1}\otimes\cube{n-i})
	\arrow [r]
	\arrow [d]
	\pushout &
	T\cube n
	\arrow [d] \\
	\displaystyle\coprod_{1 \le i \le n} T\cube{i-1}
	\arrow [r] &
	TQ^n
\end{tikzcd}
\]
It inherits a unique unmarked $n$-simplex from $T\cube n$, and marking this simplex yields $T\mQ{n}$.

Recall that an $r$-simplex in $T\cube n$ corresponds to a sequence $\phi \in \{1,\dots,r,\pm\infty\}^n$.
Since the right vertical map in the above pushout square is an epimorphism, any $r$-simplex in $TQ^n$ may also be represented (not necessarily uniquely) by such $\phi$.
Two sequences $\phi$ and $\chi$ represent the same simplex if and only if, for any $i$ with $\phi_i \neq \chi_i$, there exists $j<i$ with $\phi_j=\chi_j=+\infty$.
\begin{Def}
	For $n \ge 0$, we define $\rho^n \colon TQ^n \to \Delta^n$ by sending an $r$-simplex represented by a sequence $\phi$ to $\rho^n(\phi) \colon [r] \to [n]$ given by
	\[
	\rho^n(\phi)(p) = \max \, \{k \in [n]~|~\phi_i\le p \text{ for all } i \leq k\} \text{.}
	\]
	Note that since the condition ``$\phi_i\le p$ for all $i \leq k$'' is always vacuously true for $k=0$, the set $\{k \in [n]~|~\phi_i\le p \text{ for all } i \leq k\}$ is always non-empty.
	The map $\widetilde{\rho}^n \colon T\mQ{n} \to \mDelta{n}$ (for $n \ge 1$) has the same underlying simplicial map.
\end{Def}

\begin{rmk}
	If we regard the vertices of $T\cube n$ as binary strings of length $n$, then $T\cube n \to TQ^n$ acts on those vertices by identifying any two strings that have their first $0$ in the same position.
	Intuitively, the map $\rho^n \colon TQ^n \to \Delta^n$ is well defined because it essentially counts the number of $1$'s before the first $0$.
\end{rmk}

\begin{prop}
	The above definitions indeed yield maps $\rho^n \colon TQ^n \to \Delta^n$ and $\widetilde{\rho}^n \colon T\mQ{n} \to \mDelta{n}$ in $\msSet$.
	Moreover these maps extend to a unique natural transformation $\rho \colon TQ \to \id$ with $\rho_{\Delta^n} = \rho^n$ and $\rho_{\mDelta n}=\widetilde{\rho}^n$.
\end{prop}
\begin{proof}
	We must show that:
	\begin{enumerate}
		\item $\rho^n$ at least defines a valid map between the underlying simplicial sets;
		\item $\rho^n$ preserves marked simplices;
		\item $\rho$ is natural in $n$; and
		\item similarly for $\widetilde{\rho}$; and
		\item those maps indeed extends to a unique natural transformation $\rho$.
	\end{enumerate}
	For (1) and (3), see Proposition 6.18 and Lemma 6.19 in \cite{doherty-kapulkin-lindsey-sattler} respectively.
	We will skip (4) since it will be almost identical to (1-3).
	
	To prove (2), consider an $r$-simplex $\phi$ in $T\cube n$.
	Suppose that this simplex is:
	\begin{itemize}
		\item[(i)] non-degenerate, or equivalently each integer in $\{1,\dots,r\}$ appears at least once in $\phi$; and
		\item[(ii)] marked, or equivalently there is no sequence $1 \le i_1 < \dots < i_r \le n$ such that $\phi_{i_p} = p$.
	\end{itemize}
	We must show that $\rho^n$ sends the image of such $\phi$ in $TQ^n$ to a marked simplex in $\Delta^n$.
	We claim that there exist $1 \le p < q \le r$ with
	\[
	\min\{i~|~\phi_i=p\}>\min\{i~|~\phi_i=q\}.
	\]
	Indeed, for any $p$ and $q$, both minima are well defined because of (i), and such $p,q$ must exist for otherwise setting $i_p = \min\{i~|~\phi_i=p\}$ would violate (ii).
Now observe that there can be no $k \in [n]$ such that for all $i \leq k$, we have $\phi_i \leq p$, with $\phi_i = p$ for some such $i$. This is because the conditions defining $p$ and $q$ imply that if some such $k$ and $i$ did exist, we would have $\phi_{j} = q > p$ for some $j < i$, contradicting our assumption. Thus the sets $\{k \in [n]~|~(\forall i \le k)~\phi_i\le p\}$ and $\{k \in [n]~|~(\forall i \le k)~\phi_i\le p-1\}$ are equal; it follows that $\rho^n(\phi)(p) = \rho^n(\phi)(p-1)$. We thus see that the simplex $\rho^n(\phi)$ is degenerate (and hence marked) in $\Delta^n$.
	
	It remains to prove (5).
	Observe that, by construction of $Q$, we can regard $TQ$ as the restriction of a cocontinuous functor $[(\Delta^+)^\op,\Set] \to \msSet$ to the full subcategory $\msSet$.
	Similarly, we may regard the identity functor on $\msSet$ as the restriction of the (cocontinuous) reflection $[(\Delta^+)^\op,\Set] \to \msSet$.
	Since $[(\Delta^+)^\op,\Set]$ is the free cocompletion of $\Delta^+$, the maps $\rho^n$ and $\widetilde{\rho}^n$ extend to a unique natural transformation between those functors $[(\Delta^+)^\op,\Set] \to \msSet$.
	We thus obtain the desired natural transformation by restricting it to $\msSet$.
	Its uniqueness follows from the fact that any marked simplicial set can be written as a colimit of $\Delta^n$ and $\mDelta{n}$.
\end{proof}

Now we prove that $\rho$ is a natural weak equivalence.

\begin{lem}\label{TQ-unmarked-simplex}
	The component $\rho^n \colon TQ^n \to \Delta^n$ is a weak equivalence.
\end{lem}
\begin{proof}
	We will exhibit $\Delta^n$ as a deformation retract of $TQ^n$ with the retraction part given by $\rho^n$.
	
	We define $\zeta^n \colon \Delta^n \to TQ^n$ to be the unique map picking out the unique unmarked $n$-simplex $\iota$, \emph{i.e.}\;the one represented by the sequence $1\,2\,\dots\,n$.
	Then it is straightforward to check that $\rho^n\zeta^n=\id$ holds.
	Note that, more explicitly, $\zeta^n$ sends an $r$-simplex $\alpha \colon [r] \to [n]$ to the one represented by the sequence $\phi \in \{1,\dots,r,\pm\infty\}^n$ given by
	\[
	\phi_i = \begin{cases}
		+\infty, & i>\alpha(r),\\
		p, & \alpha(p-1)<i\le\alpha(p),\\
		-\infty, & i \le \alpha(0).
	\end{cases}
	\]
	
Comparing the definitions of $\rho^n$ and $\zeta^n$, we can see that the map $\zeta^n \rho^n \colon TQ^n \to TQ^n$ sends a simplex represented by a sequence $\phi \in \{1,\ldots,r, \pm \infty\}^n$ to the simplex represented by the sequence $\phi'$ given by $\phi'_i = \max\{\phi_j | j \leq i\}$.

	We will construct a (left) homotopy between $\zeta^n\rho^n$ and the identity at $TQ^n$. Note that for any $X \in \sSet^+$, we have a cylinder object on $X$ given by $X \times \mDelta 1$; since $X \times \mDelta 1$ is equal to the Gray tensor product $X \otimes \mDelta 1$, this follows from \cref{gray-homotopical} together with the fact that $\mDelta 1$ is contractible.
	
	We begin by defining a map $H \colon (\Delta^1)^n\times \Delta^1 \to (\Delta^1)^n$ in $\sSet$ so that its action on the $0$-simplices (regarded as binary strings) is given by:
	\[
	\begin{aligned}
		H(\varepsilon_1,\varepsilon_2,\dots,\varepsilon_n,0) &= \bigl(\min\{\varepsilon_1\},\min\{\varepsilon_1,\varepsilon_2\},\dots,\min\{\varepsilon_1,\dots,\varepsilon_n\}\bigr),\\
		H(\varepsilon_1,\varepsilon_2,\dots,\varepsilon_n,1) &= (\varepsilon_1,\varepsilon_2,\dots,\varepsilon_n).
	\end{aligned}
	\]
	In other words, $H(-,1)$ acts as the identity and $H(-,0)$ replaces all entries after the first $0$ (if it exists) by $0$'s.
	
Now we will show that $H$ lifts to a map $H \colon T \Box^n \times \mDelta 1 \to T \Box^n$ in $\msSet$.
		To see this, let us describe this map $H$ in terms of the sequence representation of simplices.
		Fix an $r$-simplex $\phi \in \{1,\dots,r,\pm\infty\}^{n+1}$, and write $q = \phi_{n+1}$.
		Then $H$ sends $\phi$ to $\chi \in \{1,\dots,r,\pm\infty\}^n$ given by
		\[
		\chi_i = 
		\begin{cases}
			\phi_i, & \phi_i \ge q,\\
			\min\bigl\{\max\{\phi_1,\dots,\phi_i\},\phi_{n+1}\bigr\}, & \phi_i<q.
		\end{cases}
		\]
		Suppose that this $\chi$ is unmarked when regarded as an $r$-simplex in $T \Box^n$, i.e.~ that it has a complete substring as defined in \cref{complete-substring-def}.
		Then there exist
		\[
		1 \le j_1 < \dots < j_{q-1} < i_q < \dots < i_r \le n
		\]
		such that
		\begin{itemize}
			\item $\phi_{i_p} = p$ for $q \le p \le r$; and
			\item $\max\{\phi_1,\dots,\phi_{j_p}\} = p$ for $1 \le p < q$.
		\end{itemize}
		(If $q=-\infty$ then this is interpreted as the existence of such $i_1,\dots,i_r$, and if $q=+\infty$ then this is interpreted as the existence of such $j_1,\dots,j_r$.)
		We can then deduce by an elementary analysis of the $\max$ function that there exist
		\[
		1 \le i_1 \le j_1 < i_2 \le j_2 < \dots < i_{q-1} \le j_{q-1}
		\]
		such that $\phi_{i_p} = p$ for all $1 \le p < q$.
		It follows that the projection of $\phi$ onto the first factor $(\Delta^1)^n$ is unmarked when regarded as an $r$-simplex in $T \Box^n$, which in turn implies that $\phi$ itself is unmarked as an $r$-simplex in $T \Box^n \times\mDelta 1$.
Thus we do indeed obtain a marked simplicial set map $H \colon T \Box^n \times \mDelta 1 \to T \Box^n$ whose underlying simplicial set map is as described above.

We must now show that the map $H \colon T \Box^n \times \mDelta 1 \to T \Box^n$ induces a map between the quotients $TQ^n \times \mDelta 1 \to TQ^n$. In other words, we must show that if a pair of simplices of $T \Box^n$ are identified in $TQ^n \times \mDelta 1$, then their images under $H$ are identified in $TQ^n$.
	
Recall that two sequences $\phi, \psi \in \{1,\ldots,r,\pm \infty\}^n$ represent the same $r$-simplex of $TQ^n$ (in other words, the corresponding simplices of $T \Box^n$ are identified in $TQ^n$) if and only if, for any $i$ with $\phi_i \neq \psi_i$, there exists $j<i$ with $\phi_j=\psi_j=+\infty$. It follows that a pair of simplices of $T \Box^n \times \mDelta 1$ represented by sequences $\phi, \psi \in \{1,\ldots,r,\pm \infty\}^{n+1}$ are identified in $TQ^n \times \mDelta 1$ if and only if both of the following two conditions are satisfied:

\begin{itemize}
\item $\phi_{n+1} = \psi_{n+1}$;
\item for any $1 \leq i \leq n$ such that $\phi_i = \psi_i$, there exists $j < i$ with $\phi_j = \psi_j = + \infty$. 
\end{itemize}
	
Suppose that $\phi$ and $\psi$ satisfy these conditions, and let their images under $H$ be respectively defined by sequences $\chi, \xi \in \{1,\ldots,r,\pm \infty\}^n$. Denote the common value $\phi_{n+1} = \psi_{n+1}$ by $q$. Suppose that for some $1 \leq i \leq n$ we have $\chi_i \neq \xi_i$, and assume without loss of generality that $\phi_i \leq \psi_i$. 

We consider three possible cases.
First, suppose that $\phi_i \leq \psi_i < q$. Then we see that $\min\{\max\{\phi_1,\ldots,\phi_i\},q\} \neq \min\{\max\{\psi_1,\ldots,\psi_i\},q\}$; in particular, this means that the sets $\{\phi_1,\ldots,\phi_i\}$ and $\{\psi_1,\ldots,\psi_i\}$ are not equal. Thus there is some $k \leq i$ such that $\phi_k \neq \psi_k$; by assumption, this implies that there is some $j < k \leq i$ such that $\phi_j = \psi_j = + \infty$. It follows that $\chi_j = \xi_j = + \infty$.

Next we consider the case in which $\phi_i < q$ while $\phi_i \geq q$; in this case it is immediate that $\phi_i \neq \psi_i$, so once again there is $j < i$ such that $\phi_j = \psi_j = + \infty$, implying that $\chi_j = \xi_j = + \infty$ as well. Finally we consider the case in which $q \leq \phi_i \leq \psi_i$; then the assumption that $\chi_i \neq \xi_i$ implies $\phi_i \neq \psi_i$, and we proceed as in the previous case.
	
Thus we see that $H$ induces a map $TQ^n \times \mDelta 1 \to TQ^n$. From the definition of $H$ and the description of $\zeta^n \rho^n$ above, we can see that this map restricts to $\zeta^n \rho^n$ on $T \Box^n \times \{0\}$ and to $\id_{T \Box^n}$ on $T \Box^n \times \{1\}$.  Thus $H$ defines a left homotopy between $\zeta^n \rho^n$ and the identity on $TQ^n$. This completes the proof.
\end{proof}

\begin{lem}\label{TQ-marked-simplex}
	The component $\widetilde{\rho}^n \colon T\mQ{n} \to \mDelta{n}$ is a weak equivalence.
\end{lem}
\begin{proof}
	Fix $n \ge 1$.
	Then it follows from \cref{TQ-unmarked-simplex} that $\zeta^n$ is a trivial cofibration.
	Thus its pushout along the $n$-marker $\Delta^n \to \mDelta{n}$ is also a trivial cofibration.
	But it is easy to check that this pushout is a section of $\widetilde{\rho}^n$, so the lemma follows by the 2-out-of-3 property.
\end{proof}

\begin{thm}\label{TQ}
	The component $\rho_X \colon TQX \to X$ at any $X \in \msSet$ is a weak equivalence.
\end{thm}
\begin{proof}
	First, we prove the special case where $X$ is $n$-skeletal (\emph{i.e.}\,the underlying simplicial set of $X$ is $n$-skeletal.)
	We proceed by induction on $n \ge -1$.
	
	The base case is easy since the cocontinuity of $TQ$ implies that $\rho_\varnothing$ is invertible.
	For the inductive step, fix $n \ge 0$ and assume that $\rho_{Y}$ is a weak equivalence for any $(n-1)$-skeletal $Y$.
	Let $X$ be an $n$-skeletal marked simplicial set and denote by $X'$ its regular $(n-1)$-skeleton.
	Then we may obtain $\rho_{X}$ by taking the pushout of each row in
	\[
	\begin{tikzcd}[row sep = large]
		TQX'
		\arrow [d, swap, "\rho_{X^{n-1}}"] &
		\bigl(\coprod TQ \partial\Delta^n\bigr)\amalg\bigl(\coprod TQ\partial\Delta^n\bigr)
		\arrow [d, "(\coprod\rho)\amalg(\coprod\rho)"]
		\arrow [r]
		\arrow [l] &
		\bigl(\coprod TQ\Delta^n\bigr)\amalg\bigl(\coprod TQ\mDelta n\bigr)
		\arrow [d, "(\coprod\rho)\amalg(\coprod\rho)"]\\
		X' &
		\bigl(\coprod \partial\Delta^n\bigr)\amalg\bigl(\coprod\partial\Delta^n\bigr)
		\arrow [r]
		\arrow [l]  &
		\bigl(\coprod\Delta^n\bigr)\amalg\bigl(\coprod\mDelta n\bigr)
	\end{tikzcd}
	\]
	where, in each of the right four objects, the first (respectively second) coproduct ranges over the unmarked (resp.\,non-degenerate marked) $n$-simplices in $X$.
	The left and the middle vertical maps are weak equivalences by the inductive hypothesis.
	The right vertical map is also a weak equivalence by \cref{TQ-unmarked-simplex,TQ-marked-simplex} (note that the weak equivalences are closed under coproducts since they may be factorised as a trivial cofibration followed by a retraction of a trivial cofibration).
	Since both of the right-pointing arrows are cofibrations, it follows that the induced map $\rho_X$ is again a weak equivalence.
	
	Now we prove the theorem for general $X$.
	For each $n$, write $X^n$ for the regular $n$-skeleton of $X$.
	Then $n \mapsto TQX^n$ and $n \mapsto X^n$ yield two sequences  $\omega^\op \to \msSet$ of cofibrations.
	Since $\rho$ provides a natural weak equivalence between these two sequences, the colimit $\rho_X \colon TQX \to X$ is still a weak equivalence.
	This completes the proof.
\end{proof}

\begin{cor}\label{Q-pres-refl-we}
The functor $Q \colon \mSet \to \mcSet$ preserves and reflects weak equivalences, where $\mSet$ is equipped with the model structure for ($n$-trivial, saturated) complicial sets, and $\mcSet$ is equipped with the model structure for the corresponding comical sets. 
\end{cor}

\begin{proof}
That $Q$ preserves weak equivalences is immediate from \cref{Q-Quillen}.

To see that $Q$ reflects weak equivalences, let $X \to Y$ be a map in $\mSet$, such that $QX \to QY$ is a weak equivalence. By \cref{TQ}, we have a commuting diagram:
	\[
	\begin{tikzcd}
		T Q X
		\arrow [r]
		\arrow [d]
          &
		 T Q Y
		\arrow [d] \\
		X 
		\arrow [r] &
		Y
	\end{tikzcd}
	\]
in which the two vertical maps are weak equivalences. Since $T$ preserves weak equivalences by \cref{T-Quillen}, $TQX \to TQY$ is a weak equivalence as well. Thus $X \to Y$ is a weak equivalence by two-out-of-three.
\end{proof}

\subsection{The Quillen equivalence}

In this final section of the paper, we show that $T \dashv U$ is a Quillen equivalence, following the strategy used in the unmarked case in \cite{doherty-kapulkin-lindsey-sattler}.
That is, we first show that this is true of $Q \dashv \int$ (cf.~\cref{Q-equivalence}), and then apply \cref{TQ} to deduce \cref{T-equivalence}.
Thus the bulk of section deals with the adjunction $Q \dashv \int$ by establishing its counit as an anodyne morphism.
Our key technical device is the notion of a coherent family of composites, which we adapt to the marked setting from \cite[Sec.~5]{doherty-kapulkin-lindsey-sattler}.

In \cite{doherty-kapulkin-lindsey-sattler}, coherent families of composites were constructed using certain subcomplexes $B^{m, n}$ of the standard cones $C^{m,n}$ and showing that the inclusions $B^{m,n} \hookrightarrow C^{m,n}$ are anodyne.
The analogous constructions in our setting require the notions of strongly comical cubes and strongly comical cones, which we now introduce.

\begin{Def}
For $n \geq 1$, $1 \leq i \leq n$, the \emph{strongly $(i,1)$-comical cube}, denoted $\overline{\Box}^n_{i,1}$, is the marked cubical set whose underlying cubical set is $\Box^n$, with a non-degenerate face marked if and only if its standard form does not contain any of the maps $\bd_{i-1,1}$ (if $i > 1$), $\bd_{i,0}$, or $\bd_{i,1}$. 

A cube $\Box^n \to X$ of a marked cubical set $X$ is \emph{strongly $(i,1)$-comical} if it factors through $\Box^n \to \overline{\Box}^n_{i,1}$.


For example, $\overline{\Box}^n_{1,1}$ is the marked $1$-cube $\mcube 1$.
We depict below the strongly comical cubes $\overline{\Box}^n_{i,1}$ for $n = 2,3$ and $1 \le i \le n$ (recall the drawing convention described in \cref{cSet-unmarked}); the non-degenerate marked cubes in these comical cubes are listed in \cref{strongly-comical-table}.
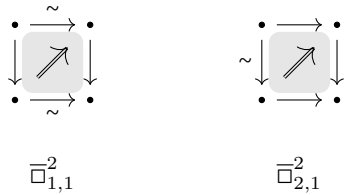
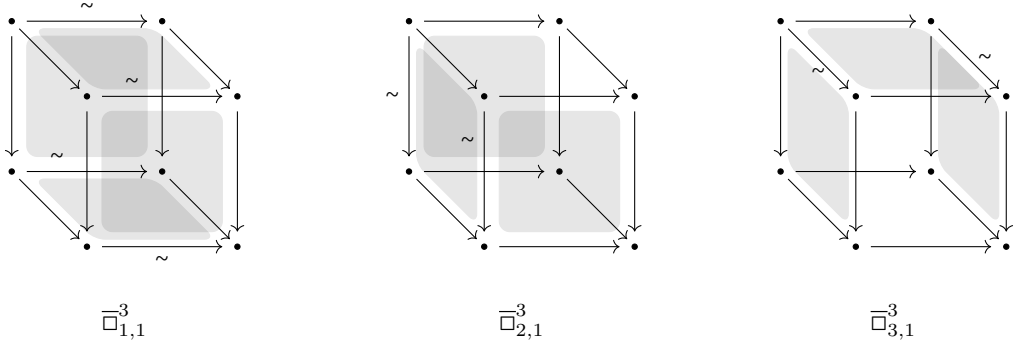
\[
\begin{tikzpicture}
	\filldraw
	(0,0) circle [radius = 1pt]
	(0,1) circle [radius = 1pt]
	(1,0) circle [radius = 1pt]
	(1,1) circle [radius = 1pt];
	
	\draw[->] (0.2,0) -- (0.8,0) node [midway, below] {$\sim$};
	\draw[->] (0.2,1) -- (0.8,1) node [midway, above] {$\sim$};
	\draw[->] (0,0.8) -- (0,0.2);
	\draw[->] (1,0.8) -- (1,0.2);
	
	\filldraw[shaded, rounded corners]
	(0.1,0.1) -- (0.9,0.1) -- (0.9,0.9) -- (0.1,0.9) -- cycle;
	
	\draw[->, double = shaded] (0.3,0.3) -- (0.7,0.7);
	
	\node at (0.5,-1) {$\overline{\Box}^2_{1,1}$};
\end{tikzpicture}
\hspace{5em}
\begin{tikzpicture}
	\filldraw
	(0,0) circle [radius = 1pt]
	(0,1) circle [radius = 1pt]
	(1,0) circle [radius = 1pt]
	(1,1) circle [radius = 1pt];
	
	\draw[->] (0.2,0) -- (0.8,0);
	\draw[->] (0.2,1) -- (0.8,1);
	\draw[->] (0,0.8) -- (0,0.2) node [midway, left] {$\sim$};
	\draw[->] (1,0.8) -- (1,0.2);
	
	\filldraw[shaded, rounded corners]
	(0.1,0.1) -- (0.9,0.1) -- (0.9,0.9) -- (0.1,0.9) -- cycle;
	
	\draw[->, double = shaded] (0.3,0.3) -- (0.7,0.7);
	
	\node at (0.5,-1) {$\overline{\Box}^2_{2,1}$};
\end{tikzpicture}
\]

\[
\begin{tikzpicture}
	\begin{scope}[blend group = multiply]
		\filldraw[shaded, rounded corners] (0.3,2.9) -- (1.9,2.9) -- (2.7,2.1) -- (1.1,2.1) -- cycle; 
		\filldraw[shaded, rounded corners] (0.3,0.9) -- (1.9,0.9) -- (2.7,0.1) -- (1.1,0.1) -- cycle; 
		\filldraw[shaded, rounded corners] (0.2,1.2) -- (1.8,1.2) -- (1.8,2.8) -- (0.2,2.8) -- cycle; 
		\filldraw[shaded, rounded corners] (1.2,0.2) -- (2.8,0.2) -- (2.8,1.8) -- (1.2,1.8) -- cycle; 
	\end{scope}
	
	\filldraw
	(1,0) circle [radius = 1pt]	
	(3,0) circle [radius = 1pt]	
	(1,2) circle [radius = 1pt]	
	(3,2) circle [radius = 1pt]	
	(0,1) circle [radius = 1pt]	
	(2,1) circle [radius = 1pt]	
	(0,3) circle [radius = 1pt]	
	(2,3) circle [radius = 1pt];
	
	\draw[->] (1.2,0) -- (2.8,0) node [midway, below] {$\sim$};
	\draw[->] (1.2,2) -- (2.8,2) node [near start, above] {$\sim$};
	\draw[->] (0.2,1) -- (1.8,1) node [near start, above] {$\sim$};
	\draw[->] (0.2,3) -- (1.8,3) node [midway, above] {$\sim$};
	
	\draw[->] (0,2.8) -- (0,1.2);
	\draw[->] (2,2.8) -- (2,1.2);
	\draw[->] (1,1.8) -- (1,0.2);
	\draw[->] (3,1.8) -- (3,0.2);
	
	\draw[->] (0.1,0.9) -- (0.9,0.1);
	\draw[->] (2.1,0.9) -- (2.9,0.1);
	\draw[->] (0.1,2.9) -- (0.9,2.1);
	\draw[->] (2.1,2.9) -- (2.9,2.1);
	
	\node at (1.5,-1) {$\overline{\Box}^3_{1,1}$};
\end{tikzpicture}
\hspace{5em}
\begin{tikzpicture}
	\begin{scope}[blend group = multiply]
		\filldraw[shaded, rounded corners] (0.2,1.2) -- (1.8,1.2) -- (1.8,2.8) -- (0.2,2.8) -- cycle; 
		\filldraw[shaded, rounded corners] (1.2,0.2) -- (2.8,0.2) -- (2.8,1.8) -- (1.2,1.8) -- cycle; 
		\filldraw[shaded, rounded corners] (0.1,1.1) -- (0.9,0.3) -- (0.9,1.9) -- (0.1,2.7) -- cycle; 
	\end{scope}
	
	\filldraw
	(1,0) circle [radius = 1pt]	
	(3,0) circle [radius = 1pt]	
	(1,2) circle [radius = 1pt]	
	(3,2) circle [radius = 1pt]	
	(0,1) circle [radius = 1pt]	
	(2,1) circle [radius = 1pt]	
	(0,3) circle [radius = 1pt]	
	(2,3) circle [radius = 1pt];
	
	\draw[->] (1.2,0) -- (2.8,0);
	\draw[->] (1.2,2) -- (2.8,2);
	\draw[->] (0.2,1) -- (1.8,1);
	\draw[->] (0.2,3) -- (1.8,3);
	
	\draw[->] (0,2.8) -- (0,1.2) node [midway, left] {$\sim$};
	\draw[->] (2,2.8) -- (2,1.2);
	\draw[->] (1,1.8) -- (1,0.2) node [near start, left] {$\sim$};
	\draw[->] (3,1.8) -- (3,0.2);
	
	\draw[->] (0.1,0.9) -- (0.9,0.1);
	\draw[->] (2.1,0.9) -- (2.9,0.1);
	\draw[->] (0.1,2.9) -- (0.9,2.1);
	\draw[->] (2.1,2.9) -- (2.9,2.1);
	
	\node at (1.5,-1) {$\overline{\Box}^3_{2,1}$};
\end{tikzpicture}
\hspace{5em}
\begin{tikzpicture}
	\begin{scope}[blend group = multiply]
		\filldraw[shaded, rounded corners] (0.3,2.9) -- (1.9,2.9) -- (2.7,2.1) -- (1.1,2.1) -- cycle; 
		\filldraw[shaded, rounded corners] (0.1,1.1) -- (0.9,0.3) -- (0.9,1.9) -- (0.1,2.7) -- cycle; 
		\filldraw[shaded, rounded corners] (2.1,1.1) -- (2.9,0.3) -- (2.9,1.9) -- (2.1,2.7) -- cycle; 
	\end{scope}
	
	\filldraw
	(1,0) circle [radius = 1pt]	
	(3,0) circle [radius = 1pt]	
	(1,2) circle [radius = 1pt]	
	(3,2) circle [radius = 1pt]	
	(0,1) circle [radius = 1pt]	
	(2,1) circle [radius = 1pt]	
	(0,3) circle [radius = 1pt]	
	(2,3) circle [radius = 1pt];
	
	\draw[->] (1.2,0) -- (2.8,0);
	\draw[->] (1.2,2) -- (2.8,2);
	\draw[->] (0.2,1) -- (1.8,1);
	\draw[->] (0.2,3) -- (1.8,3);
	
	\draw[->] (0,2.8) -- (0,1.2);
	\draw[->] (2,2.8) -- (2,1.2);
	\draw[->] (1,1.8) -- (1,0.2);
	\draw[->] (3,1.8) -- (3,0.2);
	
	\draw[->] (0.1,0.9) -- (0.9,0.1);
	\draw[->] (2.1,0.9) -- (2.9,0.1);
	\draw[->] (0.1,2.9) -- (0.9,2.1) node [midway, below] {$\sim$};
	\draw[->] (2.1,2.9) -- (2.9,2.1) node [midway, right] {$\sim$};
	
	\node at (1.5,-1) {$\overline{\Box}^3_{3,1}$};
\end{tikzpicture}
\]

\begin{table}
	\begin{center}
		\begin{tikzpicture}
			\matrix (magic) [matrix of nodes,nodes={minimum width=3cm,minimum height=1cm,draw,very thin},draw,inner sep=0]
			{
				& marked $3$-cubes & marked $2$-cubes  & marked $1$-cubes \\
				$\overline{\Box}^2_{1,1}$ & - & $\id$ & $\bd_{2,0},~\bd_{2,1}$ \\
				$\overline{\Box}^2_{2,1}$ & - & $\id$ & $\bd_{1,0}$ \\
				$\overline{\Box}^3_{1,1}$ & $\id$ & $\begin{gathered}\bd_{2,0},~\bd_{2,1},\\\bd_{3,0},~\bd_{3,1}\end{gathered}$ & $\begin{gathered}\bd_{3,0}\bd_{2,0},~\bd_{3,0}\bd_{2,1},\\\bd_{3,1}\bd_{2,0},~\bd_{3,1}\bd_{2,1}\end{gathered}$\\
				$\overline{\Box}^3_{2,1}$ & $\id$ & $\bd_{1,0},~\bd_{3,0},~\bd_{3,1}$ & $\bd_{3,0}\bd_{1,0},~\bd_{3,1}\bd_{1,0}$ \\
				$\overline{\Box}^3_{3,1}$ & $\id$ & $\bd_{1,0},~\bd_{1,1},~\bd_{2,0}$ &  $\bd_{2,0}\bd_{1,0},~\bd_{2,0}\bd_{1,1}$ \\
			};
		\end{tikzpicture}
		\caption{Marked cubes in $\overline{\Box}^n_{i,1}$}
		\label{strongly-comical-table}
	\end{center}
\end{table}

For $i + 1 \leq j \leq n + 1$, we let $\Gamma^n_{i,j}$ denote the regular subcomplex of $\stbox^n_{i,1}$ consisting of all negative faces, as well as the positive faces $(k,1)$ for which $1 \leq k \leq i-1$ or $j \leq k \leq n$.
\end{Def}

Note that in the case $j = n + 1$,  the only positive faces contained in $\Gamma^n_{i,j}$ are $(1,1)$ through $(i-1,1)$.

\begin{lem}\label{strong-comical-degens}
For $n \geq 1$, $1 \leq i \leq n$, let $x$ be an $n$-cube in a cubical set $X$.

\begin{enumerate}
\item\label{strong-comical-sigma} If $x$ is strongly $(i,1)$-comical, then so is $x \sigma_j$ for $j \geq i$.
\item\label{strong-comical-gamma-low} If $x$ is strongly $(i,1)$-comical, then $x \gamma_{j,1}$ is strongly $(i+1,1)$-comical for $j \leq i - 1$.
\item\label{strong-comical-gamma-high} If $x$ is strongly $(i,1)$-comical, then so is $x \gamma_{j, \varepsilon}$ for $j \geq i, \varepsilon \in \{0,1\}$.
\item\label{strong-comical-gamma-i} The cube $x \gamma_{i,1}$ is strongly $(i+1,1)$-comical.  
\end{enumerate}
\end{lem}

\begin{proof}
We prove \cref{strong-comical-sigma,strong-comical-gamma-i}; the proofs of \cref{strong-comical-gamma-low,strong-comical-gamma-high} are similar to these.

For \cref{strong-comical-sigma}, consider a $x \sigma_i \delta$, with $\delta$ written in standard form as follows:

$$
x\sigma_i \bd_{a_1,\varepsilon_1} \ldots \bd_{a_p,\varepsilon_p} \bd_{b_1,\varepsilon'_1}\ldots\bd_{b_q,\varepsilon'_q}
$$
where $a_p \geq i + 1$, $b_1 \leq i - 1$, and if $b_1 = i - 1$ then $\varepsilon_1 = 0$. First, suppose that $a_k = j$ for some $1 \leq k \leq p$; then  we can rewrite this expression into standard form as:

$$
x \bd_{a_1-1,\varepsilon_1} \ldots \bd_{a_{k-1}-1,\varepsilon_{k-1}}\bd_{a_{k+1},\varepsilon_{k+1}}\ldots \bd_{a_p,\varepsilon_p} \bd_{b_1,\varepsilon'_1}\ldots\bd_{b_q,\varepsilon'_q}
$$

By assumption, $a_{l} > a_{k} \geq a_{p} \geq i + 1$ for all $l < k$, so the indices $a_{l} - 1$ are still greater than or equal to $i + 1$, while all other maps in this standard form are unchanged. Thus this face of $x$ is marked by the assumption that $x$ is strongly $(i,1)$-comical.

On the other hand, suppose that no $a_{k}$ is equal to $j$; it must also be true that no $b_k$ is equal to $j$ as $b_k \leq i - 1 < j$ for all $k$. Then let $l$ be maximal such that $a_l > j$; we can rewrite this expression into standard form as:

$$
x\bd_{a_1,\varepsilon_1} \ldots \bd_{a_{l}-1,\varepsilon_{l}}\bd_{a_{l+1},\varepsilon_{l+1}}\ldots \bd_{a_p,\varepsilon_p} \bd_{b_1,\varepsilon'_1}\ldots\bd_{b_q,\varepsilon'_q}\sigma_{i-p+l-q} 
$$

This is degenerate, hence marked.

For \cref{strong-comical-gamma-i}, consider a face of $x \gamma_{i,1}$ written in standard form as

$$
x\gamma_{i,1} \bd_{a_1,\varepsilon_1} \ldots \bd_{a_p,\varepsilon_p} \bd_{b_1,\varepsilon'_1}\ldots\bd_{b_q,\varepsilon'_q}
$$
where now $a_p \geq i + 2$, $b_1 \leq i$, and if $b_1 = i$ then $\varepsilon'_1 = 0$. We can rewrite this expression as:

$$
x \bd_{a_1-1,\varepsilon_1} \ldots \bd_{a_p-1,\varepsilon_p} \gamma_{i,1} \bd_{b_1,\varepsilon'_1}\ldots\bd_{b_q,\varepsilon'_q}
$$

We consider two possible cases based on the value of $b_1$. If $b_1 = i$ then $\varepsilon'_1 = 0$, and we can rewrite the expression as:

$$
x \bd_{a_1-1,\varepsilon_1} \ldots \bd_{a_p-1,\varepsilon_p}  \bd_{b_1,\varepsilon'_1} \sigma_{i}\bd_{b_2,\varepsilon'_2}\ldots\bd_{b_q,\varepsilon'_q} =  x \bd_{a_1-1,\varepsilon_1} \ldots \bd_{a_p-1,\varepsilon_p}  \bd_{b_1,\varepsilon'_1} \bd_{b_2,\varepsilon'_2}\ldots\bd_{b_q,\varepsilon'_q}\sigma_{i-q+1}
$$

On the other hand, if $b_1 < i$, then the expression becomes:

$$
x \bd_{a_1-1,\varepsilon_1} \ldots \bd_{a_p-1,\varepsilon_p}  \bd_{b_1,\varepsilon'_1} \ldots\bd_{b_q,\varepsilon'_q}\gamma_{i-q} 
$$

Either way this cube is degenerate, hence marked.
\end{proof}

\begin{lem}\label{strong-comical-iso}
For $n, i$ as above and $i < k \leq n$, the $(k,1)$-face of $\stbox^n_{i,1}$ is isomorphic to $\stbox^{n-1}_{i,1}$.
\end{lem}

\begin{proof}
It is clear that the underlying cubical set of $\bd_{k,1}$ is $\Box^{n-1}$, so it remains to be verified that the marked faces of $\bd_{k,1}$ are precisely those which are marked in $\stbox^{n-1}_{i,1}$.

To see this, consider a face $\bd_{k,1} \delta$; write the standard form of $\delta$ as $\bd_{a_1,\varepsilon_1}\ldots\bd_{a_p,\varepsilon_p}\ldots\bd_{a_{q},\varepsilon_{q}}$, where $p$ is maximal such that $a_{p} \geq k$. Then we can rearrange $\bd_{k} \delta$ into standard form as:

$$
\bd_{a_1 + 1, \varepsilon_1}\ldots\bd_{a_p + 1,\varepsilon_p}\bd_{k,1}\bd_{a_{p+1},\varepsilon_{p+1}}\ldots\bd_{a_q,\varepsilon_q}
$$

This cube is marked if and only if this standard form does not contain any of the maps $\bd_{i-1,1}, \bd_{i,0},$ or $\bd_{i,1}$. As $k > i$ by assumption, this holds if and only if none of these maps appear in the standard form of $\delta$.
\end{proof}

\begin{lem}\label{strong-comical-anodyne}
For $n, i, j$ as above, the inclusion $\Gamma^n_{i,j} \hookrightarrow \stbox^n_{i,1}$ is anodyne.
\end{lem}

\begin{proof}
We proceed by induction on $n$. For $n = 1$, the only case to consider is the inclusion $\Gamma^1_{1,2} \hookrightarrow \stbox^n_{1,1}$, but this is isomorphic to the $(1,1)$-comical open box filling in dimension 1.

Now consider $n \geq 2$, and suppose the statement holds for $n - 1$. We first show that for $i + 2 \leq j \leq n + 1$, the inclusion $\Gamma^n_{i,j} \hookrightarrow \Gamma^n_{i,j-1}$ is anodyne. To see this, we consider the intersection of $\bd_{j-1,1}$ with $\Gamma^n_{i,j}$, i.e. the intersections of $\bd_{j-1,1}$ with each of the faces contained in $\Gamma^n_{i,j}$. 

\begin{itemize}
\item For $1 \leq k \leq j - 2$, the intersection of $\bd_{j-1,1}$ with $\bd_{k,0}$ is $\bd_{j-1,1}\bd_{k,0}$.
\item The intersection of $\bd_{j-1,1}$ with $\bd_{j-1,0}$ is empty.
\item For $j \leq k \leq  n$, the intersection of $\bd_{j-1,1}$ with $\bd_{k,0}$ is $\bd_{k,0}\bd_{j-1,1} = \bd_{j-1,1}\bd_{k-1,0}$.
\item For $1 \leq k \leq i - 1$, the intersection of $\bd_{j-1,1}$ with $\bd_{k,1}$ is $\bd_{j-1,1}\bd_{k,1}$.
\item For $j \leq k \leq n$, the intersection of $\bd_{j-1,1}$ with $\bd_{k,1}$ is $\bd_{k,1}\bd_{j-1,1} = \bd_{j-1,1}\bd_{k-1,1}$.
\end{itemize}

Thus we see that the intersection of $\bd_{j-1,1}$ with $\Gamma^n_{i,j}$ consists of all negative faces of $\bd_{j-1,1}$, together with its positive faces $\bd_{k,1}$ for $1 \leq k \leq i - 1$ and $j-1 \leq k \leq n - 1$. By \cref{strong-comical-iso}, this implies that $\bd_{j-1,1} \cap \Gamma^n_{i,j}$ is isomorphic to $\Gamma^{n-1}_{i,j-1}$, and that the inclusion $\Gamma^n_{i,j} \hookrightarrow \Gamma^n_{i,j-1}$ is the pushout of $\Gamma^{n-1}_{i,j-1} \hookrightarrow \stbox^{n-1}_{i,1}$ along the inclusion $\Gamma^{n-1}_{i,j-1} \cong \bd_{j-1,1} \cup \Gamma^n_{i,j} \hookrightarrow \Gamma^n_{i,j}$. Thus this inclusion is anodyne by the induction hypothesis.

To prove the desired statement for $n$, it thus suffices to show that $\Gamma^{n}_{i,i+1} \hookrightarrow \stbox^n_{i,1}$ is anodyne. But since every marked cube of $\Box^n_{i,1}$ is marked in $\stbox^n_{i,1}$, this map is a pushout of the $(i,1)$-comical open box inclusion.
\end{proof}

\begin{Def}
For $m \geq 1, n \geq 0$, $B^{m,n}$ is the subcomplex of $C^{m,n}$ consisting of the images of the faces $\partial_{1,1}$ through $\partial_{n+1,1}$, as well as all all faces $\partial_{i,0}$, under the quotient map $\Box^{m+n} \to C^{m,n}$.

The \emph{strongly comical $(m,n)$-cone} $\stc^{m,n}$ is given by the following pushout:
	\[
	\begin{tikzcd}
		\Box^{m+n}
		\arrow [r]
		\arrow [d]
		\pushout
          &
		 C^{m,n}
		\arrow [d] \\
		\stbox^{m+n}_{n+1,1}
		\arrow [r] &
		\stc^{m,n}
	\end{tikzcd}
	\]
In other words, $\stc^{m,n}$ is obtained from $C^{m,n}$ by marking the images under the quotient map $\Box^{m+n} \to C^{m,n}$ of the marked cubes of $\stbox^{m+n+1}_{n+1,1}$. 

An $(m,n)$-cone $x \colon C^{m,n} \to X$ in a marked cubical set $X$ is \emph{strongly comical} if it factors through $C^{m,n} \to \stc^{m,n}$. Note that this is equivalent to the assertion that the composite $\Box^{m+n} \to C^{m+n} \xrightarrow{x} X$ factors through $\overline{\Box}^{m+n}_{n+1}$. Thus a strongly comical $(m,n)$-cone in a cubical set $X$ is an $(m,n)$-cone in $X$ which is strongly $(n+1,1)$-comical when considered as an $(m+n)$-cube of $X$.

We define $\stb^{m,n}$ to be the regular subcomplex of $\stc^{m,n}$ with underlying cubical set $B^{m,n}$.
\end{Def}

\begin{lem}\label{b-c-anodyne}
For all $m, n \geq 0$, the inclusion $\stb^{m,n} \to \stc^{m,n}$ is anodyne.
\end{lem}

\begin{proof}
We will show that the following square is a pushout:
	\[
	\begin{tikzcd}
		\Gamma^{m+n}_{n+2}
		\arrow [r]
		\arrow [d]
          &
		 \stb^{m,n} 
		\arrow [d] \\
		\stbox^{m+n}_{m+n+1,0}
		\arrow [r] &
		 \stc^{m,n} 
	\end{tikzcd}
	\]
By \cref{ConeDesc}, if two cubes of $\Box^{m+n}$ are identified in $C^{m,n}$ then they are both contained in some face $\bd_{k,0}$, and hence in $B^{m,n}$. Thus the underlying diagram of cubical sets is a pushout, as the cubes of $C^{m,n}$ not present in $B^{m,n}$ are subject to no identifications.
It thus follows that the square itself is a pushout, as the non-degenerate marked cubes of $\stc^{m,n}$ are precisely the images under the quotient map $\Box^{m+n} \to C^{m,n}$ of those which are marked in $\stbox^{m+n}$.
\end{proof}

\begin{lem}[{\cite[Lem.~5.21]{doherty-kapulkin-lindsey-sattler}}]\label{BMap}
For a marked cubical set $X$, a map $x \colon B^{m,n} \to X$ is determined by a set of $(m,n-1)$-cones $x_{i,1} \colon C^{m,n-1} \to X$ for $1 \leq i \leq n$ and a set of $(m-1,n)$-cones $x_{i,0}$ for $n+1 \leq i \leq m+n$ such that for all $i_{1} < i_{2}, \varepsilon_{1}, \varepsilon_{2} \in \{0,1\}, x_{i_{2},\varepsilon_{2}}\bd_{i_{1},\varepsilon_{1}} = x_{i_{1},\varepsilon_{1}}\bd_{i_{2}-1,\varepsilon_{2}}$, with $x_{i,\varepsilon}$ being the image of $\partial_{i,\varepsilon}$ under $x$. \qed 
\end{lem}

We can adapt the notion of coherent families of composites from \cite[Def.~5.23]{doherty-kapulkin-lindsey-sattler} to the marked setting.
As in the case of $(\infty,1)$-categories, we show that every comical set admits such a family.

Recall that a coherent family of composites in a cubical quasicategory (i.e. a fibrant object in the cubical Joyal model structure on $\cSet$) represented a choice of a composite edge for each $n$-cube in the cubical quasicategory.
Although the definition of a coherent family of composites in a comical set is similar to its analogue for cubical quasicategories, the intuition differs somewhat.
Intuitively, choosing a coherent family of composites in a comical set $X$ amounts to replacing any $(m,n)$-cone in $X$ by an $(m-1,n+1)$-cone which is homotopic to it.
More precisely, thinking of filling comical open boxes as composition as in \cite[Sec.~3]{campion-kapulkin-maehara}, for an $(m,n)$-cone $x$ in a comical set $X$, we define $\theta^{m,n}(x)$ witnessing the composition of $x$ with a set of equivalences to obtain the $(m-1, n+1)$-cone $\theta^{m,n}(x) \partial_{n+2, 1}$.
Conversely,  \cref{theta-n+1-comical,theta-markings} show that we can also view $\theta^{m,n}(x)$ as composing $\theta^{m,n}(x) \partial_{n+2, 1}$ with equivalences to obtain $x$.

\begin{Def}\label{theta-construction}
A \emph{coherent family of composites} $\theta$ in a comical set $X$ is a family of functions $\theta^{m,n} \colon \cSet^+(C^{m,n},X) \to \cSet^+(\stc^{m,n+1},X)$ for $m \geq 1, n \geq 0$, satisfying the following identities: 

\begin{enumerate}
\myitem[$(\Theta 1)$]\label{ThetaFace1} for $i \leq n$, $\theta^{m,n}(x)\bd_{i,1} = \theta^{m,n-1}(x\bd_{i,1})$;
\myitem[$(\Theta 2)$]\label{ThetaFace1Id} $\theta^{m,n}(x)\bd_{n+1,1} = x$;
\myitem[$(\Theta 3)$]\label{ThetaFace0} for $m \geq 2$ and $i \geq n + 2$, $\theta^{m,n}(x)\bd_{i,0} = \theta^{m-1,n}(x\bd_{i-1,0})$;
\myitem[$(\Theta 4)$]\label{ThetaDegen} for $i \geq n + 2$, $\theta^{m,n}(x)\sigma_{i} = \theta^{m+1,n}(x\sigma_{i-1})$;
\myitem[$(\Theta 5)$]\label{ThetaLowCon} for $i \leq n$, $ \theta^{m,n}(x)\gamma_{i,1} = \theta^{m,n+1}(x\gamma_{i,1})$;
\myitem[$(\Theta 6)$]\label{ThetaHighCon} for $i \geq n + 2$, then $\theta^{m,n}(x)\gamma_{i,\varepsilon} = \theta^{m+1,n}(x\gamma_{i-1,\varepsilon})$;
\myitem[$(\Theta 7)$]\label{ThetaTheta} $\theta^{m,n+1}(\theta^{m,n}(x)) = \theta^{m,n}(x)\gamma_{n+1,1}$
\myitem[$(\Theta 8)$]\label{ThetaConeWLOG} for $x \colon C^{m-1,n+1} \to X$, $\theta^{m,n}(x) = x\gamma_{n+1,1}$.
\end{enumerate}
\end{Def}

Note that in \cite{doherty-kapulkin-lindsey-sattler}, $\theta^{0,n}$ is defined to be the degeneracy operator $\sigma_{n+1}$. This would not be appropriate in our setting, as $\stc^{m,n}$ is only defined for $m \geq 1$. However, as the definition of $\theta^{0,n}$ is merely a notational convenience, the combinatorial proofs of \cite{doherty-kapulkin-lindsey-sattler} remain valid here.

Our next goal is to prove the following theorem.

\begin{thm}\label{theta-exists}
Every comical set admits a coherent family of composites.
\end{thm}

The following lemma will be used in defining $\theta^{m,n}$ in the inductive case.

\begin{lem}\label{fix-face-theta}
Let $m \geq 2, n \geq 0$, and let $X$ be a comical set equipped with functions $\theta^{m,n}$ satisfying the identities of \cref{theta-construction} for all pairs $(m',n')$ such that $m' \leq m$, $n' \leq n$, and at least one of these two inequalities is strict. Then for any $x \colon C^{m,n} \to X$, the faces of $\theta^{m,n}(x)$ fixed by the identities $(\Theta 1)$ through $(\Theta 3)$ define a map $\stb^{m,n+1} \to X$.
\end{lem}

\begin{proof}
That these faces define a map $B^{m,n+1} \to X$ follows from \cref{BMap}; see the proof of \cite[Lem.~5.28]{doherty-kapulkin-lindsey-sattler} for details. Thus it remains to show that this map factors through $\stb^{m,n+1}$. In other words, we must show that any face of the form $\theta^{m,n}(x)\delta$ is marked if the standard form of $\delta$ is non-empty and contains only maps of the form $\bd_{k,\varepsilon}$ where either $\varepsilon = 0$ and $k \neq n + 2$ or $\varepsilon = 1$ and $k \leq n $.

First suppose that the standard form of $\delta$ contains some map of the form $\bd_{k,0}$ for $k \leq n + 1$. Then, since it contains no face of the form $\bd_{n+2,\varepsilon}$ by assumption, this is a degenerate face of $B^{m,n+1}$ by \cref{cone-desc-faces}, and is therefore marked.

Now consider a face written in standard form as:

$$
\theta^{m,n}(x)\bd_{a_1,0} \ldots \bd_{a_p,0} \bd_{b_{1},1} \ldots \bd_{b_q,1}
$$
where $a_p \geq n + 3$ and $b_1 \leq n$. Using the identities $(\Theta 1)$ and $(\Theta 3)$, we can rewrite this as :

$$
\theta^{m-p,n-q}(x\bd_{a_1-1,0} \ldots \bd_{a_p-1,0} \bd_{b_{1},1} \ldots \bd_{b_q,1})
$$
which is marked by assumption. (Note that $m - p \geq 1$ since there can be at most $m - 1$ face maps with index greater than or equal to $n + 3$.)
\end{proof}

\begin{proof}[Proof of \cref{theta-exists}]
We construct the functions $\theta^{m,n}$ by induction, following \cite[Defs.~5.25 \& 5.29]{doherty-kapulkin-lindsey-sattler}. For the base case $m = 1$, we set $\theta^{1,n}(x) = x \gamma_{n+1,1}$. (In fact, since every $(1,n)$-cone is a $(0,n+1)$-cone by \cref{Qcone}, this definition is required by identity $(\Theta 8)$.) 

Now suppose we have defined $\theta^{m',n'}$ for all such pairs with $m' \leq m, n' \leq n$, and at least one of these inequalities strict.
We will define $\theta^{m,n}$ by case analysis on $x \colon C^{m,n} \to X$, with cases depending on the ``standard form of $x$'' in the sense of the convention explained after \cref{cube-standard-form}.
For $x \colon C^{m,n} \to X$, we define $\theta^{m,n}(x)$ as follows:
\begin{enumerate}
\item If the standard form of $x$ is $z\sigma_{a_{p}}$ for some $a_{p} \geq n + 1$, then $\theta^{m,n}(x) = \theta^{m-1,n}(z)\sigma_{a_{p}+1}$;
\item If the standard form of $x$ is $z\gamma_{b_{q},1}$ for some $b_{q} \leq n - 1$, then $\theta^{m,n}(x) = \theta^{m,n-1}(z)\gamma_{b_{q},1}$;
\item If the standard form of $x$ is $z\gamma_{b_{q},\varepsilon}$ for some $b_{q} \geq n + 1$, then $\theta^{m,n}(x) = \theta^{m-1,n}(z)\gamma_{b_{q}+1,\varepsilon}$;
\item If $x$ is an $(m-1,n+1)$-cone not covered under any of cases (1) through (3), then $\theta^{m,n}(x) = x\gamma_{n+1,1}$;
\item If $x = \theta^{m,n-1}(x')$ for some $x' \colon C^{m,n-1} \to X$ and $x$ is not covered under any of cases (1) through (4) then $\theta^{m,n}(x) = x\gamma_{n,1}$;
\item If $x$ is not covered under any of cases (1) through (5), we construct $\theta^{m,n}(x)$ by extending the map $\stb^{m,n+1} \to X$ of \cref{fix-face-theta} to $\stc^{m,n+1}$, by taking a lift in the diagram
	\[
	\begin{tikzcd}
		\stb^{m,n+1}
		\arrow [r]
		\arrow [d]
          &
		 X
		\arrow [d] \\
		\stc^{m,n+1}
		\arrow [r] &
		\Box^0
	\end{tikzcd}
	\]
That the desired lift exists follows from \cref{b-c-anodyne}.
\end{enumerate}

That this definition satisfies identities $(\Theta 1)$ through $(\Theta 8)$ is proven in \cite[Props.~A1--A5]{doherty-kapulkin-lindsey-sattler}. Thus it remains only to be shown that $\theta^{m,n}(x)$ is a strongly comical $(m,n+1)$-cone (i.e., factors through $\overline{C}^{m,n+1}$) for all $x \colon C^{m,n} \to X$. For cases (1) through (5) this follows from \cref{strong-comical-degens}, while for case (6) it holds by construction.
\end{proof}

\begin{lem}\label{theta-n+1-comical}
Let $X$ be a comical set equipped with a coherent family of composites $\theta$. For $m \geq 1, n \geq 0$ and $x \colon C^{m,n} \to X$, the $(m+n+1)$-cube $\theta^{m,n}(x)$ is $(n+1,1)$-comical.
\end{lem}

\begin{proof}
We will show that all faces of $\theta^{m,n}(x)$ corresponding to marked faces of $\Box^{m+n+1}_{n+1,1}$ are marked. Recall that by construction, $\theta^{m,n}(x)$ factors through $\stc^{m,n+1}$, and is therefore strongly $(n+2,1)$-comical, meaning that the standard form of any unmarked face must contain $\partial_{n+1,1},\partial_{n+2,0}$, or $\partial_{n+2,1}$. Faces whose standard forms contain $\bd_{n+1,1}$ or $\bd_{n+2,1}$ are permitted to be unmarked by the definition of the $(n+1,1)$-comical $(m+n+1)$-cube, thus we may restrict our attention to those faces whose standard forms contain $\partial_{n+2,0}$.

Let $\delta$ denote a face of $\Box^{m+n+1}$ whose standard form contains $\bd_{n+2,0}$, and which is marked in $\Box^{m+n+1}_{n+1,1}$, and consider the face $\theta^{m,n}(x)\delta$. The assumption that $\delta$ is marked in $\Box^{m+n+1}_{n+1,1}$ implies that its standard form contains no map of the form $\bd_{n+1,\varepsilon}$. Therefore, if the standard form of $\delta$ contains any map of the form $\bd_{k,0}$ for $k \leq n$, then $\theta^{m,n}(x)\delta$ is degenerate by \cref{cone-desc-faces}, and hence marked. So assume that the standard form of $\delta$ contains no such map; then we may write this face in standard form as:

$$
z = \theta^{m,n}(x)\partial_{a_{1},\varepsilon_{1}}...\partial_{a_{p},\varepsilon_p}\partial_{n+2,0}\partial_{b_{1},1}...\partial_{b_{q},1}
$$
where $a_p \geq n + 3$. Moreover, as the standard form of $\delta$ contains no map of the form $\bd_{n+1,\varepsilon}$, we have $b_1 \leq n$.

First, suppose that $\varepsilon_i = 0$ for all $1 \leq i \leq p$; then we can rewrite this expression using the identities $(\Theta 1)$ and $(\Theta 3)$ to obtain:

$$
z = \theta^{m-p-1,n-q}(x\partial_{a_{1}-1,0}...\partial_{a_{p}-1,0}\partial_{n+1,0}\partial_{b_{1},1}...\partial_{b_{q},1})
$$
hence this face is marked.

Now suppose that $\varepsilon_i = 1$ for some $1 \leq i \leq p$, and consider the maximal value $i$ such that this condition is satisfied. Then the standard form above contains the string $\partial_{a_i,1}\partial_{a_{i+1},0}...\partial_{a_{p},0}\partial_{n+2,0}$. Our assumption that $\delta$ is marked in $\Box^{m+n+1}_{n+1,1}$ implies that there is a ``gap'' in this string, i.e. that there is some value between $n+2$ and $a_i$ which does not appear as some $a_{k}$. In particular, this implies that there are fewer than $a_i - (n + 2)$ maps in the string $\partial_{a_{i+1},0}...\partial_{a_{p},0}\partial_{n+2,0}$, i.e. that $p - i + 1 < a_i - (n+2)$. 

We apply cubical identities to the given standard form to move this string to the front, and then repeatedly apply identity $(\Theta 1)$, as follows.
 
\begin{align*}
z  & = \theta^{m,n}(x)\partial_{a_1,\varepsilon_1}\ldots\partial_{a_i,1}\partial_{a_{i+1},0}\ldots\partial_{a_{p},0}\partial_{n+2,0}\partial_{b_1,1}\ldots\partial_{b_q,1} \\
& = \theta^{m,n}(x)\partial_{a_{i+1},0}\ldots\partial_{a_{p},0}\partial_{n+2,0}\partial_{a_1 - (p - i + 1),\varepsilon_1}\ldots\partial_{a_{i} - (p - i + 1),1}\partial_{b_1,1}\ldots\partial_{b_q,1} \\
& = \theta^{m - (p - a + 1),n}(x\partial_{a_{i+1}-1,0}\ldots\partial_{a_{p}-1,0}\partial_{n+1,0})\partial_{a_1 - (p - i + 1),\varepsilon_1}\ldots\partial_{a_{i} - (p - i + 1),1}\partial_{b_1,1}\ldots\partial_{b_q,1} \\
\end{align*}

The expression $\partial_{a_1 - (p - i + 1),\varepsilon_1}\ldots\partial_{a_{i} - (p - i + 1),1}\partial_{b_1,1}\ldots\partial_{b_q,1}$ is again in standard form, and by the inequality above, we can see that $a_i - (p - i + 1) > n + 2$. Thus the standard form of the face map being applied to $\theta^{m - (p - i + 1),n}(x\partial_{a_{i+1}-1,0}...\partial_{a_{p}-1,0}\partial_{n+1,0})$ does not contain any of the maps $\partial_{n+2,0},\partial_{n+2,1},$ or $\partial_{n+1,1}$. Since $\theta^{m - (p - i + 1),n}(x\partial_{a_{i+1}-1,0}...\partial_{a_{p}-1,0}\partial_{n+1,0})$ is a strongly $(n+2,1)$-comical cube by construction, this implies that $z$ is marked.
\end{proof}

\begin{lem}\label{theta-markings}
Let $X$ be a comical set equipped with a coherent family of composites $\theta$. For $m \geq 1, n \geq 0$, and let $x \colon C^{m,n} \to X$, all faces of $\theta^{m,n}(x)$ other than $x$ itself and $\theta^{m,n}(x)\bd_{n+2,1}$ are marked. Moreover, $x$ is marked if and only if $\theta^{m,n}(x)\bd_{n+2,1}$ is marked.
\end{lem}

\begin{proof}
For $m = 1$ this is trivial, as $\theta^{1,n}(x)\bd_{n+2,1} = x\gamma_{n+1,1}\bd_{n+2,1} = x$, while the other two faces are degenerate. 

Now consider $m \geq 2$. We begin by showing that all $(m+n)$-dimensional faces of $\theta^{m,n}(x)$, other than $\theta^{m,n}(x)\bd_{n+1,1} = x$ and $\theta^{m,n}(x)\bd_{n+2,1}$, are marked. To see this, observe that:

\begin{itemize}
\item for $i \leq n+1$, $\theta^{m,n}(x)\bd_{i,0}$ is degenerate by \cref{cone-desc-faces}, hence marked;
\item for $i \geq n+2$, $\theta^{m,n}(x)\bd_{i,0}$ is marked by $(\Theta 3)$;
\item for $i \leq n$, $\theta^{m,n}(x) \bd_{i,1}$ is marked by $(\Theta 1)$;
\item for $i \geq n + 3$, $\theta^{m,n}(x) \bd_{i,1}$ is marked because $\theta^{m,n}(x)$ factors through $\overline{C}^{m,n+1}$ by definition, and is therefore strongly comical.
\end{itemize}

Now suppose that $x$ is marked. 
Then $\theta^{m,n}(x)$ is an $(n+2,1)$-comical cube whose faces of codimension one (i.e., the $(m+n)$-dimensional ones) other than $\partial_{n+2, 1}$ are marked, this implies that $\theta^{m,n}(x) \colon \Box^{m+n+1} \to X$ factors through $(\Box^{m+n+1}_{n+2,1})'$.
Moreover, since $X$ is a comical set, we obtain a lift in
	\[
	\begin{tikzcd}
		(\Box^{m+n+1}_{n+2,1})'
		\arrow [r, "\theta^{m,n}(x)"]
		\arrow [d]
          &
		 X
		\arrow [d] \\
		\tau_{m+n-1} \Box^{m+n+1}_{n+2,1}
		\arrow [r] &
		\Box^0
	\end{tikzcd}
	\]
Thus we see that all $(m+n)$-dimensional faces of $\theta^{m,n}(x)$ are marked, including $\theta^{m,n}(x)\bd_{n+2,1}$. 

In the case where $\theta^{m,n}(x)\bd_{n+2,1}$ is marked, we can show that $x$ is marked via a similar proof using \cref{theta-n+1-comical}.
\end{proof}

\begin{prop}\label{counit-anodyne}
For any comical set $X$, the counit $Q \int X \hookrightarrow X$ is anodyne.
\end{prop}

Before proving this proposition, we recall one final result characterizing those $(m,n)$-cones for which $\theta^{m,n}$ is non-degenerate.

\begin{lem}[{\cite[Prop.~5.32]{doherty-kapulkin-lindsey-sattler}}]\label{theta-non-degen}
Let $X$ be a comical set $m \geq 1, n \geq 0$, and $x \colon C^{m,n} \to X$ an $(m,n)$-cone of $X$. Then $\theta^{m,n}(x)$ is non-degenerate if and only if $x$ is non-degenerate, not an $(m-1,n+1)$-cone, and not in the image of $\theta^{m,n-1}$. Moreover, in this case $\theta^{m,n}(x)$ is not an $(m-1,n+2)$-cone. \qed
\end{lem}

\begin{proof}[Proof of \cref{counit-anodyne}]
Our proof will follow the structure of the proof of \cite[Prop.~6.23]{doherty-kapulkin-lindsey-sattler}. By \cref{theta-exists}, we may equip $X$ with a coherent family of composites $\theta$. We will build $X$ from $Q \int X$ via successive comical open box fillings and comical marking extensions, thereby showing that the inclusion of $Q \int X$ into $X$ is anodyne. 

For $m \geq 2, n \geq -1$, let $X^{m,n}$ denote the smallest regular subcomplex of $X$ containing all $(m',n')$-cones of $X$, as well as all cones of the form $\theta^{m',n'}(x)$, for $m' < m$ or $m' = m, n' \leq n$. Note that $(1,n)$-cubes and $(0,n+1)$-cubes coincide by \cref{Qcone}, and all cubes in the image of $\theta^{1,n}$ are degenerate; thus $X^{2,-1}$ is the regular subcomplex of $X$ whose non-degenerate $n$-cubes are the $(0,n)$-cones of $X$. By \cref{Q-counit-mono}, this implies $X^{2,-1} = Q \int X$.

For $m < m'$ or $m = m', n \leq n'$, $X^{m,n}$ is a regular subcomplex of $X^{m',n'}$. Thus we obtain a sequence of regular inclusions:

$$
Q \textstyle{\int} X = X^{2,-1} \hookrightarrow X^{2,0} \hookrightarrow ... \hookrightarrow X^{3,-1} \hookrightarrow X^{3,0} \hookrightarrow ... \hookrightarrow X^{m,n} \hookrightarrow ...
$$

Observe that the colimit of this sequence is $X$. Furthermore, for each $m$, $X^{m,-1}$ is the union of all regular subcomplexes $X^{m',n}$ for $m' < m$, i.e. the colimit of the sequence of regular inclusions $Q \int X \hookrightarrow ... \hookrightarrow X^{m',n} \hookrightarrow ...$. So to show that $Q \int X \hookrightarrow X$ is anodyne, it suffices to show that each map $X^{m,n-1} \hookrightarrow X^{m,n}$ for $n \geq 0$ is anodyne.

Fix $m \geq 2, n \geq 0$, and let $S$ denote the set of non-degenerate $(m,n)$-cones of $X$ which are not $(m-1,n+1)$-cones, and are not in the image of $\theta^{m,n-1}$. Let $S^+$ denote the set of all marked cubes in $S$. To construct $X^{m,n}$ from $X^{m,n-1}$, we must adjoin to $X^{m,n-1}$ all $(m,n)$-cones of $X$, and images of such cones under $\theta^{m,n}$, which are not already present in $X^{m,n}$, and mark those which are marked in $X$. Using \cref{sa1,ConeFaceDeg,theta-non-degen}, and the identities $(\Theta 1)$ to $(\Theta 8)$, we can see that these are precisely the cones in $S$ and their images under $\theta^{m,n}$.

Let $x \in S$; we will analyze the faces of $\theta^{m,n}(x)$ to determine which of them are contained in $X^{m,n-1}$. For $i \leq n$ we have $\theta^{m,n}(x) \bd_{i,1} = \theta^{m,n-1}(x\bd_{i,1})$, while for $i \geq n + 2$ or $\varepsilon = 0$, $\theta^{m,n}(x) \bd_{i,\varepsilon}$ is an $(m-1,n+1)$-cone by \cref{ConeFaceDeg}. Thus we see that the only face of $\theta^{m,n}(x)$ which is not contained in $X^{m,n-1}$ is $\theta^{m,n}(x)\bd_{n+1,1} = x$. Furthermore, the faces of $\theta^{m,n}(x)$ which are contained in $X^{m,n-1}$ form an $(n+1,1)$-comical open box by \cref{theta-n+1-comical}. 

We now construct a marked cubical set $\widetilde{X}^{m,n}$ by adding all cubes in $S$ and their images under $\theta^{m,n}$ to $X^{m,n-1}$, with the cubes of $S$ unmarked but their images under $\theta^{m,n}$ marked; this amounts to filling all of these comical open boxes. In other words, we construct the following pushout diagram:
	\[
	\begin{tikzcd} [row sep = large]
		\displaystyle{\coprod_{S} \sqcap^{m+n+1}_{n+1,1} }
		\arrow [r]
		\arrow [d]
		\pushout
          &
		 X^{m,n-1} 
		\arrow [d] \\
		\displaystyle{\coprod_{S} \Box^{m+n+1}_{n+1,1} }
		\arrow [r] &
		\widetilde{X}^{m,n}
	\end{tikzcd}
	\]
%
The marked cubical set $\widetilde{X}^{m,n}$ has the same underlying cubical set as $X^{m,n}$ but is missing some markings, namely those of the cubes in $S^+$. The map $X^{m,n-1} \hookrightarrow \widetilde{X}^{m,n}$ is anodyne, as a pushout of a coproduct of anodyne maps.

To obtain $X^{m,n}$ from $\widetilde{X}^{m,n}$, we must mark all the cubes of $S^+$. Let $x \in S^+$; then \cref{theta-n+1-comical,theta-markings} imply that all faces of $\theta^{m,n}(x)$ other than $x$ are marked in $X$, and hence also in $X^{m,n-1}$. It follows that $\theta^{m,n}(x) \colon \Box^{m+n+1} \to \widetilde{X}^{m,n}$ factors through $(\Box^{m,n}_{n+1,1})'$. We thus have the following pushout diagram:
	\[
	\begin{tikzcd} [row sep = large]
		\displaystyle{\coprod_{S^+} (\Box^{m+n+1}_{n+1,1})' }
		\arrow [r]
		\arrow [d]
		\pushout
          &
		 \widetilde{X}^{m,n}
		\arrow [d] \\
		\displaystyle{\coprod_{S^+} \tau_{m+n-1}\Box^{m+n+1}_{n+1,1} }
		\arrow [r] &
		X^{m,n}
	\end{tikzcd}
	\]
Thus the map $\widetilde{X}^{m,n} \to X^{m,n}$ is anodyne, as a pushout of a coproduct of anodyne maps. The composite map $X^{m,n-1} \hookrightarrow \widetilde{X}^{m,n} \to X^{m,n}$ is therefore anodyne as well.
\end{proof}

We can now prove our main results.

\begin{thm}\label{Q-equivalence}
The adjunction $Q : \sSet^+ \rightleftarrows \cSet^+ : \int$ is a Quillen equivalence between the model structure on $\sSet^+$ for ($n$-trivial, saturated) complicial sets and the corresponding model structure on $\cSet^+$.
\end{thm}

\begin{proof}
It suffices to show that the left adjoint $Q$ creates the weak equivalences of $\sSet^+$, and that the counit $Q \int X \hookrightarrow X$ is a weak equivalence for all fibrant objects $X$. The former statement is \cref{Q-pres-refl-we}, while the latter is immediate from \cref{counit-anodyne}.
\end{proof}

\begin{thm}\label{T-equivalence}
The adjunction $T : \cSet^+ \rightleftarrows \sSet^+ : U$ is a Quillen equivalence between the model structure on $\cSet^+$ for ($n$-trivial, saturated) comical sets and the corresponding model structure on $\sSet^+$.
\end{thm}

\begin{proof}
By \cref{TQ}, we have a natural weak equivalence $TQ \Rightarrow \id_{\sSet^+}$. Thus the composite of the derived functors of $T$ and $Q$ is naturally isomorphic to the identity on the homotopy category of $\sSet^+$. Since the derived functor of $Q$ is an equivalence of categories by \cref{Q-equivalence}, the same is therefore true of the derived functor of $T$.
\end{proof}


\bibliographystyle{amsalphaurlmod}
\bibliography{all-refs}

\end{document}